\documentclass{amsart}
\usepackage{graphicx}
\usepackage{amscd}
\usepackage{amsmath}
\usepackage{amsfonts}
\usepackage{amssymb}
\usepackage{nicefrac}
\usepackage{setspace}
\usepackage{enumerate}         
\usepackage{fixme}
\usepackage{color}
\usepackage{url}
\usepackage{amsthm}
\usepackage{hyperref}
\usepackage{xy}
\usepackage{enumitem}

\theoremstyle{plain}
\newtheorem{theorem}{Theorem}[section]

\newtheorem{corollary}[theorem]{Corollary}

\newtheorem{lemma}[theorem]{Lemma}

\newtheorem{openproblem}[theorem]{Open problem}
\newtheorem{proposition}[theorem]{Proposition}

\theoremstyle{remark}
\newtheorem{remark}[theorem]{Remark}

\newtheorem{example}[theorem]{Example}

\numberwithin{equation}{section}

\newcommand{\ind}{1\!\kern-1pt \mathrm{I}}
\newcommand{\rsto}{]\!\kern-1.8pt ]}
\newcommand{\lsto}{[\!\kern-1.7pt [}

\numberwithin{equation}{section}

\newcommand{\C}{\mathbb{C}}
\newcommand{\E}{\mathbb{E}}

\newcommand{\N}{\mathbb{N}}
\renewcommand{\P}{\mathbb{P}}
\newcommand{\R}{\mathbb{R}}

\newcommand{\cA}{\mathcal{A}}
\newcommand{\cB}{\mathcal{B}}

\newcommand{\cF}{\mathcal{F}}

\newcommand{\cL}{\mathcal{L}}

\newcommand{\cQ}{\mathcal{Q}}

\newcommand{\Id}{\operatorname{Id}}

\DeclareMathOperator{\sign}{sign}

\newcommand{\id}{\operatorname{I}}

\newcommand{\tr}{\mathop{\mathsf{Tr}}}
\newcommand{\e}{\operatorname{e}}

\renewcommand{\Re}{\mathrm{Re}}
\renewcommand{\Im}{\mathrm{Im}}

\renewcommand{\rho}{\varrho}
\newcommand{\eps}{\varepsilon}
\renewcommand{\phi}{\varphi}

\makeatletter
\DeclareFontFamily{OMX}{MnSymbolE}{}
\DeclareSymbolFont{MnLargeSymbols}{OMX}{MnSymbolE}{m}{n}
\SetSymbolFont{MnLargeSymbols}{bold}{OMX}{MnSymbolE}{b}{n}
\DeclareFontShape{OMX}{MnSymbolE}{m}{n}{
    <-6>  MnSymbolE5
   <6-7>  MnSymbolE6
   <7-8>  MnSymbolE7
   <8-9>  MnSymbolE8
   <9-10> MnSymbolE9
  <10-12> MnSymbolE10
  <12->   MnSymbolE12
}{}
\DeclareFontShape{OMX}{MnSymbolE}{b}{n}{
    <-6>  MnSymbolE-Bold5
   <6-7>  MnSymbolE-Bold6
   <7-8>  MnSymbolE-Bold7
   <8-9>  MnSymbolE-Bold8
   <9-10> MnSymbolE-Bold9
  <10-12> MnSymbolE-Bold10
  <12->   MnSymbolE-Bold12
}{}

\let\llangle\@undefined
\let\rrangle\@undefined
\DeclareMathDelimiter{\llangle}{\mathopen}%
                     {MnLargeSymbols}{'164}{MnLargeSymbols}{'164}
\DeclareMathDelimiter{\rrangle}{\mathclose}%
                     {MnLargeSymbols}{'171}{MnLargeSymbols}{'171}
\makeatother

\newcommand{\sgc}{\begin{color}{blue}}
\newcommand{\cgs}{\end{color}}

\newcommand{\agc}{\begin{color}{green}}
\newcommand{\ags}{\end{color}}

\begin{document}
\title[]
{Infinite-dimensional Wishart processes}
\begin{abstract}
We introduce and analyse infinite dimensional Wishart processes taking values in the cone $S^+_1(H)$ of positive self-adjoint trace class operators on a separable real Hilbert space $H$. Our main result
gives necessary and sufficient conditions for their existence, 
showing that these processes are necessarily of fixed finite rank almost surely, 
 but they are not confined to a finite-dimensional subspace of $S^+_1(H)$. 
 By providing explicit solutions to operator valued Riccati equations, we prove that their Fourier-Laplace transform is exponentially affine in the initial value.
 As a corollary, we obtain uniqueness in law as well as the Markov property.
 We actually show the explicit form of the Fourier-Laplace transform for a wide parameter regime, thereby also extending what is known in the finite-dimensional setting.
 Finally, under minor conditions on the parameters we prove the Feller property with respect to a slight refinement of the weak-$*$-topology on $S_1^+(H)$. Applications of our results range from tractable infinite-dimensional covariance modelling to the analysis of the limit spectrum of large random matrices.
\end{abstract}
\thanks{
Sonja Cox was partially supported by the NWO grant VI.Vidi.213.070. In addition, she would like to thank Jan Brandts and Lenny Taelman for discussions on how to prove Lemma~\ref{lem:approx_ONS2}.
Christa Cuchiero gratefully acknowledges financial support through the grant Y 1235 of the FWF START-program.}

\author{Sonja Cox, Christa Cuchiero, Asma Kheder }
\maketitle

\section{Introduction}
The goal of this paper is to introduce and analyse infinite-dimensional Wishart processes. An infinite-dimensional Wishart processes is a stochastic process $X= (X_t)_{t\geq 0}$ taking values in $S^+_1(H)$, the cone of positive self-adjoint trace class operators on a separable real Hilbert space $H$, and satisfying (in some sense) the following  stochastic differential equation:
\begin{equation}\label{eq:Wishart-formal}
dX_t= (\alpha Q + X_t A + A^* X_t) \,dt +\sqrt{X_t} \,dW_t \sqrt{Q} + \sqrt{Q} \,dW_t^* \sqrt{X_t},\, t\geq 0, \, X_0=x_0.
\end{equation}
Here $\alpha \in \R$, $A\colon D(A)\subset H \rightarrow H$ is the generator of a $C_0$-semigroup, $x_0$ and $Q$ are a positive self-adjoint bounded operators, and $(W_t)_{t\geq 0}$ is an $L_2(H)$-cylindrical Brownian motion (where $L_2(H)$ is the space of Hilbert Schmidt operators on $H$). 

Finite-dimensional Wishart processes, i.e., processes taking values in $S^+(\mathbb{R}^n)$, the cone of positive semidefinite $n \times n$ matrices, have been studied thoroughly: in \cite{bru1989diffusions, Bru:1991} the existence of finite-dimensional Wishart processes was established under certain conditions on the parameters, and stochastic differential equations were derived for the eigenvalues and eigenvectors. It was soon recognised that these finite-dimensional Wishart processes are \emph{affine}, i.e., Markov processes whose Laplace transform depends in an exponentially affine way on the initial value. A full characterisation of all $S^+(\R^{n})$-valued affine processes was presented in \cite{CFMT:11}: these extend the classical Wishart processes by allowing state-dependent jumps and a  more general drift, while the diffusion structure is analogous to  \eqref{eq:Wishart-formal}.
In \cite{graczyk2018characterization}, Wishart processes 
with starting values on  lower rank submanifolds
of $S^+(\mathbb{R}^n)$ were characterised under non-degeneracy conditions on the diffusion matrix, showing that there is an interplay between the rank of the initial values and the constant drift part; see \cite[Theorem 1.3]{graczyk2018characterization} and also \cite{M:19,letac2018laplace,letac2008noncentral}.\par 

On the one hand the interest for Wishart processes is clearly motivated by such intriguing mathematical properties, on the other hand they are highly popular in applications due to their suitability as tractable stochastic covariance models. Indeed, one important application of finite-dimensional Wishart processes is  multivariate asset price modelling with stochastic covariances, as e.g. in~\cite{buraschi2010correlation,
da2011hedging,
da2008multifactor, 
FonescaEtAl:2007, 
gourieroux2010derivative,  
grasselli2008solvable,
HarveyEtAl:1994,
leippold2008asset, PhilipovGlickman:2006} and the references therein. Note however that models e.g.\ for bond and commodity markets under the Heath–Jarrow–Morton–Musiela (HJMM) paradigm call for \emph{infinite-dimensional stochastic covariance models} (see e.g.,~\cite{ benth2014representation,benth2021barndorff, Benth:2018,  carmona2007interest,filipovic2001consistency}). 
\par 
The most tractable class are again infinite-dimensional \emph{affine} stochastic covariance processes of which some instances have already been considered in the literature: indeed, a rank-1 $S_1^+(H)$-valued process $X$ is constructed in \cite{Benth:2018} by taking $X_t = Y_t\otimes Y_t$, $t\geq 0$, where $(Y_t)_{t\geq 0}$ is an $H$-valued Ornstein-Uhlenbeck process. This turns out to be a special case of an infinite-dimensional Wishart process, see Remark~\ref{rem:OU_alternative} below. In addition, affine pure-jump covariance processes with values in the cone of positive Hilbert-Schmidt operators were introduced in~\cite{cox2022affine} and applied to stochastic volatility modelling in~\cite{cox2022infinite}.

The reason the process introduced in~\cite{cox2022affine} is of pure-jump type is that it was designed to potentially be of infinite rank, whereas there are strong indications from the finite-dimensional theory that in the presence of a non-degenerated diffusion part, i.e.,~\emph{when $Q$ in~\eqref{eq:Wishart-formal} is of infinite rank, then an infinite-dimensional Wishart process is necessarily of finite rank almost everywhere}. To explain this statement, let us return for a moment to the finite-dimensional setting: a Wishart process $X$ taking values in $S^+(\R^n)$ is a process satisfying 
\begin{equation}
dX_t = (\alpha Q + AX_t + X_t A^*) \,dt + \sqrt{X_t}dW_t \sqrt{Q} + \sqrt{Q}\,dW_t^* \sqrt{X_t},\quad t\geq 0, 
\end{equation}
where $A\in \R^{n\times n}$, $Q\in S^{+}(\R^n)$, and $W$ is a standard $\R^{n\times n}$-valued Brownian motion. It is well-known (see, e.g.,~\cite{Bru:1991,CFMT:11,graczyk2018characterization, letac2018laplace,M:19}) that if $Q$ is injective, then such a finite-dimensional Wishart process exists if and only if either $\alpha \in [n-1,\infty)$, or $\alpha \in \{0,\ldots,n-2\}$ and $\operatorname{rank}(x_0)\leq \alpha$. In case of the latter one has $\operatorname{rank}(X_t)\leq \alpha$ a.s.\ for all $t\geq 0$. When translated to the infinite-dimensional setting this suggests that Wishart processes of infinite rank are hard to come by. Indeed, we prove the following 
(see 
Theorems~\ref{thm:existence} and~\ref{thm:char_infdimwishart}, Corollary~\ref{cor:rankn}, as well as Remarks~\ref{rem:integrability} and~\ref{rem:semigroup_injective} below):
\begin{theorem}\label{thm:intro}
If $Q$ is of trace class and injective and $A$ is bounded then an analytically and probabilistically weak solution to~\eqref{eq:Wishart-formal} exists if and only if $\alpha\in \N$ and $\operatorname{rank}(x_0)\leq \alpha$.
In this case, $\operatorname{rank}(X_t)= \alpha$ a.s.\ for almost all $t>0$.
\end{theorem}

In fact, our results go beyond the realm of Theorem~\ref{thm:intro}. In general, we only assume that $A\colon D(A)\subset H \rightarrow H$ is the generator of a $C_0$-semigroup $(\e^{tA})_{t\geq 0}$ and $Q$ a bounded positive self-adjoint operator on $H$ satisfying $\int_{0}^{t} \| \e^{sA} \sqrt{Q} \|_{L_2(H)} \,ds < \infty$ for all $t>0$. For this setting we have the following results:
\begin{enumerate}
\item \label{sufficient} If $\alpha \in \N$ and $\operatorname{rank}(x_0)\leq \alpha$, then there exists a probabilistically and analytically weak solution $X$ to~\eqref{eq:Wishart-formal}, 
see Theorem \ref{thm:existence}. By construction, this solution is necessarily of rank at most $\alpha$.
\item\label{it:FL} For a Wishart process $X$ we compute its Fourier-Laplace transform  (below, $\tr$ denotes the trace) $$\E \left[ \exp( -\tr((u-iv)X_t)) \,|\, x_0\right] $$ explicitly for
\begin{enumerate}
    \item\label{it:intro:FT_laplace} $u \in S^+(H)$ (the positive self-adjoint operators) and $v=0$;
    \item\label{it:intro:FT_fourier} $v \in S^+(H)$ or $-v\in S^+(H)$ and $u=0$;
    \item\label{it:intro:FT_diag} $u\in S^+(H)$, and $v\in S(H)$ (the self-adjoint operators) and $u$, $v$, $Q$, $A$, and $x_0$ are all jointly diagonizable;
    \item\label{it:intro:FT_mixed} $\alpha \in \N$, $u\in S(H)$, $v \in S(H)$, and $t$ is sufficiently small.
\end{enumerate} 
In all cases it is of exponential affine form, i.e.,
\begin{equation*}
\E \left[ \exp( -\tr((u-iv)X_t)) \,|\, x_0\right] = \exp(-\tr(\psi(t,u-iv)x_0) - \phi(t,u-iv)), \qquad t\geq 0\,,
\end{equation*}
where 
$\psi$ and $\phi$ are solutions of operator valued Riccati equations that can be solved explicitly in all the cases listed above, see Theorem~\ref{thm:laplace_transform} for (a)--(c) 
 and Corollary~\ref{cor:laplace_general} for (d).
As a consequence we obtain that an infinite-dimensional Wishart process is an affine process satisfying the Markov property and is thus unique in law, see Corollaries~\ref{cor:Wishart_Markov} and~\ref{cor:wishart_uniqueinlaw}. 
\item If $Q$ is injective and if there exists a $t>0$ such that $e^{tA}$ is injective, then the existence of a probabilistically and analytically weak solution $X$ to~\eqref{eq:Wishart-formal} 
\emph{implies} that $\alpha \in \N$, $\operatorname{rank}(x_0)\leq \alpha$ and $\operatorname{rank}(X_t) =\alpha$ a.s. for almost all $t >0$,
see Theorem \ref{thm:char_infdimwishart} and Remark \ref{injectiverankn}.
Note that if $A$ is bounded or self-adjoint, then there exists a $t>0$ such that $e^{tA}$ is injective, see Remark~\ref{rem:semigroup_injective}. In particular, Theorem \ref{thm:intro} is thus a consequence of Theorem \ref{thm:char_infdimwishart} (and Theorem \ref{thm:existence}).
\label{necessary-1}
\item If there exists a probabilistically and analytically weak solution $X$ to~\eqref{eq:Wishart-formal}, then either $\operatorname{rank}(X_t)\geq \operatorname{rank}(Q)$ a.s.\ for almost all $t> 0$, or $\alpha \in \N$ and $\operatorname{rank}(X_t) = \alpha$ a.s.\ for almost all $t>0$, see Corollary~\ref{cor:alpha_is_n}. This provides new insights even in the finite-dimensional setting, since a characterisation of Wishart processes in $S(\R^n)$ that are of rank at most $k$ is only known when $A\equiv 0$ and $Q=\id_{\R^n}$ see~\cite[Theorem 3.10]{graczyk2018characterization}.
Moreover, in the infinite dimensional setting under the condition that $\operatorname{rank}(Q) = \infty$, this result implies that  finite-rank Wishart processes
exist if and only if $\alpha \in \mathbb{N}$, see Corollary \ref{cor:charWishartsimple}.

\label{necessary-2}
\item\label{it:feller} If $\e^{tA}$ is injective for all $t\geq 0$, then a probabilistically and analytically weak solution to~\eqref{eq:Wishart-formal} is Feller with respect to (a minor refinement of) the weak-$^*$-topology on the space of self-adjoint trace class operators, see Theorem~\ref{thm:feller}. 
\end{enumerate}

The proof of \ref{sufficient}, i.e., of the existence of a solution to~\eqref{eq:Wishart-formal}, is inspired by the construction presented in \cite{Bru:1991}: we build the solution by considering $X=Y^*Y$, where $Y$ is a suitably chosen $L_2(H,\R^{\alpha})$-valued Ornstein-Uhlenbeck process with $\alpha \in \mathbb{N}$ and $L_2(H, \R^\alpha)$ the space of Hilbert-Schmidt operators from $H$ to $\R^\alpha$. In Remark~\ref{remark:otherY} below we discuss some of the challenges that come with this approach in the infinite-dimensional setting.

Our approach for~\ref{it:FL}, i.e., for establishing the Fourier-Laplace transform of a solution to~\eqref{eq:Wishart-formal}, is in a sense also classical: we apply the It\^o formula to deduce the associated Riccati differential equations.
However, the Riccati equations only allow for a classical solution when the initial value is in $D(A)$, so a subtle approximation argument is needed to obtain the full-blown Fourier-Laplace transform.\par 
The explicit expression of the Fourier-Laplace transform of a Wishart process is not only relevant for analysing properties of the process, it is also crucial for the derivation of closed-form formulas for option pricing using Fourier techniques (see for example \cite{eberlein2010analysis, carr1999option}). To the best of our knowledge, in the finite-dimensional setting (i.e., when $H=\R^n$) the Fourier-Laplace transform has only been derived for the following cases: $u\in S^+(\R^n)$ and $v=0$ (\cite[Theorem 3]{Bru:1991}, \cite[p.12]{GM:11}), or $A=0$ and  $Q=\id_{\R^n}$, see \cite[Theorem 1.1]{Mayerhofer:2019} and the remarks concerning some flaws in the classical literature. Hence our results also extend the results in the finite-dimensional setting.

To prove \ref{necessary-1} we use the Laplace transform obtained in Theorem \ref{thm:laplace_transform} to deduce the Laplace transform of the finite-dimensional projections of an infinite-dimensional Wishart process.
The necessity of $\alpha \in \mathbb{N}$ and $\text{rank}(x_0) \leq \alpha$ then follows from the  characterisation theorem \cite[Theorem 1.1]{M:19} for finite-dimensional non-central Wishart distributions.

As for the proof of~\ref{necessary-2}, the key idea is 
to apply the It\^o formula to the determinant of a suitably chosen finite-dimensional projection of $X_t$ to obtain an expression that leads to the desired conclusion. This has been exploited in the finite-dimensional setting (without projections) e.g.\ in the proof of \cite[Proposition 4.18]{CFMT:11}. However, our proof involves a very technical and subtle approximation result that is needed to deal with the fact that $A$ is unbounded (which restricts us in the projections of $X_t$ that we can take) and the fact that we wish to consider random initial values.

Finally, regarding \ref{it:feller}, we first show that the cone of positive self-adjoint trace-class operators can be equipped with a minor modification of the weak-$*$-topology  to render this space a locally compact Polish space, see Proposition~\ref{prop:posFinRank_LCCB} (specifically, instead of testing only against compact operators, we test against operators of the form $c\id_H + K$, where $c\in \R$ and $K$ is a compact operator on $H$ and $\id_H$ the identity on $H$). This local compactness result in combination with the Laplace transform formula allows us to establish that the solution to~\eqref{eq:Wishart-formal} is Feller, see Theorem~\ref{thm:feller}. 

\subsection{Outlook and open problems}

Theorem~\ref{thm:intro} provides a clean characterisation of infinite-dimensional Wishart processes when $Q$ is injective and $A$ is bounded, but our results also give rise to various intriguing questions. Firstly, we were not able to rule out the existence of a Wishart process when $Q$ is injective, $\alpha \notin \N$, and $\e^{tA}$ is not injective for all $t>0$:
\begin{openproblem}\label{op:char_wishart}
Let $Q$ in~\eqref{eq:Wishart-formal} be injective. Does there exist an unbounded operator $A$ such that~\eqref{eq:Wishart-formal} allows for a solution for some $\alpha\notin \N$? Note that by Corollary~\ref{cor:alpha_is_n} such a process is necessarily of infinite rank a.s.~for almost all $t >0$.
\end{openproblem}

Secondly, we have little insight (even for $H=\R^n$) of existence of Wishart processes when $Q$ is not injective and $\alpha\notin \N$. Corollary~\ref{cor:alpha_is_n} (see also~\ref{necessary-2} above) does not exclude the existence of such processes and at least in certain special cases the finite dimensional results actually imply their existence.
Indeed, let $\alpha \in (0, \infty)$ and suppose that there exists an orthonormal system $(h_k)_{k=1}^{\lceil \alpha \rceil}$ such that $Q=\sum_{k=1}^{\lceil \alpha \rceil} q_k h_k \otimes h_k$ and $A \equiv 0$. Then,
it follows for instance from~\cite[Theorem 2]{Bru:1991} that 
there exists a probabilistically weak solution $(X_t)_{t\geq 0}$ to~\eqref{eq:solX_alpha} satisfying $\operatorname{rank}(X_t)\leq \lceil \alpha \rceil$ for all $x_0$ of the form
$x_0 = \sum_{k = 1}^{\lceil \alpha \rceil} a_k h_k \otimes h_k$, with $a_1,\ldots,a_{\lceil \alpha \rceil}\in\R$ distinct. Note that in this case the solution is confined to a finite dimensional subcone which is isomorphic to $S^+(\mathbb{R}^{\lceil \alpha \rceil})$. In other words, the finite dimensional solutions with values in $S^+(\mathbb{R}^{\lceil \alpha \rceil})$ are just embedded into $S_1^+(H)$. We believe however that also more complex situations could appear.

\begin{openproblem}
If $Q$ is not injective, for what $\alpha\in \R\setminus \N$ and what $x_0\in S_1^+(H)$ does a solution to~\eqref{eq:Wishart-formal} exits? Which role does the operator $A$ play?
\end{openproblem}

Finally, the analysis of existence and characterisation of Wishart processes in finite-dimensional settings is usually also related to the study of the behaviour of the associated eigenvalues and eigenvectors, see e.g.~\cite{bru1989diffusions, graczyk2013multidimensional, katori2004symmetry}. In our setting we exploit the Fourier-Laplace transform and the form of the generator of \eqref{eq:Wishart-formal}
to draw our results. In a forthcoming paper, we aim to derive also the stochastic differential equations for the eigenvalues and eigenvectors of the Wishart process\eqref{eq:Wishart-formal}. As in the recent work 
\cite{bertucci2022spectral}, these eigenvalue equations can then be related to the limiting spectral measure of 
\[
\frac{1}{N} Y^*_t Y_t,
\]
as $N$ tends to infinity, 
where $Y$ is an $n \times N$ dimensional OU-process. Indeed, as conjectured\footnote{Everything is rigorously proved up to the existence part.} in \cite{bertucci2022spectral} for the Brownian case with  $n  \geq N$
the  cumulative
distribution function of the limiting spectral measure is the unique viscosity solution of a certain  partial integro-differential equation. A similar result should also hold true for the case $n < N$, even though -- as remarked in \cite[page 21]{bertucci2022spectral} -- the behaviour of the limiting spectral measures  is expected to be 
entirely different due to an accumulation of eigenvalues at $0$. Note that when $n$ is fixed and $N$ only tends to infinity we actually recover the setting of the current paper. We thus expect that our results  can contribute to a viscosity solution theory for the limiting spectral measure when $n < N$.

\subsection{Structure of the article}
In Subsection~\ref{ssec:notation} below we introduce the notation that is used throughout this article. In Section \ref{sec:squares} we prove the existence of  finite-rank Wishart processes with values in $S_1^+(H)$. In Section \ref{sec:Riccati} we derive the Fourier-Laplace transform of a solution to \eqref{eq:Wishart-formal} and we consider some explicit examples such as the characteristic function of the $\tr(X_t)$.
In Section \ref{sec:characterization} we present necessary conditions for the existence of infinite-dimensional Wishart processes: in Subsection~\ref{ssec:char_laplace} we discuss the case when $Q$ and $e^{tA}$ are injective, for some $t>0$ and in Subsection~\ref{ssec:char_finrank} we characterise finite rank Wishart processes. Finally, in Section~\ref{sec:Feller} we show that the state space $S^+_1(H)$ is locally compact when equipped with a minor modification of the weak-$*$-topology and we use this to show that $X$ is a Feller process.

\subsection{Notation}\label{ssec:notation}
For $(X,\| \cdot \|_{X})$ a Banach space and $E\subset X$ we let $\cB(E)$ denote the (trace) Borel $\sigma$-algebra on $E$.\par  
Let $(H_1, \langle \cdot,\cdot \rangle_{H_1})$ and $(H_2, \langle \cdot, \cdot\rangle_{H_2})$ be separable Hilbert spaces (real or complex). 
Then $\langle \cdot , \cdot \rangle_{H_1}$ is linear in the first argument (and conjugate-linear in the second), $\id_{H_1}\colon H_1 \rightarrow H_1$ denotes the identity on $H_1$, $(L(H_1,H_2), \left\| \cdot \right\|_{L(H_1,H_2))})$ denotes the Banach space of bounded linear operators from
$H_1$ to $H_2$ (endowed with the operator norm), $(K(H_1,H_2),\left\| \cdot \right\|_{L(H_1,H_2))})$ denotes the space of compact operators from $H_1$ to $H_2$ (this is a closed subspace of $L(H_1,H_2)$),
and we set $L(H_1):= L(H_1,H_1)$. The adjoint of an operator $A \in L(H_1,H_2)$
is denoted by $A^*$, i.e., $\langle Ah,g\rangle_{H_1} = \langle h, A^*g\rangle_{H_2}$ for all $h\in H_1$, $g\in H_2$. Note that if $A\in L(\C^n) = \C^{n\times n}$ then $A^*$ is the conjugate transpose of $A$.  For the definition of the adjoint $A^*$ of an unbounded operator $A$ we refer to e.g.~\cite[Appendix B]{EngelNagel:2000}.

For all $p\in [1,\infty)$ let 
$(L_{p}(H_1,H_2), \left\| \cdot \right\|_{L_p(H_1,H_2)})$ be the Banach space of Schatten class operators from $H_1$ to $H_2$, i.e.,
\begin{equation}
 L_{p}(H_1,H_2) = \Big\{ A \in K(H_1,H_2) \colon \sum_{\lambda \in \sigma(A^*A)} \lambda^{p/2} < \infty \Big\},
\end{equation}
and $\| A \|_{L_p(H_1,H_2)}^{p} = \sum_{\lambda \in \sigma(A^*A)} \lambda^{p/2}$ (we assume the reader is familiar with the spectral theorems for bounded self-adjoint and compact self-adjoint operators).
In particular, $L_2(H_1,H_2)$ is the space of Hilbert-Schmidt operators and $L_1(H_1,H_2)$ is the space of trace-class operators 
from $H_1$ to $H_2$. Recall that $L_2(H_1,H_2)$ is a (separable) Hilbert space under the inner product 
$\langle A, B \rangle_{L_2(H_1,H_2)} = \sum_{n=1}^{\infty} \langle A h_n, B h_n \rangle_{H_2}$, where $(h_n)_{n\in \N}$ is an 
orthonormal basis for $H_1$ and the inner product does not depend on the choice of the orthonormal basis. Also recall that we 
have, for all $p,q\in [1,\infty)$ such that $\frac{1}{p} + \frac{1}{q} = 1$ and all $A\in L_p(H_1,H_2), B\in L_{q}(H_2,H_1)$,
$C\in L(H_2)$, $D\in L(H_1)$
that 
\begin{equation}\label{eq:Schattendual}
 \| A \|_{L_p(H_1,H_2)} = \| A^* \|_{L_p(H_2,H_1)} = \| A^* A \|_{L_{p/2}(H_1)}^{\frac{1}{2}},
\end{equation}
\begin{equation}\label{eq:ideal}
\| CAD \|_{L_{p}(H_1,H_2)} \leq \| C \|_{L(H_2)} \| A \|_{L_p(H_1,H_2)} \|D \|_{L(H_1)}, 
\end{equation}
and
\begin{equation}\label{eq:SchattenHolder}
 \| AB \|_{L_1(H_1)} \leq \| A \|_{L_{q}(H_1,H_2)} \| B \|_{L_{p}(H_2,H_1)}.
\end{equation}
In addition, we recall that the trace of $A\in L_1(H_1)$ is defined by
\begin{equation}
 \tr(A) = \sum_{n\in \N} \langle Ah_n, h_n \rangle_{H_1} \in  \C ,
\end{equation}
where $(h_n)_{n\in \N}$ is an orthonormal basis for $H_1$; $\tr(A)$ does not depend on the choice of the orthonormal basis. Writing $V'$ for the dual of a Banach space $V$, 
we recall (see, e.g.,~\cite[Section 19]{Conway:2000}) that the dual space of compact operators satisfies $(K(H_1))'\simeq L_1(H_1)$  under the duality paring
\begin{equation}
 \langle A, B \rangle_{L_1(H_1),K(H_1)}
 = \tr(B^* A),\quad A\in L_1(H_1),\, B\in K(H_1).
\end{equation}
Note that $(L_1(H_1))'\simeq L(H_1)$ under the same paring.\par 
We let $S(H_1)$, $S_c(H_1)$, and $S_p(H_1)$ denote the (closed) subspaces of $L(H_1)$, $K(H_1)$, and $L_p(H_1)$ consisting of all operators that are self-adjoint, and we let $S^+(H_1)$, $S_c^+(H_1)$, and $S_p^+(H_1)$ denote the (closed) subsets of $S(H_1)$, $S_c(H_1)$, and $S_p(H_1)$
consisting of all self-adjoint operators $A$ satisfying $\sigma(A)\subseteq [0,\infty)$, and we let $S^{++}(H_1)$, $S_c^{++}(H_1)$, and $S_p^{++}(H_1)$ denote the subsets of $S(H_1)$, $S_c(H_1)$, and $S_p(H_1)$
consisting of all self-adjoint operators $A$ satisfying $\sigma(A)\subseteq (0,\infty)$. 

We will frequently use the following lemma (which relies on the spectral theorem for compact self-adjoint operators). The proof is straightforward (and under obvious adaptations the result also holds when $H_1$ and/or $H_2$ are finite-dimensional).

\begin{lemma}\label{lemma:Lp(L2)}
Let $H_1$ and $H_2$ be separable Hilbert spaces, and for $i\in \{1,2\}$ let $A_i\in S_c(H_i)$, i.e., $A_i = \sum_{j\in \N} a_j^{(i)} h_j^{(i)} \otimes h_j^{(i)}$, a sequence $(a_j^{(i)})_{j\in \N} \in \ell_{\infty}$, and an orthonormal basis $(h_j^{(i)})_{j\in \N}$ for $H_i$. Let $\cA \in L(L_2(H_1,H_2))$ be given by $\cA(B)=A_2 B A_1$. Then $\cA \in S_c(L_2(H_1,H_2))$ and 
\begin{equation*}
 \cA(B) 
 = 
 \sum_{i,j\in \N} 
    a_i^{(1)} a_j^{(2)} (h_i^{(1)} \otimes h_j^{(2)}) \otimes (h_i^{(1)} \otimes h_j^{(2)}).
\end{equation*}
In particular, for all $p\in [1,\infty)$ one has that $\cA \in L_p(L_2(H_1,H_2))$ if and only 
if $A_1, A_2 \in L_p(H_i)$, and moreover 
\begin{equation}\label{eq:Lp(L2)}
 \| \cA \|_{L_p(L_2(H_1,H_2))}
 =
 \| A_1 \|_{L_p(H_1)} \| A_2 \|_{L_p(H_2)}.
\end{equation}
\end{lemma}

For the convenience of the reader we recall the singular value decomposition for compact operators (see e.g.~\cite[Thm.\ VI.3.6]{Werner:2000}). 
\begin{theorem}\label{thm:singular_value_decomposition}
Let $H_1,H_2$ be separable Hilbert spaces and let $A\in K(H_1,H_2)$. Then there exist $s_1\geq s_2\geq \ldots \geq 0$ and orthonormal systems $(e_n)_{n\in \N}$ and $(f_n)_{n\in \N}$ in $H_1$ and $H_2$ such that $$A = \sum_{n\in \N} s_n e_n \otimes f_n.$$
\end{theorem}

%

\section{Existence of infinite-dimensional Wishart processes}\label{sec:squares}

Let $H$ be a real Hilbert space, let $A\colon D(A)\subset H \rightarrow H$ be the generator of a $C_0$-semigroup, let $Q\in S^+(H)$, let $n\in \N$, and let $W = (W_t)_{t\geq 0}$ be an $L_2(H)$-cylindrical Brownian motion, i.e., (formally) $W = \sum_{k\in \N} \beta_k(t) C_k$ where $(C_k)_{k\in\N}$ is an orthonormal basis for $L_2(H)$ and $(\beta_k)_{k\in \N}$ is a sequence of independent standard Brownian motions (see also~\cite[Def.\ 2.2]{NeervenVeraarWeis:2015}). Finally, let $x_0 \in S_1^+(H)$. Consider the stochastic differential equation in $S^+(H)$ formally given by:
\begin{equation}\left\{
\begin{aligned}\label{eq:wishart1}
dX_t&= nQ\, dt + X_t A \,dt + A^* X_t \,dt +\sqrt{X_t} \,dW_t \sqrt{Q} + \sqrt{Q} \,dW_t^* \sqrt{X_t},\, t\geq 0;\\
X_0 & = x_0.
\end{aligned}\right.
\end{equation}

Theorem~\ref{thm:existence} below ensures the existence of a probabilistically and analytically weak $S_1^+(H)$-valued solution to~\eqref{eq:wishart1} provided $x$ is of rank at most $n$ and 
$(\e^{sA}\sqrt{Q})_{s\geq 0}$ satisfies the integrability condition~\eqref{eq:ass_integrability}.  
Inspired by~\cite{bru1989diffusions}, we prove Theorem~\ref{thm:existence} 
by showing that $X_t = Y_t^* Y_t$, where $Y$ is an $L_2(H,\R^n)$-valued Ornstein-Uhlenbeck process.

\begin{theorem}\label{thm:existence}
Let $H$ be a separable real Hilbert space, let $A\colon D(A)\subset H \rightarrow H$ be the generator of a $C_0$-semigroup $(\e^{tA})_{t\geq 0}$, 
let $Q\in S^+(H)$, let $n\in \N$, let $(\Omega,\cF,\P,(\cF_t)_{t\geq 0})$ be a filtered probability space satisfying the usual conditions and rich enough to allow for an $L_2(H)$-cylindrical $(\cF_t)_{t\geq 0}$-Brownian motion, and let $p\in [1,\infty)$ and $x_0\in L^p((\Omega,\cF_0,\P),S^+_1(H))$ be such that $x_0$ is of rank at most $n$ $\P$-a.s.
Assume moreover that 
\begin{equation}\label{eq:ass_integrability}
 \int_{0}^{t} \| \e^{s A} \sqrt{Q} \|_{L_2(H)}^2\,ds 
 < \infty
\end{equation}
for all $t\geq 0$.
Then there exists an $L_2(H)$-cylindrical $(\cF_t)_{t\geq 0}$-Brownian motion $W$ and a continuous adapted $S^+_1(H)$-valued process $(X_t)_{t\geq 0}$ such that $X_t$ is of rank at most $n$,
\begin{equation}\label{eq:Xintegrability}
 \E \sup_{s\in [0,t]} \| X_s \|_{L^1(H)}^p 
 < \infty, 
\end{equation}
and 
\begin{equation}\label{eq:solX}
\begin{aligned}
 \langle X_t g, h\rangle_{H} 
 & = 
 \langle x_0 g, h \rangle_{H}
 + 
 \int_0^{t} 
 (n\langle Q g, h \rangle_H 
 + \langle X_s A g, h \rangle_{H} 
 + \langle X_s g, Ah \rangle_H ) \,ds
 \\ & \quad 
 + \int_{0}^{t} \langle \sqrt{X}_s \,dW_s \sqrt{Q}g, h\rangle_H + \int_0^t \langle \sqrt{Q} \,dW^*_s \sqrt{X}_s g, h \rangle_H
\end{aligned}
\end{equation}
for all $t\geq 0$ and all $h,g\in D(A)$. In particular,
if $A\in L(H)$ and $Q\in S_1^{+}(H)$, then 
\begin{equation}\label{eq:solXnice}
\begin{aligned}
 X_t 
 & = 
 x_0
 + 
 \int_0^{t} 
 ( n Q + X_s A + A^* X_s ) \,ds 
 +
 \int_0^{t} \sqrt{X}_s \,dW_s \sqrt{Q}
 + 
 \int_0^t \sqrt{Q} \,dW^*_s \sqrt{X}_s
\end{aligned}
\end{equation}
for all $t\geq 0$.
\end{theorem}

\begin{proof}
In order to construct an $L_2(H,\R^n)$-valued Ornstein-Uhlenbeck process $Y$ such that $X=Y^*Y$, we first define 
$\cA\colon D(\cA) \subseteq L_2(H, \R^n) \rightarrow L_2(H, \R^n)$ by 
\begin{equation}
\begin{aligned}\label{eq:defcalA}
 D(\cA) &= \{ C \in L_2(H, \R^n) \colon CA \in L_2(H, \R^n) \};\\
 \cA(C) & = CA, \quad C\in D(\cA).
\end{aligned}
\end{equation}
Here ``$CA \in L_2(H, \R^n)$'' is to be read as: ``the linear mapping $D(A) \ni h \mapsto CAh \in \R^n$ extends to a Hilbert Schmidt operator on $H$ with values in $\R^n$''; note that $C \in D(\cA)$ if and only if there exist $v_1,\ldots,v_n\in D(A)$ and $e_1,\ldots,e_n\in \R^n$ such that $C = \sum_{k=1}^{n} v_k \otimes e_k$.
Note moreover that $\cA$ is the generator of a $C_0$-semigroup on $L_2(H,\R^n)$ given by $\e^{ t \cA}( C ) = C \e^{t A}$, $t\geq 0$. In addition, we define $\cQ \in L(L_2(H,\R^n))$ by 
\begin{equation}\label{eq:defcalQ}
\cQ (C) = CQ, \quad C\in L_2(H,\R^n).
\end{equation} 
As $\langle C,D \rangle_{L_2(H,\R^n)}= \langle C^*, D^* \rangle_{L_2(\R^n, H)}$ for all $C,D\in L_2(H,\R^n)$, it is easily verified that $\cQ \in S^+(L_2(H,\R^n))$. Moreover, due to~\eqref{eq:Lp(L2)} and~\eqref{eq:ass_integrability} we have 
\begin{equation}\label{eq:integrability}
 \int_0^{t} \| \e^{s \cA} \sqrt{\cQ} \|_{L_2(L_2(H,\R^n))}^2 \,ds 
 = \int_0^{t} \| \e^{s A} \sqrt{Q} \|_{L_2(H,\R^n)}^2 \,ds < \infty
\end{equation}
for all $t\geq 0$. \par 
Next, consider the following stochastic differential equation in $L_2(H,\R^n)$:
\begin{equation}\label{eq:SDE_Y}
dY_t= \cA Y_t dt + \sqrt{\cQ} dB_t,\quad t\geq0; \quad Y_0=y_0,
\end{equation}
where $B$ is an $L_2(H,\R^n)$-cylindrical Brownian motion and 
$y_0 \in L^{p}((\Omega,\cF_0,\P),L_2(H,\R^n))$ satisfies $y_0^*y_0=x_0$ (note that every $x\in S_1^+(H)$ that is of rank at most $n$ can be written as $y^*y$ for some $y\in L_2(H,\R^n)$, and this decomposition can be done in a measurable way, see Lemma~\ref{lem:measurable_EVD}). \par 

Classical stochastic integration theory in 
Hilbert spaces (see e.g.~\cite[Theorems 5.2 and 5.4]{DaPratoZabczyk:1992}) and~\eqref{eq:integrability} ensure the existence 
of an adapted process $Y\colon [0,\infty)\times \Omega \rightarrow L_2(H,\R^n)$ with continuous sample paths
satisfying
\begin{equation}\label{eq:solY}
 Y_t = \e^{t\cA}(y_0) 
 + \int_{0}^{t} \e^{(t-s)\cA}\sqrt{\cQ} \,dB_s
 = y_0\e^{tA} + \int_0^t \,dB_s \sqrt{Q} \e^{(t-s)A}, \quad t \geq 0.
\end{equation}

Note that the fact that $x_0\in L^p(\Omega,L_1(H))$ (whence $y_0\in L^p(\Omega,L_2(H))$), assumption~\eqref{eq:integrability} and~\eqref{eq:solY} imply that 
\begin{equation}\label{eq:Yintegrability}
\E \sup_{s\in [0,t]} \| Y_s \|_{L_2(H,\R^n)}^p < \infty\,,
\end{equation}
for all $t>0$.\par 
%

Analogous to the finite dimensional case considered in~\cite{Bru:1991}, we will show that $(Y_t^* Y_t)_{t\geq 0}$ provides an analytically and probabilistically weak solution to~\eqref{eq:wishart1}. We begin by expressing $\langle Y h, Y g\rangle_{\R^n}$, $h,g\in D(A)$, as an It\^o process. To this end, first observe that by the stochastic Fubini theorem we have
\begin{equation*}
 Y_t h = y_0 h 
 + 
 \int_{0}^{t} Y_s A h \,ds 
 +
 \int_{0}^{t} \,dB_s \sqrt{Q} h
\end{equation*}
for all $t\geq 0$ and all $h\in D(A)$.  
Let $(e_k)_{k=1}^{n}$ be an orthonormal basis for $\R^n$, and let $(h_j)_{j\in \N}$ be an orthonormal basis for $H$. 
Applying the It\^o formula (see, e.g.,~\cite[Section 4.5]{DaPratoZabczyk:1992}) we obtain
\begin{equation}\label{eq:YYito}
\begin{aligned}
\langle Y_t g, Y_t h \rangle_{\R^n}
& =
\langle y_0 g, y_0 h \rangle_{\R^n}
+
\int_{0}^{t} \langle Y_s g, Y_s A h \rangle_{\R^n} + \langle Y_s A g, Y_s h \rangle_{\R^n} \,ds
\\ & \quad +
\int_{0}^{t} \langle Y_s g, dB_s \sqrt{Q} h \rangle_{\R^n} + \langle dB_s \sqrt{Q}  g , Y_s h \rangle_{\R^n} 
\\ & \quad 
+
t \sum_{k=1}^{n} \sum_{j\in \N} 
    \langle h_j, \sqrt{Q} h\rangle_H
    \langle h_j, \sqrt{Q} g\rangle_H \langle e_k, e_k \rangle 
\\ & =  
 \langle y_0^*y_0g, h\rangle_{H} 
 + \int_{0}^{t}\left( 
    n\langle Qg, h\rangle_{H} 
    + \langle Y_s^* Y_s A g, h \rangle_{H}
    + \langle Y_s^* Y_s g, A h \rangle_{H} \right)\,ds 
 \\ & \quad 
 + \int_{0}^{t} \langle Y_s^* \,dB_s \sqrt{Q} g , h \rangle_{H} + \int_{0}^{t} \langle \sqrt{Q} \,dB_s^* Y_s g,h \rangle_{H}\,,
\end{aligned}
\end{equation}
for all $g,h\in D(A)$.\par 

Analogous to~\cite{Bru:1991}, we now wish to show that $Y_s^* \,dB_s = \sqrt{Y_s^* Y_s} d W_s$ for some $L_2(H)$-cylindrical Brownian motion $W$.  
To this end, let $\tilde{B}$ be an $L_2(H)$-cylindrical Brownian motion independent of $B$. Moreover, %
let $P_s \colon H \rightarrow H$ be\label{proj} the orthogonal 
projection onto $\operatorname{range}(Y_s^*)$, $s\in [0,\infty)$. The singular value decomposition implies that 
$\| Y_s(Y_s^* Y_s)^{-\frac{1}{2}}P_s \|_{L_2(H)} = \sqrt{n} $ 
for all $s \geq 0$, whence we can define
$W\colon L^2([0,\infty);L_2(H)) \rightarrow L^2(\Omega)$
by setting
\begin{align}\label{eq:defW}
W(\varphi) & = \int_{0}^{t} \langle Y_s (Y_s^*Y_s)^{-\frac{1}{2}} P_s \varphi_s, \,dB_s\rangle_{L_2(H)} 
 + \int_0^{t} \langle (\id_{H}-P_s)\varphi_s , \,d\tilde{B}_s \rangle_{L_2(H)}\,,
\end{align}
for all $\varphi\in L^2([0,\infty);L_2(H))$.
One readily checks that  
\begin{equation*}
\E [W(\varphi) W(\psi)] = \langle \varphi, \psi \rangle_{L^2(0,\infty;L_2(H))}
\end{equation*}
for all $\varphi, \psi \in L^2([0,\infty);L_2(H))$, i.e., $W$
is an $L_2(H)$-cylindrical Brownian motion. 
Moreover, as $\sqrt{Y_s^* Y_s}P_s = \sqrt{Y_s^* Y_s}$ and $\sqrt{Y_s^* Y_s} (\id_H - P_s) = 0$, we obtain that  $\sqrt{Y^*_s Y_s}dW_s = Y_s^* dB_s$. 
This and~\eqref{eq:YYito} imply that the 
process $X=Y^*Y$ \label{p:X=Y*Y} satisfies~\eqref{eq:solX}
for all $t\geq 0$ and $g,h\in D(A)$.
The fact that $Y$ is a continuous adapted $L_2(H,\R^n)$-valued process satisfying~\eqref{eq:Yintegrability} ensures that $X$ is a continuous adapted $S_1^+(H)$-valued process such that $X_t$ is of rank $n$ and~\eqref{eq:Xintegrability} is satisfied for all $t\geq 0$.\par 
Finally,~\eqref{eq:solXnice} follows from~\eqref{eq:solX} as under the conditions on $Q$ and $A$ all integrals in~\eqref{eq:solXnice} are well-defined and $\{ g\otimes h \colon g,h\in H\}$ separates points in $L_2(H)$. 
\end{proof}

\begin{remark}\label{rem:OU_alternative}
An equivalent way to construct the process 
$X$ in~Theorem~\ref{thm:existence} 
is to set $X_t= \sum_{i=1}^{n} Y^{(i)}_t \otimes Y^{(i)}_t$, $t\geq 0$, where $Y^{(1)}, \ldots, Y^{(n)}$ are 
independent $H$-valued Ornstein-Uhlenbeck processes.
To see this, let $B$, $y_0$ and $Y$ be as in the proof of Theorem~\ref{thm:existence} (see also ~\eqref{eq:SDE_Y}). Let $(e_i)_{i=1}^{n}$ be an orthonormal basis for $\R^n$. Define the $H$-cylindrical 
Brownian motions $B^{(1)},\ldots,B^{(n)}$ by $B^{(i)}(h 1_{[0,t)}) = B((h\otimes e_i)1_{[0,t)})$, $h\in H, t\geq 0$, and note that $B^{(1)},\ldots,B^{(n)}$ are independent. Set $y_0^{(i)} = y_0^* e_i$ 
and $Y_t^{(i)} = Y_t^* e_i$. Then 
\begin{align}\label{eq:OUprocess}
 Y_t^{(i)} & = \e^{tA^*}y_0^{(i)} + \int_{0}^{t} \e^{(t-s)A^*} \sqrt{Q}\,dB^{(i)}_s,\quad t\geq 0,
\end{align}
and $X_t= \sum_{i=1}^{n} Y^{(i)}_t \otimes Y^{(i)}_t$, $t\geq 0$.
\end{remark}
\begin{remark}
One can show that the process $X$ constructed in the proof 
of Theorem~\ref{thm:existence} is also a mild solution:
\begin{equation}
\begin{aligned}
X_t 
& = 
\e^{tA} x_0 \e^{tA^*}
+
n\int_{0}^{t} \e^{sA}Q\e^{sA^*}\,ds 
\\ 
&\quad 
+
\int_{0}^{t} \e^{(t-s)A} \sqrt{X_s} dW_s \sqrt{Q} \e^{(t-s)A^*} + \e^{(t-s)A} \sqrt{Q} dW_s^* \sqrt{X_s} \e^{(t-s)A^*}
\end{aligned}
\end{equation}
$\P$-a.s.\ for all $t\geq 0$.
\end{remark}

\begin{remark}\label{rem:integrability}
Assumption~\eqref{eq:ass_integrability} is satisfied e.g.\ in the following situations:
\begin{enumerate}
 \item $Q\in S_1^+(H)$: indeed, recall from e.g.~\cite[Theorem 1.2.2]{Pazy:1983} that there exist $M,\omega \geq 0$ such that 
 $\| \e^{sA} \|_{L(H)} \leq M \e^{\omega s}$ for all $s\geq 0$, whence by~\eqref{eq:ideal} we obtain
 \begin{equation}   
\| \e^{s A} \sqrt{Q} \|_{L_2(H)}
 \leq \| \e^{s A} \|_{L(H)} \| \sqrt{Q} \|_{L_2(H)} 
 \leq M \e^{\omega s} \| Q \|_{L_1(H)}
 \end{equation}
 for all $s\geq 0$.
 \item $(\lambda - A)$ is the generator 
 of an analytic $C_0$-semigroup  for some $\lambda \geq 0$ (see~\cite[Section 2.5]{Pazy:1983}) \emph{and} there exists a $\theta\in [0,\frac{1}{2})$ such that 
\begin{equation}\label{eq:ass_analytic}
 \| (\lambda  - A)^{-\theta} Q \|_{L_2(H)} < \infty.
\end{equation}
Indeed, in this case we have, by~\cite[Thm. 2.6.13]{Pazy:1983}, that there exist a constant $C>0$ such that 
 \begin{align*}
  \| \e^{sA} Q \|_{L_2(H)}
  & =
  \e^{\lambda s} \| (\lambda - A)^{\theta} \e^{s(A-\lambda)} (\lambda - A)^{-\theta} Q \|_{L_2(H)}
  \\ & 
  \leq C \e^{\lambda s} s^{-\theta} \| (\lambda  - A)^{-\theta} Q \|_{L_2(H)}\,,
 \end{align*}
 for all $s> 0$, which is square-integrable on $[0,t]$ for all $t\geq 0$ because $\theta<\frac{1}{2}$.\par 
 \end{enumerate}
 A `classical' example of an operator that generates an analytic semigroup is an elliptic differential operator of order $2m$ ($m\in \N$) on a smooth, bounded domain $D\subset \R^d$ ($d\in \N$) with Dirichlet boundary conditions (see e.g.~\cite[Thm. 2.7]{Pazy:1983}).
 In this case, Weyl's theorem implies that $(-A)^{-\theta} \in L_2(L^2(D))$ 
 whenever $\theta > \frac{d}{4m}$, so whenever $d < 2m$ it is possible to choose a $\theta \in [0,\frac{1}{2})$ such that~\eqref{eq:ass_analytic} is satisfied for any $Q\in L(H)$. More specifically, if $A$ is the one-dimensional Laplacian on a bounded interval with Dirichlet boundary conditions and $Q\in L(H)$, then $A$ and $Q$ satisfy~\eqref{eq:ass_analytic} with $\lambda = 0$ for any $\theta>\frac{1}{4}$.  
\end{remark}

\begin{remark}\label{remark:otherY} [Regarding more general Ornstein-Uhlenbeck processes]
The operators $\cA \colon D(\cA)\subseteq L_2(H,\R^n) \rightarrow L_2(H,\R^n)$ and $\cQ \in S^+(L_2(H,\R^n))$ we construct in~\eqref{eq:defcalA} and~\eqref{eq:defcalQ} have a very specific structure. One may therefore wonder if one can obtain an SDE of the form~\eqref{eq:wishart1} by considering $Y^*Y$ for more general Ornstein-Uhlenbeck processes $Y$. This seems to be not the case: indeed,
let $K$ be another Hilbert space and let $\cA \in L(L_2(H,K))$ and $\cQ \in S_1^+(L_2(H,K))$ be general operators (for simplicity we only consider the case that $\cA$ is bounded and $\cQ$ is of trace class here). Then
one can write $\cQ = \sum_{k\in \N} q_k C_k \otimes C_k$, where $(C_k)_{k\in \N}$ is an orthonormal basis for $L_2(H,K)$ and $(q_k)_{k\in \N}$ is a positive sequence in $\ell^1$. The solution to~\eqref{eq:SDE_Y} is still given by~\eqref{eq:solY} and still satisfies~\eqref{eq:Yintegrability}, and by applying It\^o's formula to $\langle Y_t g, Y_t h\rangle_K$ we obtain:
\begin{align*}
\langle Y_t g, Y_t h \rangle
& =
\langle Y_0 g, Y_0 h \rangle
+
\int_{0}^{t} \langle Y_s g, \cA (Y_s) h \rangle_K + \langle \cA (Y_s) g, Y_s h \rangle_K \,ds
\\ & \quad +
\int_{0}^{t} \langle Y_s g, \sqrt{\cQ}(dB_s) h \rangle_K + \langle \sqrt{\cQ}(dB_s) g , Y_s h \rangle_K 
+
t \sum_{k,n\in \N} \langle \cQ( C_k )g, h \rangle_K.  
\end{align*}
We immediately see that in order to obtain $Y_s^*Y_s$ in the deterministic integrands on the right-hand side above, one needs that $\cA(Y) = YA$ for some operator $A\in L(H)$. 
Note however, if we allowed for more general linear drift parts in ~\eqref{eq:wishart1} as for instance in the finite dimensional setting of \cite{CFMT:11}, then $\mathcal{A}$ could potentially also be more general. 

Next, we consider the $L_2(H)$-valued martingale $M_t:=\int_0^{t} (Y_s^*Y_s)^{-\frac{1}{2}}Y_s^* \sqrt{\cQ}(d B_s^*)$, $t\geq 0$. One may verify that the quadratic covariation of $M$ is given by 
\begin{align*} 
& \E \left[ \langle M_t, H_1 \rangle_{L_2(H)} \langle M_t, H_2 \rangle_{L_2(H)} \right]
\\ & \quad 
= \E \int_{0}^{t} \langle \cQ( Y_s (Y_s^*Y_s)^{-\frac{1}{2}}P_s H_1), Y_s (Y_s^*Y_s)^{-\frac{1}{2}} P_s H_2 \rangle_{L_2(H,K)} \,ds
\end{align*}
for all $H_1,H_2 \in L_2(H)$ (here $P_s$ is again the orthogonal projection onto $\operatorname{range}(Y_s)$). Thus, in order to be able to conclude that $M$ is a  Brownian motion, we need that $\cQ(C)=CQ$ for some operator $Q \in L_2(H)$.
Finally, if indeed $\cA(C)= CA$ and $\cQ(C)=CQ$ for some $A,Q\in L(H)$, then Equation~\eqref{eq:Lp(L2)} implies that 
\begin{equation}\label{eq:L2estSG}
 \int_{0}^{T} \| \e^{s \cA} \cQ \|_{L_2(L_2(H,K))}^2 \,ds
= \| \id_{K} \|_{L_2(K)}^2
\int_{0}^{T} \|\e^{sA} Q\|_{L_2(H)}^2 \,ds,
\end{equation}
for all $s\geq 0$, which can only be finite if $K$ is finite-dimensional (note that taking $A$ to be unbounded will not help here). As the integral in~\eqref{eq:L2estSG} must be finite in order for $Y$ to be well-defined (see~\eqref{eq:solY}), it seems that the setting we consider in Theorem~\eqref{thm:existence} is indeed the most general. 

\end{remark}

\section{Fourier and Laplace transforms of infinite-dimensional Wishart processes}\label{sec:Riccati}

The goal of this section is to derive expressions for the Fourier-Laplace transform of an infinite-dimensional Wishart process, see Theorem~\ref{thm:laplace_transform} and Corollary~\ref{cor:laplace_general}. 
For Corollary~\ref{cor:laplace_general} we calculate the Laplace transform of the solution we constructed for the proof of Theorem~\ref{thm:existence}, i.e., the Laplace transform of the process $X=Y^*Y$ where $Y$ is an appropriately chosen $L_2(H,\R^n)$-valued Ornstein-Uhlenbeck process (see p.~\pageref{p:X=Y*Y}). In particular, we exploit that we know that $X$ is of this particular form.\par  
On the other hand, in Theorem~\ref{thm:laplace_transform} we do \emph{not} assume we know that $X=Y^*Y$, we merely assume we know that $X$ is a Wishart process (i.e., a solution to~\eqref{eq:wishart_laplace} below). From this, we derive Riccati equations for the Fourier-Laplace transform, and then provide solutions to these Riccati equations. Naturally, the formulas derived in Corollary~\ref{cor:laplace_general} and Theorem~\ref{thm:laplace_transform} coincide on the intersection of their domains; Corollary~\ref{cor:laplace_general} is relevant because it provides the Fourier-Laplace transform $\E[ \exp(\tr(X_t (u+iv)))]$ for in $u,v\in S(H)$ provided that $t$ is sufficiently small, whereas Theorem~\ref{thm:laplace_transform} is essential for the characterisation of Wishart processes (see Section~\ref{sec:characterization}). In addition, as a corollary of Theorem~\ref{thm:laplace_transform} we obtain that any Wishart process has the Markov property, is unique in law, and is an affine process, see Corollaries~\ref{cor:Wishart_Markov} and~\ref{cor:wishart_uniqueinlaw}.
%
%
%
%
We close this section with some examples, see Subsection~\ref{subs:examples_laplacetrans}.

In order to present our results  we must introduce some notation: for a real Hilbert space $H$ we denote its complexification by $H_{\C}$, i.e., $H_{\C}=H \oplus iH$ endowed with 
$\langle h_1+ig_1, h_2 + i g_2\rangle_{H_{\C}} = \langle h_1, h_2\rangle_H + \langle g_1, g_2 \rangle_H +i (\langle g_1, h_2\rangle_H - \langle h_1, g_2\rangle_H)$, $h_1,h_2,g_1,g_2\in H$. Moreover, we define $\Re(h+ig):=h$, $\Im(h+ig) := g$, and $\overline{h+ig} = h-ig$ for all $h,g\in H$ (note that $\langle h, g \rangle_{H_{\C}} = \langle \bar{g}, \bar{h} \rangle_{H_{\C}}$). For $A\in L(H_{\C})$
we define the transpose $A^T\in L(H_{\C})$
by $\langle A^T h, g\rangle_{H_{\C}} = \langle A \bar{g}, \bar{h} \rangle_{H_{\C}}$, $h,g\in H_{\C}$. Note that in contrast to $A^*$, the operator $A^{\top}$ corresponds to the transpose without conjugation.


Note that any $A\in L(H)$ extends in a trivial and norm-conserving way to an operator $\tilde{A}\in L(H_{\C})$ by setting $\tilde{A}(h + ig)=Ah+iAg$, $h,g\in H$. Moreover, if $A\in S(H)$, then $\tilde{A}^T=\tilde{A}=\tilde{A}^*$, in particular, $\tilde{A}\in S(H_{\C})$ (from now on we do not distinguish between $A$ and $\tilde{A}$). In this section we frequently encounter the set $S^+(H)\oplus iS(H)\subseteq L(H_{\C})$, which
denote the operators $A\in L(H_{\C})$ for which there exist (necessarily unique) $A_1 \in S^+(H)$, $A_2\in S(H)$ such that $A=A_1 + i A_2$. The set $S(H)\oplus i S(H)$ is defined analogously. Note that if $A\in S(H)\oplus iS(H)$, then $A^{T}=A$.
\label{p:Hcomplexification}

\subsection{Establishing the Fourier and Laplace transforms}

Theorem~\ref{thm:laplace_transform} below states the main result of this section, a related result can be found in Corollary~\ref{cor:laplace_general} below.

\begin{theorem}\label{thm:laplace_transform}
Let $H$ be a separable real Hilbert space, let $(\Omega,\cF,\P,(\cF_t)_{t\geq 0})$ be a filtered probability space, let $A\colon D(A)\subset H \rightarrow H$ be the generator of a $C_0$-semigroup, 
let $Q\in S^+(H)$, let $\alpha \in \R$, let $x_0\in L^1((\Omega,\cF_0,\P),S_1^+(H))$,
and assume 
\begin{equation}\label{eq:ass_integrability2}
\int_{0}^{t}\| \e^{sA} \sqrt{Q}\|_{L_2(H)}^2\,ds < \infty
\end{equation} 
for all $t\geq 0$.
Assume moreover that there exists an $L_2(H)$-cylindrical $(\cF_t)_{t\geq 0}$-Brownian motion $W$ and an adapted stochastic process $X\colon [0,\infty)\times \Omega \rightarrow S^+_1(H)$ with continuous sample paths satisfying
\begin{equation}\label{eq:solX_laplace}
\begin{aligned}
 \langle X_t g, h\rangle_{H} 
 & = 
 \langle x_0 g, h \rangle_{H}
 + 
 \int_0^{t} 
 (\alpha \langle Q g, h \rangle_H 
 + \langle X_s A g, h \rangle_{H} 
 + \langle X_s g, Ah \rangle_H ) \,ds
 \\ & \quad 
 + \int_{0}^{t} \langle \sqrt{X}_s \,dW_s \sqrt{Q}g, h\rangle_H + \int_0^t \langle \sqrt{Q} \,dW^*_s \sqrt{X}_s g, h \rangle_H
\end{aligned}
\end{equation}
for all $t\geq 0$ and all $h,g\in D(A)$. Set\footnote{Note that $A^*$ generates the $C_0$-semigroup $\e^{tA^*} = (\e^{tA})^*$, see~\cite[Section I.5.14 and II.2.6]{EngelNagel:2000}, and thus~\eqref{eq:ass_integrability2} ensures that $Q_t\in S_1(H)$.}
\begin{align}\label{eq:Qtdef}
 Q_t = \int_{0}^{t} \e^{sA^*}Q \e^{sA}\,ds,\quad t\geq 0.
\end{align}
Let $t>0$, $u \in S^+(H)$, and $v\in S(H)$,
then 
\begin{align}\label{eq:tr(psiQ)integrable}
| \exp ( - \tr(\psi(t,u-iv)X_t))| & \leq 1,
& \int_0^{t} |\tr(\psi(s,u+iv)Q)|\,ds &<\infty,
\end{align}
and we have
\begin{equation}\label{eq:Xisaffine}
\E \left[ \exp( -\tr((u-iv)X_t)) \,|\, x_0\right]
=
\exp\left( -\tr(\psi(t,u-iv)x_0) - \alpha \int_{0}^{t}\tr(\psi(s,u-iv)Q)\,ds \right)
\end{equation}
in each of the following three cases (note that we make the implicit assertion that the inverses appearing below exist; see also Proposition~\ref{prop:riccati} below):
\begin{enumerate}
\item \label{it:laplace_transform:laplace} $v=0$, $u\in S^+(H)$, and
\begin{align}\label{eq:def_psi:laplace}
   \psi(s,u) &= \e^{sA}\sqrt{u}(\id_H+2\sqrt{u}Q_s \sqrt{u})^{-1}\sqrt{u} \e^{sA^*}, \quad s\in [0,t];
\end{align}
\item \label{it:laplace_transform:fourier} $u=0$, $v \in S^+(H)\cup S^{-}(H)$, and 
\begin{align}\label{eq:def_psi:fourier}
   \psi(s,-iv) &= \e^{sA}\sqrt{|v|}(i\operatorname{sign}(v)\id_{H_{\C}}+2\sqrt{|v|}Q_s \sqrt{|v|})^{-1}\sqrt{|v|} \e^{sA^*}, \quad s\in [0,t];
\end{align}
\item \label{it:laplace_transform:laplace_fourierComm} $A$, $u$, $v$, and $Q$ are jointly diagonizable, i.e., there exists an orthonormal basis $(e_n)_{n \in \N}$ for $H$ and sequences $(a_n)_{n\in \N}$ in $\R$, $(u_n)_{n\in N}$ in $[0,\infty)$, $(v_n)_{n\in N}$ in $\R$, $(q_n)_{n\in \N} $ in $[0,\infty)$ such that $A = \sum_{n\in \N} a_n e_n \otimes e_n$, $u-iv=\sum_{n\in \N} (u_n -iv_n) e_n \otimes e_n$, and $Q=\sum_{n\in \N} q_n e_n \otimes e_n$, and
\begin{equation}
\begin{aligned}\label{eq:def_psi:diagonizable}
   \psi(s,u-iv) &= \e^{sA}(\id_{H_{\C}}+2(u-iv)Q_s)^{-1}(u-iv) \e^{sA^*}
   \\ & = \sum_{n\in \N}  
     \frac{b_n e^{2a_n t}\left(a_n^2+ a_n (u_n+iv_n) q_n(e^{2a_n t} -1)\right)}
     {|a_n + (u_n-iv_n) q_n (e^{2a_n t} -1)|^2} ( e_n\otimes e_n), 
     \quad t \geq 0.
\end{aligned}
\end{equation}
\end{enumerate}  
\end{theorem}

We postpone the proof of Theorem~\ref{thm:laplace_transform} to the next subsection, and first consider some corollaries. Corollary~\ref{cor:Wishart_Markov} establishes the Markov property of solutions to \eqref{eq:solX_laplace} in the sense that the law of $X_t$ conditioned on the $\sigma$-algebra generated by 
$(X_r)_{r \in [0,s]}$, which we denote by $\cL[X_t | (X_r)_{r\in [0,s]}]$, only depends on the $\sigma$-algebra generated $X_s$. We deliberately do not introduce a state space for the Markov process at this point as this depends very much on the conditions on $\alpha, A$ and $Q$ (see Section \ref{sec:characterization} and Section \ref{sec:Feller}).

\begin{corollary}\label{cor:Wishart_Markov}
Assume the setting of Theorem~\ref{thm:laplace_transform}. Then $X$ has the Markov property, i.e., $\cL[X_t|(X_r)_{r\in [0,s]}] = \cL[X_t | X_s]$, and $X$ is an affine process in the sense that 
\[
\log(\E [ \exp(-\tr(uX_t)) | X_s ]), \quad u \in S^+(H),
\]
is an affine function in $X_s$ for every $s \leq t$.
\end{corollary}

\begin{proof}
The assertion concerning the affine property of $X$ simply follows from \eqref{eq:Xisaffine}.
Concerning the Markov property, note that the Stone-Weierstrass theorem implies that the law of an $S^+(\R^n)$-valued random variable $\tilde{X}$ is uniquely characterised by $\{\E[\e^{\tr(u \tilde{X})}] \colon u\in S^+(\R^n)\}$.
Moreover, the law of an $S^+(H)$-valued random variable $\bar{X}$ is uniquely characterised by the laws of $P_n^* \bar{X} P_n$, $n\in \N$, where $P_n\in L(\R^n,H)$, $n\in \N$, is a sequence of isometries satisfying $\lim_{n\rightarrow \infty} P_nP_n^*h =h$ for all $h\in H$. 
Thus $\cL[X_t|\cF]$, the law of $X_t$ given a $\sigma$-algebra $\cF$, is characterised by 
$\{\E[\e^{\tr(u X_t)} | \cF ] \colon u\in S^+(H)\}$.\par 
By the continuity of $X$ it suffices to consider only finitely many marginals, i.e.~to prove that $\cL(X_t|X_{s_1},\ldots,X_{s_n})=\cL(X_t|X_{s_n})$ for all $0\leq s_1 < \ldots < s_n < t$. Moreover, by the tower property of the conditional expectation this reduces to prove $\cL(X_t|X_{r}, X_{s})=\cL(X_t|X_{s})$ for all $0\leq r < s < t$. Finally, by the definition of the conditional expectation we thus only need to verify for all $u\in S^+(H)$, $0\leq r < s < t $, and $B\in \cB(S^+_1(H))$ that 
\begin{equation*}
    \E \left [
       \E [ \exp(-\tr(uX_t)) | X_s ] 1_{B}(X_r) 
    \right ]
    =    
    \E \left [
       \exp(-\tr(uX_t)) 1_{B}(X_r)
    \right ].
\end{equation*}
But this follows from Theorem~\ref{thm:laplace_transform}~\ref{it:laplace_transform:laplace} and the fact that $\psi(s-r,\psi(t-s,u))=\psi(t-r,u)$ for all $s,t\geq 0$ and all $u\in S^+(H)$, see also Proposition~\ref{prop:riccati}~\ref{it:riccati:unique} and~\ref{it:riccati:psi1} below (here $\psi$ is as defined in~\eqref{eq:def_psi:laplace}).
\end{proof}

\begin{corollary}\label{cor:wishart_uniqueinlaw}
Assume the setting of Theorem~\ref{thm:laplace_transform}. Then $X$ is unique in law.
\end{corollary}

\begin{proof}
As $X$ has continuous sample paths, it suffices to verify that $( X_{t_k})_{k=0}^{m}$ is unique in law for all $m\in \N$ and all $0\leq t_0 \leq \ldots \leq t_m$.
Arguing as in the proof of Corollary~\ref{cor:Wishart_Markov}, we have that the law of $(X_{t_k})_{k=1}^{m}$ is determined by 
\begin{equation*}
    \left\{\E\left( \prod_{k=1}^{m}\exp(-\tr(u_k,X_{t_k}) \right) \colon u_1,\ldots, u_m\in S^+(H)\right\} 
\end{equation*} 
This together with the tower property of the conditional expectation, the Markov property of $X$ (see Corollary~\ref{cor:Wishart_Markov}), and Theorem~\ref{thm:laplace_transform} prove that $X$ is indeed unique in law.
\end{proof}
\begin{corollary}\label{cor:laplace_transform_simple}
Assume the setting of Theorem~\ref{thm:laplace_transform}, in particular, let $u\in S^+(H)$.
Then\footnote{For $B\in S^+_1(H)$ we have $\det(\id_H+B) = \prod_{\lambda \in \sigma(B)} (1+\lambda)^{d(\lambda)}$, where $d(\lambda)\in \N$ is the geometric multiplicity of $\lambda$. Note that $\log(\prod_{\lambda\in \sigma(B)}(1+\lambda)^{d(\lambda)}) = \sum_{\lambda \in \sigma(B)} d(\lambda)\log(1+\lambda) \leq \sum_{\lambda \in \sigma(B)} d(\lambda) \lambda = \| B \|_{L_1(H)}$.} $\sqrt{u}Q_t \sqrt{u}\in S^+_1(H)$, $\det(\id_H + 2 \sqrt{u}Q_t \sqrt{u}) \in [1,\infty)$, and
\begin{equation}\label{eq:laplace_transform_simple}
\begin{aligned}
&\E[ \exp(-\tr(uX_t)\,|\,x_0 ]
\\ & \qquad =
\det(I_H + 2 \sqrt{u} Q_t \sqrt{u})^{-\frac{\alpha}{2}}
\exp\left(-\tr\left(\e^{tA} \sqrt{u}(\id_H+ 2\sqrt{u} Q_t \sqrt{u} )^{-1} \sqrt{u}\e^{tA^*} x_0\right)\right).
\end{aligned}
\end{equation}
\end{corollary}

\begin{proof}
As $Q_t \in S^+_1(H)$, we have $\sqrt{u}Q_t \sqrt{u}\in S^+_1(H)$, whence $\det(\id_H + 2 \sqrt{u}Q_t \sqrt{u}) \in [1,\infty)$. 
Now, define $\tilde{\phi}(t) = -\frac{\alpha}{2}\log(\det(\id_H + 2\sqrt{u} Q_t \sqrt{u} ))$. 
Note that $\tilde{\phi}(0)=0$, and 
recalling the definition of $\psi$ from~\eqref{eq:def_psi:laplace} in Theorem \ref{thm:laplace_transform}, and using that $\frac{d}{dt}\det(A_t) =\tr(A_t'A_t^{-1}) \det(A_t)$ whenever $[0,T]\ni t\mapsto A_t \in S^+_1(H)$ is differentiable, we obtain
$
\frac{d}{dt}\tilde{\phi}(t)
=
-\alpha \tr( \psi(t,u)Q).
$
This implies 
\begin{equation}\label{eq:simplephi}
 \e^{-\alpha \int_0^t \tr( \psi(s,u)Q)ds} \,
 = \det(I_H + 2 \sqrt{u} Q_t \sqrt{u})^{-\frac{\alpha}{2}}. 
\end{equation}
\end{proof}

\begin{remark}\label{rem:laplace_transform_limitations}
Note that~\eqref{eq:simplephi} may not hold if we replace $\psi(s,u)$ by $\psi(s,u+iv)$ as the logarithm is not uniquely defined on $\C$, for details see~\cite{Mayerhofer:2019}.
\end{remark}

Recall from the proof of Theorem~\ref{thm:existence} that within the setting of this proposition, one can construct a process $X$ satisfying~\eqref{eq:solX_laplace} with $\alpha=n\in \N$ by considering $X=Y^*Y$, where $Y$ is an $L_2(H,\R^n)$-valued Ornstein-Uhlenbeck process. For this explicit construction of $X$ the Laplace transform can be calculated directly (i.e., without relying on Theorem~\ref{thm:laplace_transform}), see Proposition~\ref{thm:laplace_transform_OU} below (note that~\eqref{eq:laplace_transform_simple} indeed coincides with~\eqref{eq:Laplacetrafoexplicit} on the intersection of the respective parameter ranges). This proposition allows us to  conclude that $\E[ \e^{-\tr(BX_t)} ]<\infty$ provided \begin{equation} 
\inf\{\Re(\lambda)\colon \lambda\in \sigma(B)\} > \tfrac{1}{2} \| Q_t \|_{L(H)}^{-1}.
\end{equation} 
In particular, while Theorem~\ref{thm:laplace_transform} only provides $\E[ \e^{-\tr(BX_t)} ]$ when $\Re(\sigma(B))\geq 0$, Corollary~\ref{cor:laplace_general} below provides, in the setting of Theorem~\ref{thm:existence}, that $\E[ \e^{-\tr(BX_t)} ]$ for some $B$ for which $\Re(\sigma(B))\ngeq 0$. Note that Theorem~\ref{thm:laplace_transform} in not actually used for the proof of Proposition~\ref{thm:laplace_transform_OU} and Corollary~\ref{cor:laplace_general}.

\begin{proposition}\label{thm:laplace_transform_OU}
Assume the setting of Theorem~\ref{thm:existence} and let $X$ be the probabilistically and analytically weak solution to~\eqref{eq:wishart1} constructed in the proof of Theorem~\ref{thm:existence}, see page~\pageref{p:X=Y*Y}.
Let $(Q_t)_{t\geq 0}$ be as defined in~\eqref{eq:Qtdef}. Assume in addition that one of the following two conditions is satisfied:
\begin{enumerate}
\item $u\in S^+(H)$, 
\item $u\in S^{-}(H)$ and 
$\| u \|_{L(H)} \leq \frac{1}{2} \| Q_t \|_{L(H)}$.
\end{enumerate}
Then
$\id_H + 2 \sign(u) \sqrt{|u|}Q_t \sqrt{|u|}$ is invertible and 
  \begin{equation}
    \begin{split}\label{eq:Laplacetrafoexplicit}
&\mathbb{E}\left[\exp(
 - \tr( u X_t)
 )\right]\\
&\qquad=
\det(\id_H + 2 \sign(u) \sqrt{|u|} Q_t \sqrt{|u|})^{-\frac{n}{2}}
\\ 
& \qquad \qquad \times 
\exp\left( -\tr(   
  \e^{tA} \sqrt{|u|} (\id_H + 2 \sign(u)\sqrt{|u|} Q_t \sqrt{|u|})^{-1}\e^{tA} x_0)
 \right)
    \end{split}
  \end{equation}
where $\det(\id_H +  2\operatorname{sign}(u) \sqrt{|u|} Q_t \sqrt{|u|})$ is defined as in~\eqref{eq:def_det}.
\end{proposition}

\begin{proof}
See Appendix~\ref{app:Laplace-tarnsform}.
\end{proof}

\begin{corollary}\label{cor:laplace_general}
Assume the setting of Theorem~\ref{thm:existence} and let $X$ be the probabilistically and analytically weak solution to~\eqref{eq:wishart1} constructed in the proof of Theorem~\ref{thm:existence}, see page~\pageref{p:X=Y*Y}. Let $\omega \in \R\setminus\{0\}$ and $M\in (0,\infty)$ be such that $\|\e^{tA}\|_{L(H)}\leq M\e^{\omega t}$ for all $t \geq 0$. Then for all $t \geq 0$ and all $u, v\in S(H)$ satisfying
\begin{equation}\label{eq:laplace_ball}
(\| u \|_{L(H)}^2 + \| v \|_{L(H)}^2)^{\frac{1}{2}} \|Q \|_{L(H)} < \tfrac{|\omega|}{M^2| \e^{2\omega t} - 1|}
\end{equation}
we have that $\id_{H_{\C}}+2(u-iv)Q_s$, $s\in [0,t]$, is invertible on $H_{\C}$, with inverse 
\begin{equation}\label{eq:inverse_in_LT}
(\id_{H_{\C}}+2(u-iv)Q_s)^{-1}= \sum_{k=0}^{\infty} (-2(u-iv)Q_s)^k,
\end{equation}
and~\eqref{eq:Xisaffine} holds when $\alpha \in \N$ with  
\begin{align}\label{eq:def_psi:laplace_mixed}
   \psi(s,u-iv) &= \e^{sA}(\id_H+2(u-iv)Q_s)^{-1}(u-iv) \e^{sA^*},  \quad s\in [0,t].
\end{align}
\end{corollary}

\begin{proof}
Note that $\| u+iv\|_{L(H_{\C})} = (\| u \|_{L(H)}^2 + \| v \|_{L(H)}^2)^{\frac{1}{2}}$, whence~\eqref{eq:laplace_ball} implies
\begin{align*}
\| (u+iv) Q_s \|_{L(H_{\C})}
& \leq \| u+iv\|_{L(H_{\C})} \| Q \|_{L(H)}
M^2 \int_{0}^{s} \e^{2\omega r}\,dr
< \tfrac{| \e^{2\omega s} - 1|}{2 |\e^{2\omega t} - 1|} \leq  \tfrac{1}{2}
\end{align*} 
for all $s\in [0,t]$. This implies that $\id_{H_{\C}} + 2(u-iv)Q_s$ is invertible
with inverse~\eqref{eq:inverse_in_LT}.\par
Proposition~\ref{thm:laplace_transform_OU} implies that the Laplace transform of $X_t$ is given by~\eqref{eq:Laplacetrafoexplicit}, i.e., with~\eqref{eq:Xisaffine} with $\psi$ as in~\eqref{eq:def_psi:laplace}.
Observe that $\psi$ as defined in~\eqref{eq:def_psi:laplace_mixed} is analytic in its second argument and coincides with $\psi$ as defined in~\eqref{eq:def_psi:laplace} on the intersection of their domains (see also Lemma~\ref{lem:simple_inverse}). On the other hand, Proposition~\ref{thm:laplace_transform_OU} implies that the domain of the analytic function $(S(H)+iS(H))\ni u+iv\mapsto \E[ \exp( -\tr((u+iv)X_t) )]$ contains the set $\{u+iv\colon u,v\in S(H), \sigma(u)\subseteq (-\frac{1}{2}\| Q_t\|_{L(H)},\infty) \}$.
This and the uniqueness of the analytic expansion imply that~\eqref{eq:Xisaffine} holds with $\psi$ as in~\eqref{eq:def_psi:laplace_mixed}.
\end{proof}

\subsection{Proof of Theorem~\ref{thm:laplace_transform}}
The approach for proving Theorem~\ref{thm:laplace_transform}
is classical in the context of affine processes,
it involves applying the It\^o formula to 
\begin{equation*} s\mapsto \exp(-\tr(\psi(t-s,u-iv)X_s)-\phi(t-s,u-iv)),
\end{equation*}
where in our case $\phi(t,u-iv) = \alpha \int_0^t \tr( \psi(s, u -i v) Q) ds$.
This allows one to derive a Riccati-type differential equation for $\psi$ (and $\phi$ can subsequently be recovered from $\psi$). In Proposition~\ref{prop:riccati} below we establish the existence of a solution to the Riccati equation for several situations. \par 
Recall that $H_{\C}$, 
$C^T$ for $C \in L(H_{\mathbb{C}})$, and $S^+(H)\oplus iS(H)$ were introduced on p.~\pageref{p:Hcomplexification}. We also note that we write $B_n \stackrel{\text{s.o.t.}}{\rightarrow} B $ when $(B_n)_{n\in\N}$ is a sequence of bounded linear operators converging to $B$ in the strong operator topology.

\begin{proposition}\label{prop:riccati}
Let $H$ be a real Hilbert space, let $Q\in S^+(H)$, let $A\colon D(A)\subset H \rightarrow H$ be the generator of a $C_0$-semigroup, let 
\begin{equation}\label{eq:integrabilityasspropriccati}
\int_{0}^{t} \| \e^{sA} \sqrt{Q} \|_{L_2(H)}^2 \,ds < \infty
\end{equation}
for all $t\geq 0 $ and let $(Q_t)_{t\geq 0}$ be defined by~\eqref{eq:Qtdef}.
Let $B\in L(H_{\C})$, and consider the following ordinary differential equation in $L(H_{\C})$:
\begin{equation}\label{eq:Riccati}
\left\{
\begin{aligned}
\tfrac{\partial}{\partial t}\psi(t,B)
& =
-\tfrac{1}{2}(\psi(t,B)+\psi^T(t,B))Q(\psi(t,B)+\psi^T(t,B))
\\ & \quad + A\psi(t,B) + \psi(t,B)A^*,\quad t \in [0,\infty);
\\ \psi(0,B)&=B.
\end{aligned}
\right.
\end{equation}
We say that $\psi(\cdot,B)\in C^{1}([0,\infty),L(H_{\C}))$ is a solution to~\eqref{eq:Riccati} if $A\psi(t,B), \psi(t,B)A^*\in L(H_{\C})$ for all $t\in [0,\infty)$ and moreover~\eqref{eq:Riccati} holds.
The following holds:
\begin{enumerate}
 \item\label{it:riccati:unique} any solution to~\eqref{eq:Riccati} is unique;
 \item\label{it:riccati:psi1} if $B\in S^{+}(H)$ then $\id_H + 2 \sqrt{B}Q_t \sqrt{B}$ is invertible and the mapping $t\mapsto (\id_H + 2\sqrt{B}Q_t \sqrt{B})^{-1}$ is continuously differentiable with derivative
 \begin{equation*}
     t\mapsto -2 (\id_H + 2\sqrt{B}Q_t \sqrt{B})^{-1} \sqrt{B} \e^{tA^*} Q \e^{tA} \sqrt{B} (\id_H + 2\sqrt{B}Q_t \sqrt{B})^{-1}.
 \end{equation*}
 Moreover the process $\psi(\cdot,B)$ given by~\eqref{eq:def_psi:laplace}, i.e.,
 \begin{equation}\label{eq:def_psi1_reminder}
     \psi(t,B) = \e^{tA} \sqrt{B} (\id_{H} + 2\sqrt{B}Q_t \sqrt{B})^{-1}\sqrt{B} \e^{tA^*}, \quad t\geq 0,
 \end{equation}
 satisfies
 \begin{align}\label{eq:psi1bounds}
 \| \psi(t,B)\|_{L(H)} 
 &
 \leq \| B \|_{L(H)}\| \e^{tA} \|^2_{L(H)} ,
 &
 \| \psi(t,B) Q \|_{L_1(H)} 
 &
 \leq \| B \|_{L(H)} \| \e^{tA}\sqrt{Q} \|_{L_2(H)}^2,
 \end{align}
 and  $\psi(t,B)\in S^+(H)$ for all $t\geq 0$.
 If in addition $A\sqrt{B} \in L(H)$ then $\psi(\cdot,B)$
 is a solution to~\eqref{eq:Riccati}, and if moreover $\sqrt{B}, A\sqrt{B}, A^2\sqrt{B} \in L_1(H)$, then $\psi, A\psi, \psi A^*, A\psi A^*\in C^1([0,\infty), L_1(H))$.
 \item\label{it:riccati:psi2} if 
$iB\in S^{+}(H)\cup S^{-}(H)$
 then $\operatorname{sign}(iB)\id_{H_\C} + 2i\sqrt{|B|}Q_t \sqrt{|B|}$ is invertible for all $t\geq 0$ and the mapping $t\mapsto (\operatorname{sign}(iB)\id_{H_\C} + 2i\sqrt{|B|}Q_t \sqrt{|B|})^{-1}$ is continuously differentiable with derivative
 \begin{equation*}
 t\mapsto -2i (\operatorname{sign}(iB)\id_{H_\C} + 2i\sqrt{|B|}Q_t \sqrt{|B|})^{-1} \sqrt{|B|} \e^{tA^*} Q \e^{tA}\sqrt{|B|} (\operatorname{sign}(iB)\id_{H_\C} + 2i\sqrt{|B|}Q_t \sqrt{|B|})^{-1}.
 \end{equation*}
 Moreover the process $\psi(\cdot,B)$ given by~\eqref{eq:def_psi:fourier} with $v=iB$, i.e., 
 \begin{equation}\label{eq:def_psi2_reminder}
 \psi(t,B) = \e^{tA} \sqrt{|B|} \left(i\operatorname{sign}(iB)\id_{H_{\C}} + 2\sqrt{|B|}Q_t \sqrt{|B|}\right)^{-1}\sqrt{|B|} \e^{tA^*}, \quad t\geq 0,
 \end{equation} 
 satisfies~\eqref{eq:psi1bounds} (with $H$ replaced by $H_\C$ in the norms whenever needed)
 and  $\psi(t,B)\in S^+(H)\oplus iS(H)$ for all $t\geq 0$.
 If in addition $A\sqrt{|B|}\in L(H)$, then
 $\psi(\cdot,B)$ 
 is a solution to~\eqref{eq:Riccati} satisfying $\psi(t,B)\in S^+(H_{\C})\oplus iS(H_{\C})$ for all $t\geq 0$, and if moreover $\sqrt{|B|}, A\sqrt{|B|}, A^2\sqrt{|B|} \in L_1(H)$, then $\psi, A\psi, \psi A^*, A\psi A^*\in C^1([0,\infty), L_1(H_\C))$.
\item\label{it:riccati:psi3} if $B\in S^+(H) \oplus iS(H)$ and $A$, $B$, and $Q$ are jointly diagonizable, 
i.e., if there exists an orthonormal basis $(e_n)_{n \in \N}$ for $H$ and sequences $(a_n)_{n\in \N}$ in $\R$, $(b_n)_{n\in N}$ in $[0,\infty)\times i\R$, $(q_n)_{n\in \N} $ in $[0,\infty)$ such that $A = \sum_{n\in \N} a_n e_n \otimes e_n$, $B=\sum_{n\in \N} b_n e_n \otimes e_n$, and $Q=\sum_{n\in \N} q_n e_n \otimes e_n$,  then $\id_{H_\C} + 2 B Q_t $ is invertible for all $t\geq 0$, with 
\begin{equation*}
(\id_{H_\C} + 2B Q_t)^{-1}
=
\sum_{n\in \N}  
     \frac{a_n^2+ a_n b_n q_n(e^{2a_n t} -1)}
     {|a_n + \overline{b_n} q_n (e^{2a_n t} -1)|^2} ( e_n\otimes e_n), 
     \quad t \geq 0,
\end{equation*}
and $t\mapsto (\id_{H_\C} + 2 B Q_t)^{-1}$ is continuously differentiable with derivative 
\begin{equation*}
   t \mapsto  -2\e^{2tA} B (\id_{H_\C} + 2 B Q_t)^{-2}.
\end{equation*}
 Moreover the process $\psi(\cdot,B)$ given by~\eqref{eq:def_psi:laplace_mixed} with $-iv=B$, i.e.,
 \begin{equation}\label{eq:def_psi3_reminder}
 \begin{aligned}
     \psi(t,B) & = \e^{tA} (\id_{H_{\C}} + 2B Q_t )^{-1}B \e^{tA^*}\\
     & = \sum_{n\in \N}  
     \frac{b_n e^{2a_n t}\left(a_n^2+ a_n \overline{b_n} q_n(e^{2a_n t} -1)\right)}
     {|a_n + b_n q_n (e^{2a_n t} -1)|^2} ( e_n\otimes e_n), 
     \quad t \geq 0,
     \end{aligned}
 \end{equation}
 satisfies~\eqref{eq:psi1bounds} (with $H$ replaced by $H_\C$ in the norms whenever needed)
 and  $\psi(t,B)\in S^+(H)\oplus iS(H)$ for all $t\geq 0$.
If moreover $AB\in L(H_{\C})$, then $\psi(\cdot,B)$
 is a solution to~\eqref{eq:Riccati}. If moreover $B, AB, A^2B \in L_1(H)$, then $\psi$, $A\psi$, $\psi A^*, A\psi A^*\in C^1([0,\infty), L_1(H_\C))$.
\end{enumerate}
\end{proposition}

\begin{proof}
\begin{enumerate}
\item Uniqueness of a solution to~\eqref{eq:Riccati} is immediate from the fact that the coefficients of the differential equation are Lipschitz continuous on bounded subsets of $L(H_{\C})$.
\item Note that $\sqrt{B}Q_t \sqrt{B}\in S^+(H)$, whence it follows from e.g.~\cite[Theorem VII.1.4]{Werner:2000} that $\id_H + 2\sqrt{B}Q_t \sqrt{B}$ is invertible with $(\id_H + 2\sqrt{B}Q_t \sqrt{B})^{-1}\in S^+(H)$, and $\| (\id_H + 2\sqrt{B}Q_t \sqrt{B})^{-1}\|_{L(H)}\leq 1$. To verify continuous differentiability, note that for $s,t\geq 0$ we have 
\begin{align*}
&   (\id_H + 2\sqrt{B}Q_t \sqrt{B})^{-1}
    -    
    (\id_H + 2\sqrt{B}Q_s \sqrt{B})^{-1}
\\& \quad     =
    -2(\id_H + 2\sqrt{B}Q_t \sqrt{B})^{-1}
     \sqrt{B}(Q_t - Q_s) \sqrt{B}
    (\id_H + 2\sqrt{B}Q_s \sqrt{B})^{-1}
\end{align*}
whence differentiability follows from the differentiability of $Q_t$ and the uniform boundedness of the inverses. The bounds~\eqref{eq:psi1bounds} now follow from the definition of $\psi$, the fact that $\| \sqrt{B} \|_{L(H)} \leq \sqrt{\| B \|_{L(H)}}$ (see~\cite[Theorem VII.1.4]{Werner:2000}) and~\eqref{eq:SchattenHolder}. One also readily verifies $\psi(t,B)\in S^+(H)$ for all $t\geq 0$, and that if one has $A\sqrt{B} \in L(H)$ then $\psi(\cdot,B)$ is a solution to~\eqref{eq:Riccati} (note that $A\sqrt{B} \in L(H)$ implies $\sqrt{B}A^*\in L(H)$). The final assertion regarding the case that $A^2\sqrt{B}\in L_1(H)$ is also easily verified.
\item The proof of this statement is entirely analogous to the proof above (in particular,~\cite[Theorem VII.1.4]{Werner:2000} ensures existence and global boundedness of the inverse). However, to see that 
$(i\sign(iB) I_{H_\C} + 2\sqrt{|B|}Q_t \sqrt{|B|})^{-1} \in S^+(H)\oplus iS(H)$ it helps to observe that for $C\in S^+(H_{\C})$
we have $$(C + i\id_{H_{\C}} )^{-1} = (\id_{H_{\C}} + C^2)^{-1}(C-i\id_{H_{\C}}) \in S^+(H)\oplus iS(H).$$
\item This statement is easily verified by hand (note that in this case $A$, $B$, and $Q_t$ all commute). In order to obtain the bounds~\eqref{eq:psi1bounds} note that once again we have $\| (I_{H_{\C}} + 2 B Q_t)^{-1}\|_{L(H)} \leq 1$.
\end{enumerate}
\end{proof}

We will also need the following approximation lemma:

\begin{lemma}\label{lem:approx_riccati}
Assume the setting of Theorem~\ref{thm:laplace_transform}.
Let $(h_n)_{n\in \N}$ be an orthonormal basis for $H$ (and thus also for $H_{\C}$) satisfying\footnote{Note that $D(A^2)$ is a dense subspace of $H$ by e.g.~\cite[Prop.\ 1.8]{EngelNagel:2000} whence the desired orthonormal basis can be obtained by applying the Gram-Schmidt procedure to a countable subset of $D(A^2)$ that is dense in $H$.} $h_n\in D(A^2)$ for all $n\in \N$. Let $P_n\in L(H_{\C})$ denote the orthogonal projection onto $\operatorname{span}(\{h_1,\ldots,h_n\})$, $n\in \N$. 
\begin{enumerate}
    \item Let $u\in S^+(H)$ and set $u_n = (P_n \sqrt{u} P_n)^2$. Then $u_n \in S^+(H)$ and for $\psi$ as defined in~\eqref{eq:def_psi1_reminder} we have
    \begin{enumerate}[label = (\alph*),ref=(\alph*)]
        \item\label{it:psin_sat_assA} $\psi(\cdot,u_n) \in C^1([0,\infty),L_1(H_{\C}))$ 
%
solves~\eqref{eq:Riccati}  and satisfies $A\psi(\cdot,u_n)$, $\psi(\cdot,u_n) A^*$, $A\psi(\cdot,u_n) A^* \in C^1([0,\infty),L_1(H_{\C}))$ and $\psi(t,u_n)\in S^+(H) \oplus i S(H)$ 
 for all $t\geq 0$;
%
    \item\label{it:lhs_charfun_psin_conv} 
    $\lim_{n\rightarrow \infty} \E [\exp(-\tr(u_n X_t ) )| x_0 ]  =\E [\exp(-\tr(u X_t )) | x_0 ] $ a.s.\ for all $t\geq 0$, where $X$ is the process satisfying~\eqref{eq:solX_laplace};
        \item \label{it:rhs_charfun_psin_conv}for all $t\geq 0$ we have
        \begin{align*}
            &\lim_{n\rightarrow \infty} 
            \exp\left(-\tr(\psi(t,u_n)x_0) - \alpha \int_{0}^{t} \tr(\psi(s,u_n)Q) \,ds \right)
            \\
            &= \exp\left(-\tr(\psi(t,u)x_0) - \alpha \int_{0}^{t} \tr(\psi(s,u)Q) \,ds \right)
            \quad \text{a.s.}
        \end{align*}
    \end{enumerate}
    \item Let $v\in S^+(H)\cup S^-(H)$ and set $v_n = \sign(v)(P_n \sqrt{|v|} P_n)^2$. Then $v_n \in S^+(H)\cup S^-(H)$ and for $\psi(\cdot,-iv_n)$, $\psi(\cdot,-iv)$ as defined in~\eqref{eq:def_psi2_reminder} we have
    that~\ref{it:psin_sat_assA}--\ref{it:rhs_charfun_psin_conv} above hold (but with $u_n$ replaced by $-iv_n$ and $u$ replaced by $-iv$).
    \item Suppose $A$, $u$, $v$, and $Q$ are jointly diagonizable, i.e., there exists an orthonormal basis $(e_k)_{k \in \N}$ for $H$ and sequences $(a_k)_{k\in \N}$ in $\R$, $(u_k)_{k\in N}$ in $[0,\infty)$, $(v_k)_{k\in N}$ in $\R$, $(q_k)_{k\in \N} $ in $[0,\infty)$ such that $A = \sum_{k\in \N} a_k e_k \otimes e_k$, $u-iv=\sum_{k\in \N} (u_k -iv_k) e_k \otimes e_k$, and $Q=\sum_{k\in \N} q_k e_k \otimes e_k$. Define $B_n = \sum_{k=1}^n (u_k -iv_k) e_k \otimes e_k$. Then for $\psi(\cdot,B_n), \psi(\cdot,u-iv)$ as defined in~\eqref{eq:def_psi3_reminder} we have
    that~\ref{it:psin_sat_assA}--\ref{it:rhs_charfun_psin_conv} above hold (but with $u_n$ replaced by $B_n$ and $u$ replaced by $u-iv$). 
\end{enumerate}

\end{lemma}

\begin{proof}\begin{enumerate}
\item Note that $A^2u_n\in L_1(H)$ by construction (indeed, $u_n$ is of finite rank and $A^2 P_n \in L(H)$), whence~\ref{it:psin_sat_assA} follows from Proposition~\ref{prop:riccati}~\ref{it:riccati:psi1}. As $\| P_m \|_{L(H)} \leq 1$ and $P_m \stackrel{\text{s.o.t.}}{\rightarrow} \id_{H_{\C}}$ (`s.o.t.' is `strong operator topology'), we obtain $ \| u_n \|_{L(H)} \leq \| u \|_{L(H)}$ and $u_n \stackrel{\text{s.o.t.}}{\rightarrow} u$ by Lemma~\ref{lem:sot_traceconvergence}. As moreover 
$\tr(u_n X_t)\geq 0$ for all $n\in \N$ and all $t\geq 0$, assertion~\ref{it:lhs_charfun_psin_conv} follows from Lemma~\ref{lem:sot_traceconvergence} and the (conditional) dominated convergence theorem. Finally,~\ref{it:rhs_charfun_psin_conv} follows from Lemmas~\ref{lem:invconv_sot},~\ref{lem:sot_traceconvergence},~\eqref{eq:psi1bounds} in Proposition~\ref{prop:riccati}~\ref{it:riccati:psi1},  Assumption~\eqref{eq:ass_integrability2}, and the dominated convergence theorem.
\item The proof is analogous to the proof above, but we use Proposition~\ref{prop:riccati}~\ref{it:riccati:psi2} instead of Proposition~\ref{prop:riccati}~\ref{it:riccati:psi1}.
\item Note that $e_k\in \cap_{m\in \N}(D(A^m))$ so in particular $A^2 B_n \in L_1(H_{\C})$. Thus we can follow similar reasoning as above, but now using Proposition~\ref{prop:riccati}~\ref{it:riccati:psi3}.
\end{enumerate}
\end{proof}

\begin{proof}[Proof of Theorem~\ref{thm:laplace_transform}]
Fix $t>0$, $u\in S^+(H)$, and $v\in S(H)$ and assume one of the three cases~\ref{it:laplace_transform:laplace}--\ref{it:laplace_transform:laplace_fourierComm} is satisfied.  
First of all note that for $\psi$ as in~\eqref{eq:def_psi:laplace},~\eqref{eq:def_psi:fourier}, and~\eqref{eq:def_psi:diagonizable}
we have, by Proposition~\ref{prop:riccati}, that $\psi(t,u+iv)\in S^+(H)\oplus iS(H)$ for all $t\geq 0$ and we have the bounds \eqref{eq:psi1bounds}; this together with~\eqref{eq:ass_integrability2} ensures that~\eqref{eq:tr(psiQ)integrable} holds.\par 
Next, observe that in view of Lemma~\ref{lem:approx_riccati} it suffices to prove the remaining assertions of the theorem for initial values $u_n$ and $v_n$ as specified in Lemma~\ref{lem:approx_riccati}.
To simplify the notation we shall omit the dependence of $\psi$ on them, i.e., we write $\psi(t)$ instead of $\psi(t,u_n+iv_n)$. Note that by Lemma~\ref{lem:approx_riccati}, $\psi$ thus satisfies following condition.

\begin{itemize}
\item[\textbf{Condition A:}] The function $\psi$ lies in $ C^1([0,\infty),  S^+(H) \oplus i S(H))$ and satisfies~\eqref{eq:Riccati} such that $A\psi,$ $ \psi A^*,$ $ A\psi A^* \in C^1([0,\infty),L_1(H_{\C}))$. 
\end{itemize}

Let now $(h_k)_{k\in \N}$ be an orthonormal basis for $H$ (and thus also for $H_{\C}$) satisfying $h_k\in D(A^*)$  for all $k\in \N$. Moreover, let $\lambda \in \rho(A)$, with $\rho(A)$ being the resolvent set of $A$ (see~\cite[Theorem I.1.10]{EngelNagel:2000}) and let $R(\lambda, A):= (\lambda - A)^{-1}$ denote the resolvent of $A$. Let $P_m \in L(H_{\C})$ denote the orthogonal projection onto $\operatorname{span}(\{h_1,\ldots,h_m\})$ and define a further approximation $\psi_m\in C^1([0,t],L(H_{\C}))$ by 
\begin{equation}\label{eq:def_psim}
    \psi_m(s) = R(\lambda, A) P_m (\lambda - A) \psi(s) (\lambda  - A^*) P_m R(\lambda, A^*),
\end{equation} 
for all $s \geq 0$, $m\in \N$, i.e.,
\begin{align}\label{eq:simple_f}
    \psi_m(s)=\sum_{k,\ell=1}^{m} f_{k,\ell}(s) R(\lambda,A) h_k \otimes R(\lambda,A)  h_{\ell},\quad s\in [0,t],\, m\in \N,
\end{align}
where $f_{k,\ell}(s) = \langle (\lambda - A ) \psi(s) (\lambda-A^*) h_k, h_{\ell} \rangle_{H_\C} $.
Note that $f_{k,\ell}\in C^1([0,\infty),\C)$ due to the fact that $(\lambda - A )\psi (\lambda - A^*)\in C^1([0,\infty),L_1(H_\C))$.
To ease notation, we also introduce the processes $\phi,\phi_m \colon [0,t]\times \Omega \rightarrow \R$ ($m\in \N$) which are given by
\begin{equation}\label{eq:def_phi}
\phi_m(s) = \alpha \int_0^s \tr(\psi_m(s)Q)\,ds,\quad 
\phi(s) = \alpha \int_0^s \tr(\psi(s)Q)\,ds,
\quad s \in [0,t],\, m\in \N,
\end{equation}
and we introduce the processes $Z,Z^{m} \colon [0,t]\times \Omega \rightarrow \R$ ($m\in \N$) which are given by
\begin{equation}\label{eq:def_Z}
Z_s^{m} = \exp(-\tr (\psi_m(t-s) X_s) - \phi_m(t-s)), \quad 
Z_s = \exp(-\tr (\psi(t-s) X_s) - \phi(t-s)),
\end{equation}
$s \in [0,t]$, $m\in \N$.
The reason why we introduce $\psi_m$, $\phi_m$
 and $Z_m$ is that due to the presence of the unbounded operator $A$, we do not directly have access to the dynamics of the full process $X$ (as $X$ is in general not a semimartingale) and therefore we 
 cannot directly apply  the It\^o formula to $Z$. 
 Instead, we apply the It\^o formula to $Z^m_t$ and use that $\psi_m$ is of the form~\eqref{eq:simple_f}, allowing us to exploit~\eqref{eq:solX_laplace}, being a real-valued semimartingale. More specifically, 
let $(g_j)_{j\in \N}$ be an orthonormal basis\footnote{Note that $(g_j)_{j\in \N}$ can be obtained by applying a Gram-Schmidt procedure to $(R(\lambda,A)h_j)_{j\in \N}$.} for $H$ such that $(g_j)_{j=1}^{m}$ is an orthonormal basis for $\operatorname{span}(\{ R(\lambda,A)h_1,\ldots, R(\lambda, A)h_m\})$ for all $m\in \N$, in particular, $\langle R(\lambda, A)h_{\ell}, g_j \rangle_H =0$ whenever $j>\ell$. Then we obtain from~\eqref{eq:solX_laplace}, the fact that $\psi_m$ is of the form~\eqref{eq:simple_f}, the It\^o formula, and~\eqref{eq:def_phi}
that 
\begin{equation}\label{eq:simpleIto1}
\begin{aligned}
&Z_{s}^{m} 
= 
\exp(-\tr (\psi_m(t-s) X_{s}) -\phi_m(t-s))
\\ &
 =
\exp\left(-\sum_{j,k,\ell=1}^{m} f_{k,\ell}(t-s) \langle g_j, X_{s} R(\lambda,A) h_k \rangle_{H_\C} \langle R(\lambda,A) h_{\ell}, g_j \rangle_H-\phi_m(t-s) \right)
\\ & 
=
Z_0^{m} 
+ 
\int_{0}^{s}
Z_r^{m}
\tr \Big((\psi'_m(t-r) - A\psi_m(t-r) - \psi_m(t-r)A^*) X_r \Big)
\,dr
\\
&\quad  - \int_{0}^{s}
   Z_r^{m} \tr\left(\psi_m(t-r) \left( \sqrt{X_r} \,dW_r \sqrt{Q} + \sqrt{Q} \,dW_r^* \sqrt{X_r}\right)\right)
   \\
&\quad + \tfrac{1}{2} \sum_{i,j\in \N}
\int_{0}^{s} Z_r^{m} \left(\tr\left( \psi_m(t-r)\left(\sqrt{X_r}(h_i \otimes h_j) \sqrt{Q} + \sqrt{Q}(h_j \otimes h_i) \sqrt{X_r}\right)\right)\right)^2\,dr
\end{aligned}
\end{equation}
$\P$-a.s.\ for all $s\in [0,t]$.

Note that $\langle C^T h, g \rangle_{H_{\C}} = \langle C g , h \rangle_{H_{\C}}$ whenever $\Im(g)=\Im(h)=0$ and $C\in L(H_{\C})$. Therefore, noting that $\Im(h_k)=0$ for all $k\in \N$, we have
\begin{equation}\label{eq:simple_HSnorm}
\begin{aligned}
\sum_{i,j\in \N} (\tr(C (h_i\otimes h_j)))^2 
& = \sum_{i,j\in \N} 
(\langle C h_i, h_j \rangle_{H_{\C}})^2
= \sum_{i,j\in \N} 
\langle C h_i, h_j \rangle_{H_{\C}}
\langle C^{T} h_j, h_i \rangle_{H_{\C}}
\\
& = \sum_{i,j\in \N}      
\langle C h_i, h_j \rangle_{H_{\C}}
\langle h_j, (C^T)^*h_i \rangle_{H_{\C}}
\\
& = \sum_{i\in \N}
\langle Ch_i, (C^T)^*h_i \rangle_{H_{\C}}
= \tr(C^TC)
\end{aligned}
\end{equation}
for all $C\in L_2(H_{\C})$.
Also note that $\tr(C+D)=\tr{C}+\tr{D}$, $(CD)^{T}=D^TC^T$, $\tr(C)=\tr(C^T)$, $\tr(CD)=\tr(DC)$ , and $(C^T)^T = C$ for all $C,D\in L_1(H_{\C})$, whence
\begin{equation}\label{eq:rewrite_2ndorderterm}
\begin{aligned}
& \sum_{i,j\in \N}
\left(\tr\left( \psi_m(t-r)\left(\sqrt{X_r}(h_i \otimes h_j) \sqrt{Q} + \sqrt{Q}(h_j \otimes h_i) \sqrt{X_r}\right)\right)\right)^2
\\ 
& =\sum_{i,j\in \N}
\left(\tr\left( \sqrt{Q}(\psi_m(t-r)+\psi_m^T(t-r))\sqrt{X_r}(h_i \otimes h_j)\right)\right)^2
\\ 
& = \tr\left(
(\psi_m(t-r)+\psi_m^T(t-r)) Q (\psi_m(t-r)+\psi_m^T(t-r))X_r
\right),
\end{aligned}
\end{equation}
where we used~\eqref{eq:simple_HSnorm} for the second equality and the fact that $X_r,\psi_m(r)\in L_1(H_{\C})$ to ensure all quantities are well-defined. Note that $\psi(r)\in S^+(H)\oplus i S(H)$
implies that $\psi_m(r)\in S^+(H)\oplus i S(H)$, in particular $\psi_m(r)^T = \psi_m(r)$. In conclusion,~\eqref{eq:simpleIto1} reduces to
\begin{equation}\label{eq:simpleIto2}
\begin{aligned}
Z_{s}^{m} 
&=
Z_0^{m}
    + \int_0^{s} Z_r^{m} \tr\left(
     (\psi'_m(t-r) - A\psi_m(t-r) - \psi_m(t-r)A^*)X_r\right) \,dr 
\\ & \quad      
     + 2\int_0^{s} Z_r^{m} \tr\left(\psi_m(t-r) Q \psi_m(t-r) X_r \right)
\,dr 
\\
&\quad  - \int_{0}^{s}
   Z_r^{m} \tr\left(\psi_m(t-r) \left( \sqrt{X_r} \,dW_r \sqrt{Q} + \sqrt{Q} \,dW_r^* \sqrt{X_r}\right)\right)
\end{aligned}
\end{equation}
$\P$-a.s. for all $s\in [0,t]$.\par 

Next we wish to take $m\rightarrow \infty$ in~\eqref{eq:simpleIto2}. 
Lemma~\ref{lem:approxPsi}, and the fact that $\psi, A\psi, \psi A^*, A\psi A^* \in C^1([0,\infty),L_1(H_{\C}))$ (see Condition A) we have 
\begin{equation}\label{eq:psimbounded}
\sup_{m\in \N} (\| \psi_m \|_{C^1([0,t],L_1(H_{\C}))} + \| A \psi_m\|_{C^1([0,t],L(H_{\C}))} + \|  \psi_m A^* \|_{C^1([0,t],L(H_{\C}))}) < \infty,
\end{equation}
and 
\begin{equation}\label{eq:phimbounded}
\sup_{m\in \N} \| \phi_m \|_{C([0,t],\C)} 
\leq 
\sup_{m\in \N} \alpha \int_{0}^{t} |\tr(\psi_m(s) Q) | \,ds 
< \infty,
\end{equation}
as well as 
$\phi_m(s)\rightarrow \phi(s)$,
$\psi_m(s) \stackrel{\textnormal{s.o.t.}}{\rightarrow} \psi(s) $, $A\psi_m(s) \stackrel{\textnormal{s.o.t.}}{\rightarrow} A\psi(s)$, $\psi_m(s) A^* \stackrel{\textnormal{s.o.t.}}{\rightarrow} A^*$, $\psi'_m(s) \stackrel{\textnormal{s.o.t.}}{\rightarrow} \psi'(s)$ for all $s\in [0,t]$, $m\in \N$. This together with  Lemma~\ref{lem:sot_traceconvergence} and the fact that $\psi$ solves~\eqref{eq:Riccati} imply
\begin{equation}\label{eq:detint_to0}
\begin{aligned}
\lim_{m\rightarrow \infty}\left| \tr\left(
     (\psi'_m(r) - A\psi_m(r) - \psi_m(r) A^* + 2 \psi_m(r) Q \psi_m(r))X_{t-r}\right) \right| = 0
\end{aligned}
\end{equation}
for all $r\in [0,t]$. We want to stress at this point that the construction of $\psi_m$ involving resolvents, i.e., essentially involving an orthonormal projection in $D(A)$ instead of an orthonormal projection in $H$, is necessary  to ensure~\eqref{eq:detint_to0} -- indeed, only by projecting in $D(A)$ do we gain control over $A \psi_m$.
Moreover,~\eqref{eq:psimbounded} implies that 
\begin{equation}\label{eq:Zmbounded}
    \sup_{m\in \N} \sup_{r\in [0,t]} | Z_r^m |
    \leq \sup_{m\in \N} \sup_{r\in [0,t]}  \e^{-\phi_m(r)} < \infty,
\end{equation}
and Lemma~\ref{lem:sot_traceconvergence} implies that  $\lim_{m\rightarrow \infty} Z_r^m = Z_r$ a.s.\ for all $r\in [0,t]$. By arguments similar to~\eqref{eq:rewrite_2ndorderterm}, we can calculate the quadratic variation of $Z^{m}$:
\begin{equation}\label{eq:est_StochInt}
\begin{aligned}
\langle Z^m \rangle_s &= 
\left\langle
\int_{0}^{\cdot}
   Z_r^{m} \tr\left(\psi_m(t-r) \left( \sqrt{X_r} \,dW_r \sqrt{Q} + \sqrt{Q} \,dW_r^* \sqrt{X_r}\right)\right)
\right\rangle_s
\\
& \qquad = 
\sum_{i,j\in \N}
\int_{0}^{s}
   \left|
    Z_r^{m}
    \tr\left(\psi_m(t-r) \left( \sqrt{X_r} (h_i\otimes h_j) \sqrt{Q} + \sqrt{Q} (h_j \otimes h_i) \sqrt{X_r}\right)\right)
   \right|^2
\,dr
\\ & \qquad 
= 4\int_{0}^{s}
    |Z_r^m|^2 \tr(\psi_m(t-r)Q\psi_m(t-r)X_r)
\,dr,
\end{aligned}
\end{equation}
for all $s\in [0,t]$.
From the above observations we can conclude that 
\begin{equation}
\langle Z^m \rangle_s \rightarrow 
    4\int_{0}^{s}
    |Z_r|^2 \tr(\psi(t-r)Q\psi(t-r)X_r)
\,dr\quad \text{a.s.}
\end{equation}
for all $s\in [0,t]$.
Combining~\eqref{eq:simpleIto2} with the bounds~\eqref{eq:psimbounded} and~\eqref{eq:phimbounded} and the convergence results above we arrive at 
\begin{equation}\label{eq:simpleIto3}
\begin{aligned}
Z_{t}
&=
Z_0 - \int_{0}^{t}
   Z_r^{m} \tr\left(\psi_m(t-r) \left( \sqrt{X_r} \,dW_r \sqrt{Q} + \sqrt{Q} \,dW_r^* \sqrt{X_r}\right)\right) \quad\text{a.s.}
\end{aligned}
\end{equation}
It follows from~\eqref{eq:def_phi},~\eqref{eq:def_Z}, and the fact that $\psi(0)=u_n+iv_n$ (recall that we work in the realm of Lemma \ref{lem:approx_riccati}) that 
\begin{align*}
    Z_t & = \exp(-\tr((u_n-iv_n)X_t ))
    & \text{and} &
    & Z_0 
    & 
    = \exp\left(-\tr(\psi(t)x_0)-\alpha \int_0^t \tr(\psi(s)Q)\,ds \right).
\end{align*}
Thus in particular $|Z_t| \leq 1$, implying that the stochastic integral on the right-hand side of~\eqref{eq:simpleIto3} has finite moments conditional on $x_0$ and thus its expectation conditioned on $x_0$ is $0$.
This implies that by taking the conditional expectation in~\eqref{eq:simpleIto3} we get
\begin{equation}
\begin{aligned}
&\E \left[ 
\exp(-\tr((u_n-iv_n)X_{t}))  
| x_0 
\right] = \E \left[ Z_t | x_0 \right] \\ 
& \quad = \E[ Z_0 | x_0 ]  = \exp\left(
        -\tr(\psi(t)x_0)
        - \alpha \int_0^t \tr(\psi(s)Q) \,ds\right).
\end{aligned}
\end{equation}
Invoking Lemma \ref{lem:approx_riccati} completes the proof of Theorem~\ref{thm:laplace_transform}. 
\end{proof}

\subsection{Examples of the Fourier and Laplace transform}\label{subs:examples_laplacetrans}

In this section we provide some concrete calculations of the Fourier and Laplace transform of a Wishart process.

First of all, by taking $u=0$ and $v=r\id_{H}$ ($r\in \R$) in Theorem~\ref{thm:laplace_transform}~\ref{it:laplace_transform:fourier} we obtain the 
characteristic function of $\|X_t\|_{L_1(H)} = \tr(X_t)$. By a slight abuse of notation we write $\psi(t,r)$ in Corollary~\ref{cor:charfunc_trX} below instead of $\psi(t, - i r \id_H)$ which is used in Theorem~\ref{thm:laplace_transform}~\ref{it:laplace_transform:fourier}.

\begin{corollary}\label{cor:charfunc_trX}
Assume the setting of Proposition~\ref{thm:laplace_transform}.
Then \begin{equation}\label{eq:charfunc_trX}
\E \left[ \exp( ir\| X_t\|_{L_1(H)}) \,|\, x_0\right]
=
\E \left[ \exp( ir\tr(X_t)) \,|\, x_0\right]
=
\exp( - \tr(\psi(t, r)x_0) - \phi(t, r))
\end{equation}
for all $r\in \R$, where 
\begin{align}
 \psi(t,r)
 &
 = r\e^{tA}
    \left(
         i \id_{H_{\C}} + 2rQ_t 
    \right)^{-1}
    \e^{tA^*},
\label{eq:psidef_trX} \\  
\phi(t,r) 
 &
 = \alpha \int_{0}^t \tr(\psi(s,r)Q)\,ds,\label{eq:phidef_trX}
\end{align}
for all $r\in \R$, $t\geq 0$, and $(Q_t)_{t\geq 0}$ is defined by~\eqref{eq:Qtdef}.
\end{corollary}


\begin{example}\label{ex:laplace_transform_easy}
Assume the setting and notation of Theorem~\ref{thm:laplace_transform} with $\alpha =n\in \N$ and suppose moreover that $x_0 = \sum_{j=1}^{n} x_j h_j\otimes h_j$ and $u=\sum_{j=1}^{m} u_j g_j \otimes g_j$ for some $m\in \N$, $x_1,\ldots,x_n, u_1,\ldots,u_m\in [0,\infty)$, and some orthonormal systems $(h_j)_{j=1}^{n}$ and $(g_j)_{j=1}^{m}$ in $H$. Fix $t\geq 0$ and define the matrix $\hat{Q}_t\in \R^{m\times m}$ by
\begin{align*}
 (\hat{Q}_t)_{j,k} & = \sqrt{u_j u_k}
 \int_0^{t} \langle \sqrt{Q} \e^{sA} g_j, \sqrt{Q} \e^{sA} g_k \rangle_H \,ds,&\quad j,k\in \{1,\ldots,m\}.
\end{align*}
Then 
$\det(\id_H + 2 \sqrt{u} Q_t \sqrt{u})
 = \det(\id_{\R^m}+2\hat{Q}_t)$
and 
\begin{align*} 
&
\| \sqrt{x_0} \e^{tA} \sqrt{u} (I_{H}+2\sqrt{u}Q_t \sqrt{u})^{-\frac{1}{2}} \|_{L_2(H)}^2
\\ & \qquad 
= \sum_{j=1}^{m}
    \langle 
        \sqrt{u} (\id_H +2\sqrt{u} Q_t \sqrt{u} )^{-1} g_j,
        \e^{tA^*} x_0 \e^{tA} \sqrt{u} g_j 
    \rangle_H
\\ & \qquad 
= \sum_{j,k=1}^{m}
    \sqrt{u_j}
    \langle 
        \sqrt{u} (\id_H +2\sqrt{u} Q_t \sqrt{u} )^{-1} g_j, g_k 
    \rangle_H
    \langle 
        \e^{tA^*} x_0 \e^{tA} g_j, g_k 
    \rangle_H
\\ & \qquad
= \sum_{j,k=1}^{m}
    \sqrt{u_j u_k}
    \left(
        (\id_{\R^m} + 2\hat{Q}_t)^{-1}
    \right)_{j,k}
    \sum_{i=1}^{n} x_i \langle h_i, \e^{tA} g_j\rangle_H \langle h_i,  \e^{tA} g_k \rangle_{H}.
\end{align*}
In conclusion, from~\eqref{eq:Laplacetrafoexplicit} we obtain
\begin{equation}
\begin{aligned}\label{eq:laplace_transform_example}
& \E[ \exp(-\tr(X_t u )]
 =
(\det(\id_{\R^m}+2\hat{Q}_t))^{-\frac{n}{2}}
\\ & \quad 
\times 
\exp\left( 
\sum_{j,k=1}^{m}
    \sum_{i=1}^{n}
    x_i
    \sqrt{u_j u_k}
    \left(
        (\id_{\R^m} + 2\hat{Q}_t)^{-1}
    \right)_{j,k}
    \langle h_i, \e^{tA} g_j\rangle_H 
    \langle h_i,  \e^{tA} g_k \rangle_{H}
\right).
\end{aligned}
\end{equation}
Similarly, for $\psi(t) = \e^{tA}\sqrt{u}(i \id_H + 2\sqrt{u} Q_t \sqrt{u})^{-1}\sqrt{u} e^{tA^*}$, $t\geq 0$, we have  
\begin{align*}
& \tr\left(\psi(t) x_0 \right) 
\\
& \quad = 
\sum_{j,k=1}^{m}
    \sum_{i=1}^{n}
    x_i
    \sqrt{u_j u_k}
    \left(
        (i\id_{\R^m} + 2\hat{Q}_t)^{-1}
    \right)_{j,k}
    \langle h_i, \e^{tA} g_j\rangle_H 
    \langle h_i,  \e^{tA} g_k \rangle_{H}
\end{align*}
and 
\begin{align*}
\int_{0}^t \tr\left( \psi(s) Q \right) \,ds
=
\sum_{j,k=1}^{m} 
    \left(
        (i\id_{\R^m} + 2\hat{Q}_t)^{-1}
    \right)_{j,k} (\hat{Q}_t)_{j,k}
    \Big),
\end{align*}
whence we obtain from Theorem~\ref{thm:laplace_transform}~\ref{it:laplace_transform:fourier} that
\begin{align*}
& \E[ \exp(i\tr(X_t u )]
 =
\exp\Big(
     - n \sum_{j,k=1}^{m} 
    \left(
        (i\id_{\R^m} + 2\hat{Q}_t)^{-1}
    \right)_{j,k} (\hat{Q}_t)_{j,k}
    \Big)
\\ & \quad 
\times 
\exp\Big(
- \sum_{j,k=1}^{m}
    \sum_{i=1}^{n}
    x_i
    \sqrt{u_j u_k}
    \left(
        (i\id_{\R^m} + 2\hat{Q}_t)^{-1}
    \right)_{j,k}
    \langle h_i, \e^{tA} g_j\rangle_H 
    \langle h_i,  \e^{tA} g_k \rangle_{H}
\Big).
\end{align*}
\end{example}

\begin{example}\label{ex:laplace_transform_easiest}
We consider the following example (in the setting of Proposition~\ref{thm:laplace_transform}): $H=L^2(0,1)$, $\alpha=n\in \N$,
$A = \sum_{j\in \N} a_j h_j\otimes h_j$ 
with $a_j = - j^2\pi^2$ and $h_j(x)= \sin(j\pi x)$, $x\in (0,1)$, and $Q=\id_{L^2}$, i.e., $A$ is the Dirichlet Laplacian (see also Remark~\ref{rem:integrability}). In addition, we let $h\in H$ satisfy $\| h \|_H = 1$ and set $u= h\otimes h$. We then see that $\sqrt{u} Q_t \sqrt{u}$ (with $Q_t$ as in~\eqref{eq:Qtdef}) is given by 
\begin{equation*}
 \sqrt{u} Q_t \sqrt{u} = \int_0^{t} \left\| \e^{sA} h \right\|_{H}^{2}\,ds\,  h \otimes h.
\end{equation*}
For notational ease, we introduce 
$ q_t :=  \int_0^{t} \left\|\e^{sA} h \right\|_{H}^{2} \,ds $
so that we obtain the following from~\eqref{eq:Laplacetrafoexplicit}
\begin{align*}
 \E \left[ 
   \exp( - \langle X_t h, h \rangle_H )
 \right]
 & =
 (1+2q_t)^{-\frac{n}{2}} \exp\left( 
    \tfrac{1}{1+2q_t} \| \sqrt{x_0} \e^{tA} h \|_{H}^2
 \right).
\end{align*}
\end{example}

\section{Characterisation of infinite-dimensional Wishart processes}\label{sec:characterization}

In this section we establish necessary conditions for the existence of infinite-dimensional Wishart processes, i.e., necessary conditions for the existence of a solution to~\eqref{eq:Wishart-formal}. 
As mentioned in the introduction, we consider two settings: firstly, we prove that if $Q$ is injective and if moreover $\e^{tA}$ is injective for some $t>0$, then the existence of a Wishart process implies that $\alpha \in \N$, see Theorem~\ref{thm:char_infdimwishart} below. The proof of this result is based on the Laplace transform calculated in the previous section and the finite-dimensional characterisation of Wishart distributions~\cite{graczyk2018characterization,letac2018laplace,M:19}. 

Secondly, we establish (without any assumptions on the semigroup $(\e^{tA})_{t\geq 0}$) that if $X$ is a Wishart process, then either $\operatorname{rank}(X_t)\geq \operatorname{rank}(Q)$ a.s.\ for almost all $t> 0$, or $\alpha \in \N$ with $\alpha < \operatorname{rank}(Q)$ and $\operatorname{rank}(X_t) = \alpha$ a.s.\ for almost all $t>0$, see Corollary~\ref{cor:alpha_is_n}. In particular, 
if $\operatorname{rank}(Q) = \infty$, then a finite-rank Wishart process exists if and only if $\alpha \in \N$, see Corollary~\ref{cor:charWishartsimple} below.

\subsection{A characterisation of Wishart processes when \texorpdfstring{$Q$ and $\e^{tA}$ are}{the covariance operator is} injective}\label{ssec:char_laplace}

Recall the following characterisation of non-central Wishart distributions, see~\cite[Theorem 1.3]{graczyk2018characterization},~\cite{letac2018laplace}, or~\cite[Theorem 1.1]{M:19}
(see also Lemma~\ref{lem:simple_inverse}):

\begin{theorem}\label{thm:char_wishart}
Let $n\in \N$, $\alpha\in [0,\infty)$, and $b\in S^+(\R^n)$, and $Q\in S^{++}(\R^n)$. Then there exists a probability measure $\mu$ on $S^+(\R^n)$ satisfying 
\begin{align}\label{eq:wishart_laplace}
\int_{S^+(\R^n)} \exp(-\tr(u\xi)) \,\mu(d\xi)
&=
\det(\id_{\R^n}+2Qu)^{-\frac{\alpha}{2}}\exp(-bu(\id_{\R^n}+2Qu)^{-1}), 
\end{align}
for all $u \in S^+(\mathbb{R}^n)$ if and only if one of the following two conditions is satisfied:
\begin{enumerate}
    \item $\alpha \in \{0,1,\ldots,n-2\}$ and $\operatorname{rank}(b)\leq \alpha$,
    \item $\alpha \geq  n-1$.
\end{enumerate}
\end{theorem}

The following is a direct consequence of Corollary~\ref{cor:laplace_transform_simple} and Lemma~\ref{lem:simple_inverse}:

\begin{corollary}\label{cor:laplace_transform_findim}
Assume the setting of Theorem~\ref{thm:laplace_transform} and let $(h_i)_{i=1}^{n}$ be an orthonormal system in $H$.
Define $P_n\colon \R^n \rightarrow H$ by $P_n(x)=\sum_{i=1}^{n}x_ih_i.$ Then $Y=P_n^* X P_n$ is an $S^+(\R^n)$-valued stochastic process with a Laplace transform given by
\begin{align}
& \E[\exp(-\tr(uY_t) | x_0 ] \notag
\\
&= \notag
\det(I_{\R^n} + 2 \sqrt{u} P_n^* Q_t P_n \sqrt{u})^{-\frac{\alpha}{2}}
\\ & \label{eq:laplace_transform_findim1}
\quad \times 
\exp\left(-\tr\left(P_n^*\e^{tA^*} x_0 \e^{tA}P_n \sqrt{u}(I_{\R^n} + 2 \sqrt{u} P_n^* Q_t P_n \sqrt{u})^{-1} \sqrt{u}  \right)\right)
\\  \notag
& = \det(I_{\R^n} + 2  P_n^* Q_t P_n u)^{-\frac{\alpha}{2}}
\\ & \label{eq:laplace_transform_findim2}
\quad \times 
\exp\left(-\tr\left(P_n^*\e^{tA^*} x_0 \e^{tA}P_n u (I_{\R^n} + 2 P_n^* Q_t P_n u)^{-1} \right)\right)
\end{align}
for all $u\in S^+(\R^n)$, where $Q_t$ is defined in \eqref{eq:Qtdef}. 
\end{corollary}

\begin{proof}
The fact that $X$ is $S^+_1(H)$-valued implies that $P_n^* X P_n$ is $S^+(\R^n)$ valued. Moreover, note that~\eqref{eq:laplace_transform_findim1} follows from the formula for $\E[\exp(-\tr(uP_n^* X_t P_n))] $ in Corollary~\ref{cor:laplace_transform_simple},
noting that $\tr(ABC)=\tr(CBA)$ for operators $A,B$, and $C$ with appropriate domains and co-domains. 
In addition, Lemma~\ref{lem:simple_inverse} implies that~\eqref{eq:laplace_transform_findim2} follows from~\eqref{eq:laplace_transform_findim1}.
\end{proof}

Combining Theorem~\ref{thm:char_wishart} and Corollary~\ref{cor:laplace_transform_findim} we obtain the following:

\begin{theorem}\label{thm:char_infdimwishart}
Let $H$ be a separable real Hilbert space, let $A\colon D(A)\subset H \rightarrow H$ be the generator of a $C_0$-semigroup, 
let $Q\in S^+(H)$, let $\alpha \in \R$, let $(\Omega,\cF,\P,(\cF_t)_{t\geq 0})$ be a filtered probability space rich enough to allow for an $L_2(H)$-cylindrical $(\cF_t)_{t\geq 0}$-Brownian motion, let $x_0\in L^1((\Omega,\cF_0,\P),S_1^+(H))$,
and assume 
\begin{equation}\label{eq:ass_integrability3}
\int_{0}^{t}\| \e^{sA} \sqrt{Q}\|_{L_2(H)}^2\,ds < \infty
\end{equation} 
for all $t\geq 0$. Assume moreover that $Q$ is injective and that there exists a $\tau>0$ such that $\e^{\tau A}$ is injective. 
Then the following are equivalent:
\begin{enumerate}
\item\label{it:char:alpha_in_N} $\alpha\in \N$ and $\operatorname{rank}(x_0)\leq \alpha$ a.s.
\item\label{it:char:sol_exits} there exists an $L_2(H)$-cylindrical $(\cF_t)_{t\geq 0}$-Brownian motion $W$ and an adapted stochastic process $X\colon [0,\infty)\times \Omega \rightarrow S^+_1(H)$ with continuous sample paths satisfying
\begin{equation}\label{eq:solX_char}
\begin{aligned}
 \langle X_t g, h\rangle_{H} 
 & = 
 \langle x_0 g, h \rangle_{H}
 + 
 \int_0^{t} 
 (\alpha \langle Q g, h \rangle_H 
 + \langle X_s A g, h \rangle_{H} 
 + \langle X_s g, Ah \rangle_H ) \,ds
 \\ & \quad 
 + \int_{0}^{t} \langle \sqrt{X}_s \,dW_s \sqrt{Q}g, h\rangle_H + \int_0^t \langle \sqrt{Q} \,dW^*_s \sqrt{X}_s g, h \rangle_H
\end{aligned}
\end{equation}
for all $t\geq 0$ and all $h,g\in D(A)$. 
\end{enumerate}
\end{theorem}
\begin{proof}
The implication~\ref{it:char:alpha_in_N} $\rightarrow$ \ref{it:char:sol_exits} follows from Theorem~\ref{thm:existence}. For the reverse implication note that by assumptions on $Q$ and $(\e^{tA})_{t\geq 0}$ we can conclude that $Q_{\tau}=\int_0^{\tau} e^{s A^*} Q e^{sA} \, ds$ as defined in \eqref{eq:Qtdef} lies in $S^{++}_1(H)$: indeed, $Q_t\in S_1^+(H)$ by~\eqref{eq:ass_integrability3}, moreover, $\e^{sA}$ is injective for all $s\in [0,\tau]$ due to the semigroup property. So for $h\in H\setminus\{0\}$
we have $ \e^{sA} h \neq \{0\}$ for all $s\in [0,\tau]$ and as $Q$ is also injective  $\langle Q_{\tau} h,h\rangle_{H} = \int_{0}^{\tau} \langle Q \e^{sA} h, \e^{sA} h \rangle_{H} \,ds >0$, so indeed $Q_{\tau}\in S^{++}_1(H)$. \par 
Let $(h_i)_{i\in \N}$ be an orthonormal basis for $H$ and for $n\in \N$ let $P_n\colon \R^n\rightarrow H$ be defined by $P_n(x)=\sum_{i=1}^{n} x_i h_i$. Note that $P_n^* Q_{\tau} P_n \in S^{++}(\R^n)$ for all $n\in \N$. Considering now the process $P_n^* X P_n$, it follows from Corollary~\ref{cor:laplace_transform_findim} with $t=\tau$ and Theorem~\ref{thm:char_wishart} that $\alpha \in \N$ (because one can always pick $n>\alpha+2$) and moreover that $\operatorname{rank}(P_n^* \e^{\tau A^*}x_0 \e^{\tau A} P_n)\leq \alpha$ a.s.\ for all $n\geq \alpha+2$. Note that if $h\notin \operatorname{ker}(\e^{\tau A^*}x_0 \e^{\tau A})$, then there exists an $N_h\in \N$ such that $P_n^* \e^{\tau A^*}x_0 \e^{\tau A} P_n P_n^*h \neq 0$ for all $n\geq N_h$.
In particular, if $\operatorname{rank}(P_n^* \e^{\tau A^*}x_0 \e^{\tau A} P_n)\leq \alpha$ for all $n\geq \alpha+2$, then $\operatorname{rank}(\e^{\tau A^*}x_0 \e^{\tau A}) \leq \alpha$. Due to the fact that $\e^{\tau A}$ is injective, we can conclude that $\operatorname{rank}(x_0)\leq \alpha$.
\end{proof}

\begin{remark}\label{rem:semigroup_injective}
The assumption that $\e^{\tau A}$ is injective for some $\tau>0$ is not harmless: e.g.\ the shift semigroup on $L^2(0,1)$ does not satisfy this property. On the other hand, by the spectral mapping theorem 
(see e.g.~\cite[Theorem 10.55]{Neerven:2022}) if $A$ is self-adjoint then $\e^{tA}$ is injective for all $t\geq 0$. Moreover, if $A$ is bounded, then $t\mapsto \e^{tA}$ is strongly continuous and thus $\e^{tA}$ is injective for sufficiently small $t>0$.
\end{remark}

\begin{remark}
We cannot conclude from~\eqref{cor:laplace_transform_findim} that $P_n^* XP_n$ is an affine process in the sense of~\cite[Definition 2.1]{CFMT:11}, because in general $P_n^* X P_n$ is not a Markov process. However, if $A\equiv 0$, then there exists an $n$-dimensional affine Markov process $Y$ such that $Y_t$ and $P_n^* X_t P_n$ are identical in law for all $t\geq 0$, indeed, $Y$ is precisely the affine process with admissible parameter set $(P_n^* Q P_n, \alpha P_n^* Q P_n, 0,0, 0,0,0)$
in the notation of~\cite[Theorem 2.4]{CFMT:11}. This observation follows from~\ref{eq:laplace_transform_findim2} and~\cite[Theorem 2.4]{CFMT:11}.
\end{remark}

\begin{remark}\label{injectiverankn}
Assume the setting of Theorem \ref{thm:char_infdimwishart} and suppose moreover that $\alpha \in \N$. By Corollary \ref{cor:wishart_uniqueinlaw}
we have that 
any process satisfying \eqref{eq:solX_char} coincides in law with the Wishart process constructed in Theorem \ref{thm:existence} and is thus of rank at most $\alpha \in \mathbb{N}$. In fact, from from Corollary \ref{cor:rankn} below we get that $\operatorname{rank}(X_t)=\alpha$ $\mathbb{P}$-a.s. for almost all $t >0$.
\end{remark}

\subsection{A characterisation of finite rank Wishart processes without parameter restrictions}\label{ssec:char_finrank}

The following proposition is inspired by~\cite[Proposition 4.18]{CFMT:11}:

\begin{proposition}\label{prop:alpha_geq_n}
Let $\alpha \in \R$, $m\in \N$, let $H$ be a separable real Hilbert space, let $(\Omega,\cF,\P,(\cF_t)_{t\geq 0})$ be a filtered probability space, let $A\colon D(A)\subset H \rightarrow H$ be the generator of a $C_0$-semigroup,
and let $Q\in S^+(H)$ satisfy $\operatorname{rank}(Q)>m$. Assume moreover that there exists an $L_2(H)$-cylindrical $(\cF_t)_{t\geq 0}$-Brownian motion $W$ and an adapted stochastic process $X\colon [0,\infty)\times \Omega \rightarrow S^+(H)$ with continuous sample paths satisfying
$\P(\operatorname{rank}(X_0)=m)>0$ and
\begin{equation}\label{eq:solX_alpha}
\begin{aligned}
 \langle X_t g, h\rangle_{H} 
 & = 
 \langle x_0 g, h \rangle_{H}
 + 
 \int_0^{t} 
 (\alpha \langle Q g, h \rangle_H 
 + \langle X_s A g, h \rangle_{H} 
 + \langle X_s g, Ah \rangle_H ) \,ds
 \\ & \quad 
 + \int_{0}^{t} \langle \sqrt{X}_s \,dW_s \sqrt{Q}g, h\rangle_H + \int_0^t \langle \sqrt{Q} \,dW^*_s \sqrt{X}_s g, h \rangle_H
\end{aligned}
\end{equation}
for all $t\geq 0$ and all $h,g\in D(A)$. 
Then $\alpha \geq m$. 
If moreover 
\begin{equation}\label{eq:Xstaysofrankm}
\P(\{\operatorname{rank}(X_0)=m\} \cap \{\inf\{ t\geq 0 \colon \operatorname{rank}(X_t) > m\}\} )>0
\end{equation}
then $\alpha = m$. 
\end{proposition}

The proof of this proposition involves applying It\^o's formula to the determinant of the $\R^{(m +1)\times (m+1)}$-valued process $P_{m+1}^* X_t P_{m+1}$, where $P_{m+1} \colon \Omega \rightarrow L(\R^{m+1},H)$ is such that $X_0|_{\{ \operatorname{rank}(X_0) = m\}} = P_{m+1}^* D P_{m+1}$ for some positive diagonal operator $D$ and $\operatorname{ker}(P_{m+1}^*X_0 P_{m+1}) \setminus \operatorname{ker}(P_{m+1}^* Q P_{m+1})\neq \emptyset$.
As such, the approach is similar to the proof of~\cite[Proposition 4.18]{CFMT:11}. However, the details of the proof are more technical for two reasons: a delicate approximation argument is needed when the  eigenvectors of $X_0$ are not in the domain of $A$, and we need to deal with a non-deterministic initial value. For the readers' convenience the proof is given in Appendix~\ref{app:proof_Prop_alpha_geq_n}.

The final assertion 
of Theorem~\ref{thm:char_infdimwishart} implies that if there exists an $m < \operatorname{rank}(Q)$ such that $\operatorname{rank}(X_u(\omega)) =m$ with positive probability for all $u$ in some non-trivial interval $[s,t]$, then necessarily $\alpha=m$. This is formalised in the following corollary.

\begin{corollary}\label{cor:alpha_is_n}
Let $\alpha \in \R$ and let $H$ be a separable real Hilbert space, let $(\Omega,\cF,\P,(\cF_t)_{t\geq 0})$ be a filtered probability space, let $A\colon D(A)\subset H \rightarrow H$ be the generator of a $C_0$-semigroup, let $Q\in S^+(H)$ and assume there exists an $L_2(H)$-cylindrical $(\cF_t)_{t\geq 0}$-Brownian motion $W$ and an adapted stochastic process $X\colon [0,\infty)\times \Omega \rightarrow S^+(H)$ with continuous sample paths satisfying
\begin{equation}\label{eq:solX_alpha_cor}
\begin{aligned}
 \langle X_t g, h\rangle_{H} 
 & = 
 \langle x_0 g, h \rangle_{H}
 + 
 \int_0^{t} 
 (\alpha \langle Q g, h \rangle_H 
 + \langle X_s A g, h \rangle_{H} 
 + \langle X_s g, Ah \rangle_H ) \,ds
 \\ & \quad 
 + \int_{0}^{t} \langle \sqrt{X}_s \,dW_s \sqrt{Q}g, h\rangle_H + \int_0^t \langle \sqrt{Q} \,dW^*_s \sqrt{X}_s g, h \rangle_H
\end{aligned}
\end{equation}
for all $t\geq 0$ and all $h,g\in D(A)$.
Then either $\operatorname{rank}(X_t)\geq \operatorname{rank}(Q)$ a.s.\ for almost all $t> 0$, or $\alpha \in \N \cap [0,\operatorname{rank}(Q))$ and $\operatorname{rank}(X_t) = \alpha$ a.s.\ for almost all $t>0$.
\end{corollary}

\begin{proof}
Suppose that we do not have $\operatorname{rank}(X_t)\geq \operatorname{rank}(Q)$ a.s.\ for almost all $t> 0$. Then there exists an $m\in \N\cap [0,\operatorname{rank}(Q))$ 
and a $t>0$ such that
\begin{equation*}
\lambda\otimes \P (\{(s,\omega)\in [0,t]\times \Omega\colon \operatorname{rank}(X_s(\omega)) = m \} ) \neq 0. 
\end{equation*}
As $(X_t)_{t\geq 0}$ is predictable and $\operatorname{rank}(\cdot)$ is lower semi-continuous, this ensures that there exist $0\leq r < s \leq t$ and $B\in \cF_{r}$ such that $\operatorname{rank}(X_u(\omega)) = m $ for all $(u,\omega)\in [r,s]\times B$. Note that~\eqref{eq:solX_alpha_cor} implies
\begin{equation}\label{eq:solX_alpha_cor2}
\begin{aligned}
 \langle X_u g, h\rangle_{H} 
 & = 
 \langle X_r g, h \rangle_{H}
 + 
 \int_r^{u} 
 (\alpha \langle Q g, h \rangle_H 
 + \langle X_{v} A g, h \rangle_{H} 
 + \langle X_v g, Ah \rangle_H ) \,dv
 \\ & \quad 
 + \int_{r}^{u} \langle \sqrt{X}_v \,dW_v \sqrt{Q}g, h\rangle_H + \int_r^u \langle \sqrt{Q} \,dW^*_v \sqrt{X}_v g, h \rangle_H
\end{aligned}
\end{equation}
for all $u\geq 0$ and all $h,g\in D(A)$. Clearly, if $m=0$ then $X\equiv 0$,
so necessarily $\alpha =0$. If $m>0$, we apply Proposition~\ref{prop:alpha_geq_n} with $X_0 = X_r$ and $\cF_t = \cF_{t+r}$ to conclude that $\alpha =m$, leading to a contradiction.
\end{proof}

\begin{corollary}\label{cor:rankn}
Assume the setting of Theorem~\ref{thm:existence}, 
in particular, let $X$ be the process satisfying~\eqref{eq:solX} for all $g,h\in D(A)$.
Assume moreover that $\operatorname{rank}(Q)>n$.
Then $\operatorname{rank}(X_t)=n$ $\P$-a.s.\ for almost all $t> 0$.
\end{corollary}

\begin{proof}
We know from the construction of the solution of $X$, see the proof of Theorem~\ref{thm:existence}, that $X$ is of rank at most $n$. The result thus follows directly from Corollary~\ref{cor:alpha_is_n}.
\end{proof}

\begin{remark}
An immediate consequence of Remark~\ref{rem:OU_alternative} and Corollary~\ref{cor:rankn} is that $n$ \emph{stochastically independent} $H$-valued Ornstein-Uhlenbeck processes $Y^{(1)},\ldots,Y^{(n)}$ as defined in~\ref{eq:OUprocess} are \emph{linearly independent} $\P$-a.s.\ for almost all $t\geq 0$, provided $\operatorname{rank}(Q)>n$. Indeed, one would expect this to be the case even for $\operatorname{rank}(Q)=n$, but this clearly requires a different proof.
\end{remark}

Combining the above with Theorem~\ref{thm:existence} we arrive at the following characterisation of finite rank Wishart processes:

\begin{corollary}\label{cor:charWishartsimple}
Let $\alpha \in \R$, let $H$ be a separable real Hilbert space, let $(\Omega,\cF,\P,(\cF)_{t\geq 0})$ be a filtered probability space, let $A\colon D(A)\subset H \rightarrow H$ be the generator of a $C_0$-semigroup, let $Q\in S^+(H)$ satisfy $\int_0^{t} \| \e^{sA}\sqrt{Q} \|_{L_2(H)}^2 \,ds < \infty$ for all $t\geq 0$, and let $X_0\colon \Omega \rightarrow S_1^+(H)$ be $\cF_0$-measurable.
Then the following are equivalent:
\begin{enumerate}
 \item\label{it:char_solexists} There exists an $L_2(H)$-cylindrical $(\cF_t)_{t\geq 0}$ Brownian motion and an adapted stochastic process $X\colon [0,\infty)\times \Omega \rightarrow S^+(H)$ with continuous sample paths satisfying $\operatorname{rank}(X_t) < \operatorname{rank}(Q)$ a.s.\ for almost all $t\geq 0$
 and 
\begin{equation*}
\begin{aligned}
 \langle X_t g, h\rangle_{H} 
 & = 
 \langle x_0 g, h \rangle_{H}
 + 
 \int_0^{t} 
 (\alpha \langle Q g, h \rangle_H 
 + \langle X_s A g, h \rangle_{H} 
 + \langle X_s g, Ah \rangle_H ) \,ds
 \\ & \quad 
 + \int_{0}^{t} \langle \sqrt{X}_s \,dW_s \sqrt{Q}g, h\rangle_H + \int_0^t \langle \sqrt{Q} \,dW^*_s \sqrt{X}_s g, h \rangle_H
\end{aligned}
\end{equation*}
for all $t\geq 0$ and all $h,g\in D(A)$. 
\item \label{it:char_alpha_is_n} $\alpha\in \N\cap [0,\operatorname{rank}(Q))$.
\end{enumerate}
Moreover, if these equivalent conditions hold then $\operatorname{rank}(X_t)=\alpha$ $\P$-a.s.\ for almost all $t\in [0,\infty)$.
\end{corollary}

\begin{proof}
The implication~\ref{it:char_alpha_is_n}$\Rightarrow$\ref{it:char_solexists} follows from Theorem~\ref{thm:existence}.
Assume~\ref{it:char_solexists}, then there exists an $m\in \N_0\cap [0,\operatorname{rank}(Q))$ such that $\lambda\otimes \P(\{ (s,\omega)\in [0,1]\times \Omega \colon \operatorname{rank}(X_s(\omega))=m \})>0$. Thus \ref{it:char_alpha_is_n} follows from Corollary~\ref{cor:alpha_is_n}. The final statement follows from Corollary~\ref{cor:rankn}.
\end{proof}

\begin{remark}\label{rem:alpha_is_n}We believe that Corollary~\ref{cor:charWishartsimple} even provides new insights for the finite-dimensional setting, indeed, it generalises the characterisation provided by \cite[Theorem 3.10]{graczyk2018characterization} to include a linear drift and a matrix $Q$ in the dynamics of the Wishart process.
\end{remark}


\section{Infinite dimensional Wishart processes are Feller}\label{sec:Feller}
In this section we study the Feller property of $S_1^+(H)$-valued Wishart processes. For 
that we extend the space of compact operators $K(H)$ to include the identity $\id_H$ (represented by an extra dimension) and prove that the cone of positive trace class operators endowed with the relative weak-$*$-topology in this extended space is a locally compact Polish space. 

\subsection{The cone \texorpdfstring{$S^+_1(H)\times \R_+$}{} is a locally compact Polish space}

Let $H$ be a separable Hilbert space. Recall that the dual of $K(H)$ can be identified with $L_1(H)$ 
under the pairing $\langle A, B \rangle_{L_1(H),K(H)} = \sum_{k\in \N} \langle A e_k, B e_k \rangle$, where $(e_k)_{k\in \N}$ is an orthonormal basis\footnote{$\langle \cdot, \cdot \rangle_{L_1(H),K(H)}$ is independent of the choice of the ONB $(e_k)_{k\in \N}$.} of $H$ (and the dual of $L_1(H)$ can be identified with $L(H)$ by the same paring). Hence we have that the dual of $K(H) \times \R$ can be identified with $S_1(H) \times \R$
under the canonical pairing 
\begin{align}\label{eq:pairing1} 
\langle (A,x), (B,y) \rangle_{L_1(H) \times \R,K(H)\times \R} = \langle A, B \rangle_{L_1(H),K(H)} + xy.
\end{align}
For notational simplicity we introduce another pairing 
$\llangle \cdot , \cdot \rrangle: (L_1(H) \times \mathbb{R}) \times (K(H) \times \mathbb{R}) \to \mathbb{R}$ defined by 
\begin{align}\label{eq:def_alternative_pairing}
 \llangle (A,x), (B,y) \rrangle  
&= \langle A, B \rangle_{L_1(H), K(H) } + (\langle A, \operatorname{Id} \rangle_{L_1(H), L(H) } +x) y\,.
   \end{align}
To see that every linear functional on $K(H) \times \mathbb{R}$ can be expressed via this pairing, we compare it with the canonical one given by \eqref{eq:pairing1}. Indeed, note that 
\[
L\colon L_1(H) \times \mathbb{R} \to L_1(H) \times \mathbb{R},  \quad  L(A,x)=(A, x + \langle A, \operatorname{Id} \rangle_{L_1(H), K(H) })
\]
is an isomorphism, and that for every $(B,y) \in K(H) \times \mathbb{R}$, we have
\[
\llangle (A,x), (B,y) \rrangle=\langle  L(A,x), (B,y) \rangle_{L_1(H) \times \mathbb{R}, K(H) \times \mathbb{R}}.
\]

%

\begin{proposition}\label{prop:posFinRank_LCCB}
Let $H$ be a separable Hilbert space, let $K:=S^+_1(H)\times \R_+ $ and let $\tau_{\textnormal{w}^*}$ be the relative weak-$*$ topology on $K$. Then:
\begin{enumerate}
 \item $K$ is a convex cone in $S_1(H)\times \R$;
 \item\label{it:K_charac} for all $(A,x)\in S_1(H)\times \R$ it holds that 
 $(A,x)\in K$ if and only if $\llangle (A,x), (B,y) \rrangle \geq 0$ for all $(B,y)\in S_c(H)\times [0,\infty)$ satisfying $B+ y \Id\geq 0$;
 \item $K$ is closed in $S_1(H)\times \R$ with respect to the weak-$*$-topology;
 \item $(K,\tau_{\textnormal{w}^*})$ is $\sigma$-compact and locally compact;
 \item \label{it:metrizable} $(K,\tau_{\textnormal{w}^*})$ is separable and metrizable.
\end{enumerate}
\end{proposition}

\begin{proof}
Let $(h_n)_{n\in \N}$ be an orthonormal basis for $H$.
Recall that for $A\in S_1(H)$ we have $A\geq 0$ if and only if $\langle Ah,h\rangle_H \geq 0$ for all $h\in H$, which in turn holds if and only if $\langle A, B\rangle_{L_1(H),K(H)} \geq 0$ for all $B\in S_c^+(H)$. 
\begin{enumerate}
 \item Note that $K\cap -K = \{0\}$. Moreover, for all $(A_1,x_1), (A_2,x_2)\in  K$ and all $\alpha, \beta \geq 0$ one has $\alpha (A_1, x_1) + \beta (A_2, x_2)\in K$, whence $K$ is indeed a convex cone. 
 \item If $(A,x)\in K$, then clearly $\llangle (A,x), (B,y) \rrangle \geq 0$ for all $(B,y)\in S_c(H)\times [0,\infty)$ satisfying $B+ y \Id\geq 0$.
 To prove the reverse, suppose $\llangle (A,x), (B,y) \rrangle \geq 0$ for all $(B,y)\in S_c(H)\times [0,\infty)$ satisfying $B+ y \Id\geq 0$.
 Then in particular $\llangle (A,x), (B,0) \rrangle \geq 0$ for all $B\in S_c^+(H)$, so $A\in S^+_1(H)$. Moreover, note that there exists a positive sequence $a=(a_n)_{n\in \N} \in \ell_1$ and an orthonormal basis $(e_n)_{n\in \N}$ of $H$ such that $A = \sum_{n\in \N} a_n e_n \otimes e_n$. Setting $B_N := \sum_{n=1}^N e_n \otimes e_n$,
 we see that $-B_N+\Id\geq 0$ and thus $0\leq \llangle (A,x), (-B_N,1) \rrangle = \sum_{n=N+1}^{\infty} a_n + x$. Letting $N\rightarrow \infty$ we see that $x\geq 0$.
 \item This follows directly from~\ref{it:K_charac}: indeed, let $I$ be some directed set and $(A_n,x_n)_{n\in I}$ be weak-$*$-Cauchy net in $K$. Denote by $(A,x)\in L_1(H)\times \R$ its weak-$*$-limit. Then
 \begin{equation}\label{eq:nets}
0\leq \lim_n \llangle (A_n,x_n), (B,y) \rrangle 
 = 
\llangle (A,x), (B,y) \rrangle
\end{equation}
for all $(B,y)\in S_c(H)\times [0,\infty)$ satisfying $B+ y \Id\geq 0$.
\item Define sets $U_M\subseteq K$, $M\in [0,\infty)$, by setting
\begin{align*}
 U_M &:= \{ (B,y) \in K \colon \llangle (B,y), (0,1) \rrangle \leq M \}
 =
 \{ (B,y) \in K \colon \| B \|_{L_1(H)} + |y| \leq M \}
 \\
 & = 
 \{ (B,y) \in L_1(H)\times \R \colon \| B \|_{L_1(H)} + | y | \leq M \} \cap K.
\end{align*}
The Banach-Alaoglu theorem implies that $U_M$ is compact with respect to $\tau_{\textnormal{w}^*}$. As $\cup_{M\in \N} U_M = K$ we see that $(K,\tau_{\textnormal{w}^*})$ is $\sigma$-compact. Moreover, $U_M$ is a neighbourhood of $(A,x)$ for all $(A,x)\in K$ satisfying $\| A \|_{L_1(H)} + |x|< M$. Thus, $(K,\tau_{\textnormal{w}^*})$ is locally compact.


\item As $S_1(H)$ is in fact norm-separable (since $H$ is assumed to be separable), the separability of $(K,\tau_{\textnormal{w}^*})$ is immediate. \par 
Let $(B_k,y_k)_{k\in \N}$ be (norm)-dense in $K(H)\times \R$ 
and assume moreover that $(B_1,y_1)=(0,1)$. We claim that for a net $(A_n,x_n)_{n\in I}$ to converge to $(A,x)$ in $(K,\tau_{\textnormal{w}^*})$ it is necessary and sufficient that
 \begin{align}
\forall k\in \N\colon  
&& \lim_{n} \llangle  (A_n,x_n), (B_k, y_k)  \rrangle & =  \llangle (A,x),   (B_k, y_k)  \rrangle\,.
 \label{eq:char_weak_conv}
\end{align} 
Necessity is obvious, to show sufficiency observe that
if~\eqref{eq:char_weak_conv} holds, then recalling that $(B_1,y_1)=(0,1)$ we obtain
  \[  
\lim_n \| A_n \|_{L_1(H)} + |x_n|
=
\lim_n \llangle (\langle A_n,x_n), (0,1) \rrangle 
= \| A \|_{L_1(H)} + |x|.
 \]
 Thus $(A_n,x_n)_{n\in I}$ is bounded. 
Therefore we have
 \begin{align*}
 | 
    \llangle 
        (A_n-A,x_n-x), (B,y)
   \rrangle
 |
& \leq  
  | 
    \llangle 
         (A_n-A,x_n-x) , (B-B_k,y-y_k) 
    \rrangle
    |
\\ & \quad 
+
  | 
    \llangle 
        (A_n -A ,x-x_n) , (B_k,y_k)
    \rrangle
    |
\\ & \quad 
    +
  | 
    \llangle 
         (A_n-A,x_n-x) , (B_k-B,y_k-y)
    \rrangle
 |\\
 & \leq 
 4( 2\sup_{n\in I} \| A_n\|_{L_1(H)} +\sup_{n\in I} |x_n|) \| (B-B_k, y-y_k) \|_{L(H)\times \R}
 \\ & \quad 
 +
  | 
    \langle 
         (A_n -A ,x-x_n) , (B_k,y_k)
    \rangle_{L_1(H)\times \R,K(H)\times \R}
    |
 \end{align*}
for every $(B,y)\in K(H)\times \R$ and every $k\in \N$.
As $(B_k,y_k)_{k\in \N}$ is dense in $K(H)\times \R$ it now follows from~\eqref{eq:char_weak_conv} that 
$
 \lim_{n}
 | 
    \llangle 
         (A_n-A,x_n-x), (B,y)
    \rrangle
 |
 \rightarrow 0
$.\par 
Now define $d\colon K\times K \rightarrow [0,\infty)$ by
\begin{align*}
d((A_1,x_1),(A_2,x_2)) 
& = 
\sum_{k\in \N} 2^{-k}
    |
        \llangle (A_1-A_2, x_1-x_2), (B_k,y_k) \rrangle
    |
        \wedge 
        1\,.
\end{align*}
It is easily verified that $d$ is a metric, and the above implies that $d$ generates $\tau_{\textnormal{w}^*}$.\par 
\end{enumerate}
\end{proof}

\subsection{The Feller property}

The fact that the cone $K=S_1^+(H) \times \mathbb{R}_+$ 
equipped with $\tau_{\textnormal{w}^*}$ is locally compact, separable, and metrizable (see Proposition~\ref{prop:posFinRank_LCCB} above) 
allows us to establish that the Wishart process $X$ satisfying
\eqref{eq:solX_laplace} is a Feller process as defined in~\cite[Definition III.2.1]{RevuzYor:1999}.

For the Feller property of $X$, we shall consider the corresponding semigroup $(P_t)_{t \geq 0}$ acting on $C_0(K)$, the space of $\tau_{\textnormal{w}^*}$-continuous functions vanishing at infinity, i.e.,
$$P_tf(x,z):=\mathbb{E}_{x}[f(X_t,z)]$$
for all $f \in C_0(K)$, $x=X_0$ and $(x,z) \in K$ 
and show that it is a Feller semigroup.
Indeed, local compactness of $K$ allows for a one-point compactification: we let $\Delta\notin K$ denote the point at infinity, and when writing $(x_k,z_k) \to \Delta$ we thus mean that $\| x_k\|_{L_1(H)} \to \infty$ and/or $|z_k| \to \infty$.

\begin{theorem}\label{thm:feller}
Let $H$ be a separable real Hilbert space, let $(\Omega,\cF,\P,(\cF_t)_{t\geq 0})$ be a filtered probability space, let $A\colon D(A)\subset H \rightarrow H$ be the generator of a $C_0$-semigroup $(\e^{tA})_{t\geq 0}$, 
let $Q\in S^+(H)$, let $\alpha \in \R$, and assume $
\int_{0}^{t}\| \e^{sA} \sqrt{Q}\|_{L_2(H)}^2\,ds < \infty$ and $\e^{tA}$ is injective for all $t\geq 0$.\par 
Let $(K,\tau_{\textnormal{w}^*})$ be as in Proposition~\ref{prop:posFinRank_LCCB} and let $K_0\subset K$ be a $\tau_{\textnormal{w}^*}$-closed subset such that for every $(x,z)\in K_0$ there exists an $L_2(H)$-cylindrical $(\cF_t)_{t\geq 0}$-Brownian motion $W$ and an adapted stochastic process $X\colon [0,\infty)\times \Omega \rightarrow S^+_1(H)$ with continuous sample paths satisfying
\begin{equation}
 \P( (X_t,z)\in K_0 ) = 1 
 \quad 
 \textnormal{for all } t \geq 0
\end{equation}
and 
\begin{equation}\label{eq:solX_feller}
\begin{aligned}
 \langle X_t g, h\rangle_{H} 
 & = 
 \langle x_0 g, h \rangle_{H}
 + 
 \int_0^{t} 
 (\alpha \langle Q g, h \rangle_H 
 + \langle X_s A g, h \rangle_{H} 
 + \langle X_s g, Ah \rangle_H ) \,ds
 \\ & \quad 
 + \int_{0}^{t} \langle \sqrt{X}_s \,dW_s \sqrt{Q}g, h\rangle_H + \int_0^t \langle \sqrt{Q} \,dW^*_s \sqrt{X}_s g, h \rangle_H
\end{aligned}
\end{equation}
for all $t\geq 0$ and all $h,g\in D(A)$.
Let $\tau_{\textnormal{w}^*}^0$ be the relative $\tau_{\textnormal{w}^*}$-topology on $K_0$. Then the mapping $P_t$ defined by 
\begin{equation}
    P_t f(x,z) := \E_{x}[f(X_t,z)],
    \quad f\in C_0((K_0,\tau_{\textnormal{w}^*}^0),\R),\, (x,z)\in K_0,
\end{equation}
is a Feller semigroup on $ C_0((K_0,\tau_{\textnormal{w}^*}^0),\R)$.
\end{theorem}

\begin{proof}

Proposition~\ref{prop:posFinRank_LCCB} implies that $(K,\tau_{\textnormal{w}^*})$ is locally compact with countable base, so by~\cite[Proposition III.2.4]{RevuzYor:1999} it suffices to verify that for all $f \in C_0 (K_0)$ we have
\begin{align}
&\lim_{t \downarrow 0} P_t f (x,z) = f(x,z),\\
&(x,z) \mapsto P_tf(x,z) \in C_0 (K_0), \quad 
 (x,z) \in  K_0,\,  t \in \mathbb{R}_+. \label{eq:Feller}
\end{align}
The first property follows from
dominated convergence and continuity of the trajectories (with respect to the norm topology as shown in Theorem~\ref{thm:existence} and thus also with respect to $\tau_{\textnormal{w}^*}$). Concerning the second property 
we follow the arguments of the proof of \cite[Proposition 3.4]{CFMT:11} adapted to the current setting.
Indeed, we first note that as $\|P_t f\|_{L^{\infty}(K)} \leq \|f\|_{L^{\infty}(K)}$, it suffices to verify~\eqref{eq:Feller} for $f$ in a dense subset of $C_0(K_0)$. By a locally compact version of the Stone-Weierstrass theorem (see,
e.g., \cite[Corollary 8.3]{Conway:1990}), the linear span of the set 
\[
A(K_0)
:= 
\{\exp(-\llangle  (\cdot, \cdot) , (V,w) \rrangle)  \, | \, V \in S_c^{++}(H),\, w \in \mathbb{R}_{++} \},
\]
is dense in $C_0(K_0)$. Recall that $S_c^{++}(H)$ denotes strictly 
positive definite compact operators and $\llangle \cdot , \cdot \rrangle$ is defined by~\eqref{eq:def_alternative_pairing}. Indeed, 
$(x,z) \mapsto \exp(-\llangle (V,w), (x,z) \rrangle)$ is $\tau_{\textnormal{w}^*}$-continuous due to the continuity of the pairing and the exponential function.  Moreover, it vanishes at $\Delta$ as $V \in S_c^{++}(H)$ and $w \in \mathbb{R}_{++}$. Hence   $A(K_0)$
is a subalgebra of $C_0 (K_0)$.
Moreover, $A(K_0)$ clearly 
separates points and $f>0$ for all $f\in A(K_0)$.\par 
From Corollary ~\ref{cor:laplace_transform_simple}, we know the form of 
$
P_t\exp(-\llangle (\cdot,\cdot), (V,w) \rrangle),
$
namely for all $(x,z)\in K_0$ we have
\begin{align*}
&P_t\exp(-\llangle  (\cdot,\cdot), (V,w) \rrangle)(x,z)
 = 
 \e^{-zw} \E_{x} \left[
    \exp(-\tr((V + w \id_H) X_t )
 \right]
\\
& \quad =
 \e^{-zw} \det(\id_H + 2 \sqrt{V+wI_H} Q_t \sqrt{V+wI_H})^{-\frac{\alpha}{2}}\\
& \quad \quad \times \exp\left( - \tr\left( 
 \e^{tA} \sqrt{V+wI_H} (\id_H + 2 \sqrt{V+wI_H}Q_t\sqrt{V+wI_H})^{-1}\sqrt{V+w I_H} \e^{tA^*}
  x \right)\right),
\end{align*}
where $Q_t$ is defined by~\eqref{eq:Qtdef}.\par 
Note now that for all $V \in S_c^{++}(H)$ and $w \in \mathbb{R}_{++}$, we have
\[
\e^{tA} \sqrt{V+wI_H} (\id_H + 2 \sqrt{V+wI_H}Q_t\sqrt{V+wI_H})^{-1}\sqrt{V+w I_H} \e^{tA^*} \in S^{++}(H).
\]
This follows from the fact that
 $V  \in S_c^{++}(H)$,  $\e^{tA}$  is injective for every $t >0$ by assumption and  that $g x g^* \in S^{++}(H)$ for $x \in  S^{++}(H)$  and injective operators $g \in L(H)$. Hence, 
 $$P_t\exp(-\llangle (\cdot,\cdot),(V,w) \rrangle) (x,z)\to 0$$ 
 as  $(x,z) \to \Delta$ implying that $P_t\exp(-\llangle  (\cdot,\cdot) ,(V,w)\rrangle) \in C_0(K_0)$.

\end{proof}

By taking $K_0$ in Theorem~\ref{thm:feller} to be the set
\begin{align}\label{eq:Kn}
K_n:= \{ (A,x)\in S^+_1(H)\times  \R_+
\colon \operatorname{rank}(A) \leq n \}, 
\end{align}
we immediately obtain the following corollary:

\begin{corollary}
Assume the setting of Theorem~\ref{thm:existence}, in particular, let $X$ satisfy~\ref{eq:solX}. Assume that $\e^{tA}$ is injective for all $t\geq 0$, 
and let $\tau_{\textnormal{w}^*}^n$ be the $\tau_{\textnormal{w}^*}$-relative topology on $K_n$ defined in \eqref{eq:Kn}. Then $X$ is Feller with respect to $C_0((K_n,\tau_{\textnormal{w}^*}^n),\R)$.
\end{corollary}

\begin{remark}
 The above corollary with $K_n$ as state space applies of course also to the setting when  additionally to the injectivity of $\e^{tA}$ for all $t \geq 0$, $Q$ is required to be injective. Indeed, by Remark~\ref{injectiverankn} and with $\alpha=n$, we can in this case even choose the smaller state space
 \[
\widetilde{K}_n:= \{ (A,x)\in S^+_1(H)\times  \R_+
\colon \operatorname{rank}(A) = n \}, 
\] 
if $\operatorname{rank}(x_0)=n$.
\end{remark}

\appendix
\section{Some simple lemmas}

\begin{lemma}\label{lem:simple_inverse}
Let $A,B\in S^+(\R^n)$. Then $\id_{\R^n}+BA$ and $\id_{\R_n}+\sqrt{A}B\sqrt{A}$ are invertible,
\begin{align}\label{eq:simple_det}
\det(\id_{\R^n}+BA) = \det(\id_{\R_n}+\sqrt{A}B\sqrt{A}),
\end{align}
and 
\begin{align}\label{eq:simple_inverse}
A(\id_{\R^n}+BA)^{-1} &= \sqrt{A}(\id_{\R_n}+\sqrt{A}B\sqrt{A})^{-1}\sqrt{A}.
\end{align}
\end{lemma}

\begin{proof}
Note that $\sqrt{A}B\sqrt{A}\in S^+(\R^n)$, so $\id_{\R_n}+\sqrt{A}B\sqrt{A}$ is clearly invertible. Next, let $(h_i)_{i=1}^{n}$ be an orthonormal basis of eigenvectors for $A$ such that the corresponding eigenvalues $(a_i)_{i=1}^{n}$ form a decreasing sequence; set $m=\max(\{i\colon a_i\neq 0\}\cup\{0\})$ and define $\tilde{A} = \sum_{i=1}^m a_i h_i \otimes h_i + \sum_{i=m+1}^{n} h_i \otimes h_i$. We have 
\begin{align*}
\id_{\R^n} +BA & 
= \tilde{A}^{-\frac{1}{2}}(\id_{\R^n} + \sqrt{A} B \sqrt{A} ) \tilde{A}^{\frac{1}{2}}
\end{align*}
from which we conclude that~\eqref{eq:simple_det} holds and that $\id_{\R^n}+BA$
is invertible with inverse
\begin{align*}
(\id_{\R_n} + BA)^{-1} & 
= \tilde{A}^{-\frac{1}{2}}(\id_{\R_n} + \sqrt{A} B \sqrt{A} )^{-1} \tilde{A}^{\frac{1}{2}}.
\end{align*}
Thus $x\in N(A)$ implies $(\id_{\R_n} + BA)^{-1}x\in N(A)$, and we conclude that~\eqref{eq:simple_inverse} holds.
\end{proof}


\begin{lemma}\label{lem:invconv_sot}
Let $H$ be a (real or complex) Hilbert space and $A, A_n\in L(H)$ $(n\in \N)$. Assume that $A_n \stackrel{\text{s.o.t.}}{\rightarrow} A$ (`s.o.t' stands for `strong operator topology'), that
$A$ and $A_n$ are invertible ($n\in \N$), and that $\sup_{n\in \N} \| A_n^{-1}\|_{L(H)} < \infty$.
Then $ A_n^{-1}\stackrel{\text{s.o.t.}}{\rightarrow} A^{-1}$.
\end{lemma}

\begin{proof} 
For all $h\in H$ we have 
\begin{align*}
    \| (A_n^{-1} - A^{-1})h \|_H 
    &
    = \| A_n^{-1}(\id_H - A_n A^{-1})h \|_H 
    \leq \sup_{k\in \N}\| A_k^{-1}\|_H 
    \| (A - A_n) A^{-1}) h \|_H \rightarrow 0.
\end{align*}
\end{proof}

\begin{lemma}\label{lem:sot_traceconvergence}
Let $H$ be a separable (real or complex) Hilbert space, let $A\in L_1(H)$, and let $B, B_1,B_2,\ldots,C, C_1,C_2,\ldots\in L(H)$ such that 
$\sup_{n\in \N} \| B_n \|_{L(H)}  < \infty$, $B_n \stackrel{\text{s.o.t.}}{\rightarrow} B$, and $C_n \stackrel{\text{s.o.t.}}{\rightarrow} C$. Then
\begin{enumerate}
\item $B_n C_n \stackrel{\text{s.o.t.}}{\rightarrow } BC$,
\item $\tr(B_n A) \rightarrow \tr(BA)$.
\end{enumerate}
\end{lemma}

\begin{proof}
The first assertion follows from the fact that $B_nC_n -BC = (B_n - B)C_n + B(C_n -C)$.
For the second, let $s_1\geq s_2 \geq \ldots \geq 0$ and let $(e_k)_{k\in\N}$, $(f_k)_{k\in \N}$ be orthonormal systems in $H$ such that $A=\sum_{k\in \N} s_k e_k \otimes f_k$ (see Theorem~\ref{thm:singular_value_decomposition}). 
Recall that $\tr(B_n A) = \sum_{k\in \N} \langle B_n A e_k, e_k \rangle_H = \sum_{k\in \N} s_k \langle B_n f_k, e_k \rangle_H$, and note that $\sum_{k\in \N} |s_k \langle B_n f_k, e_k \rangle_H| \leq \| A \|_{L_1(H)} \sup_{n\in \N} \|B_n\|_{L(H)} < \infty$, whence the result follows from the dominated convergence theorem and the fact that $\lim_{n\rightarrow \infty} \langle B_n f_k, e_k \rangle_H = \langle B f_k, e_k\rangle_H$ for all $k\in \N$.
\end{proof}

\begin{lemma}\label{lem:approxPsi}
Let $H$ be a separable real Hilbert space, $A\colon D(A) \subseteq H \rightarrow H$ the generator of a $C_0$-semigroup,  $(h_k)_{k\in \N}$ an orthonormal basis for $H$ satisfying $h_k\in D(A^*)$ for all $k\in \N$
$B\in L(H_{\C})$ such\footnote{The complexification $H_{\C}$ of $H$ is introduced at the beginning of Section~\ref{sec:Riccati}.} that $AB$, $BA^*$, $A B A^* \in L(H_{\C})$, let $\lambda \in \rho(A)$ (the resolvent set of $A$), and let $P_m\in L(H)$ be the orthonormal projection onto $\operatorname{span}(h_1,\ldots,h_m)$, $m\in \N$. For notational brevity we introduce\footnote{We have $D(P_m^{A})=H$ as $h_1,h_2,\ldots \in D(A^*)$.}
\begin{equation}
    P_m^{A}\colon H \rightarrow H,\quad P_m^{A}h=R(\lambda,A) P_m (\lambda - A)h \quad (h\in H),
\end{equation}
where $R(\lambda, A):= (\lambda  - A)^{-1}$ denotes the resolvent of $A$.
Then  
\begin{equation}
\begin{aligned}\label{eq:bound_project}
&
\sup_{m\in \N} \max\left\{ 
    \| 
        P_m^{A} B (P_m^A)^*
    \|_{L(H_\C)},
    \| 
       A P_m^{A} B (P_m^A)^*
    \|_{L(H_\C)},    
    \| 
       P_m^{A} B (P_m^A)^* A^*
    \|_{L(H_\C)}
    \right\}
\\ & \quad
\leq \left(1+ \max\{|\lambda|,1\} \| R(\lambda,A)\|_{L(H)} \right) \| R(\lambda,A) \|_{L(H)}   \| (\lambda - A)  B (\lambda - A^*) \|_{L(H_{\C})}  < \infty,
\end{aligned}
\end{equation}
and we have
$P_m^{A} B (P_m^A)^* \stackrel{\text{s.o.t.}}{\rightarrow} B $,
$A P_m^{A} B (P_m^A)^* \stackrel{\text{s.o.t.}}{\rightarrow}  A B$,
and
$P_m^{A} B (P_m^A)^* A^* \stackrel{\text{s.o.t.}}{\rightarrow}  B A^*$.
If moreover $B\in L_1(H_{\C})$ is such that $(\lambda - A) B (\lambda - A^*)\in L_1(H_\C)$ and $C\in L(H)$, then 
\begin{equation}\label{eq:L1bound_proj}
   \sup_{m\in \N} \| P_m^{A} B (P_m^{A})^* C\|_{L_1(H_{\C})} < \infty, 
\end{equation}
and $\tr(P_m^{A} B (P_m^{A})^* C) \rightarrow \tr(BC) $ as $m\rightarrow \infty$.
\end{lemma}

\begin{proof}
Estimate~\eqref{eq:bound_project} is immediate from the definition of $P_m^{A}$. It follows from the fact that $P_m \stackrel{\text{s.o.t.}}{\rightarrow} \id_H$ and that $(\lambda - A) B(\lambda - A^*), AR(\lambda,A), R(\lambda,A^*)A^*\in L(H_{\C})$ and repeated application of~Lemma~\ref{lem:sot_traceconvergence}
that $P_m^A B (P_m^{A})^* \stackrel{\text{s.o.t.}}{\rightarrow} B$, $A P_m^A B (P_m^{A})^*  \stackrel{\text{s.o.t.}}{\rightarrow} AB$, and $P_m^A B (P_m^{A})^* A^* \stackrel{\text{s.o.t.}}{\rightarrow} BA^*$.
Finally, if $B\in L_1(H_{\C})$ is such that $(\lambda - A )B(\lambda-A^*)\in L_1(H_{\C})$ then
\begin{align*}
 \| P_m^A B (P_m^A)^* C\|_{L_1(H_\C)} 
    & \leq \|R(\lambda,A)\|_{L(H_{\C})}^2 \| C \|_{L(H_\C)} \|(\lambda - A )B(\lambda-A^*)\|_{L_1(H_\C)} .
\end{align*} 
Moreover, 
\begin{align*}
   \tr(P_m^A B (P_m^A)^* C) = \tr( P_m R(\lambda,A^*)C R(\lambda,A) P_m (\lambda -A) B (\lambda-A^*)). 
\end{align*}
As $P_m R(\lambda,A^*)C R(\lambda,A) P_m \stackrel{\text{s.o.t.}}{\rightarrow} R(\lambda,A^*)C R(\lambda,A)$ and $(\lambda - A )B(\lambda-A^*)\in L_1(H_{C})$ it now follows from Lemma~\ref{lem:sot_traceconvergence} that $\tr(P_m^A B (P_m^A)^* C)\rightarrow \tr(B C)$. 
\end{proof}

\section{Laplace transform of the squared Ornstein--Uhlenbeck process}\label{app:Laplace-tarnsform}

Recall from the proof of Theorem~\ref{thm:existence} that 
the solution to~\eqref{eq:wishart1} is obtained by taking $X=Y^*Y$, where $Y$ is a Ornstein-Uhlenbeck process. In this setting the Laplace transform of $X$ can be obtained directly (i.e., one does not need the Riccati equations as in Proposition~\ref{thm:laplace_transform}), cf. 
Proposition~\ref{thm:laplace_transform_OU}. Of course, as we assume we already know that $X=Y^*Y$, we could not use this to deduce the uniqueness in law of solutions to~\eqref{eq:wishart1} as we did in Corollary~\ref{cor:wishart_uniqueinlaw}.

In order to prove Proposition~\ref{thm:laplace_transform_OU} we need the following well-known lemma, for lack of a suitable reference we provide a proof for the readers' convenience.

\begin{lemma}\label{lemma:expexpnorm}
Let $H$ be a Hilbert space, $\mu\in H$, and let $X$ be 
an $H$-valued centered Gaussian random variable with covariance operator 
\begin{equation*} 
 Q =\sum_{k=1}^{\infty} q_k h_k \otimes h_k,
\end{equation*} 
where $(q_k)_{k\in \N}$ is a non-negative sequence in $\ell_1$ and $(h_k)_{k\in\N}$ is an orthonormal basis of $H$. 
Then $\id_H + 2\alpha Q \in L(H)$ is invertible
for all $\alpha>-(2\max_{k\in \N} q_k)^{-1} = -\frac{1}{2}\| Q \|_{L(H)}^{-1}$,
\begin{align}\label{eq:def_det}
\det(\id_H + 2 \alpha Q) &:= \prod_{k=1}^{\infty}(1+2\alpha q_k) \in (0,\infty),
\\ \text{and} \quad \label{eq:expexpnorm}
 \E\left[ \e^{- \alpha\| X + \mu \|^2_H } \right]
 & =
  \det(\id_H + 2 \alpha Q)^{-\frac{1}{2}} \exp(-\langle (\id_H + 2\alpha Q )^{-1}\mu,\mu\rangle_H).
\end{align}
\end{lemma}

\begin{proof}
As $(q_k)_{k\in \N}$ is a non-negative sequence and $\min_{k\in \N} 2\alpha q_k>-1$, clearly $\id_H + 2\alpha Q$ is invertible, and as $\log(1+x) \leq x$ for all $x\in (-1,\infty)$, 
we have, for $\alpha \geq 0$ and $m,n\in \N$,
\begin{align*}
0\leq \log\left(\prod_{k=m}^{n} (1+2 \alpha q_k)\right) 
\leq 2\alpha \sum_{k=m}^{n} q_k.
\end{align*} 
Similarly, for $\alpha \leq 0$ and $m,n\in \N$ we have 
\begin{align*}
0\leq -\log\left(\prod_{k=m}^{n} (1+2 \alpha q_k)\right)
&= \sum_{k=m}^{n} \log\left( 1 + \tfrac{2|\alpha| q_k }{1+2\alpha q_k} \right) 
\leq 2|\alpha |\sum_{k=m}^{n} \frac{q_k}{1+2\alpha q_k} 
\\
& \leq 2 |\alpha| (1+2\alpha \max_{k\in \N}q_k)^{-1} \sum_{k=m}^{n} q_k.
\end{align*}
As $(q_k)_{k\in \N}\in \ell_1$, we see that $\lim_{n\rightarrow \infty} \prod_{k=1}^{n} (1 + 2\alpha q_k)$ exists (and is finite).\par 
Let $P_n = \sum_{k=1}^{n} h_k \otimes h_k$ ($n\in \N$), i.e., $P_n\in L(H)$ is the orthogonal projection onto $\operatorname{span}\{h_1,\ldots,h_n\}$. By the dominated convergence theorem (for $\alpha \geq 0$) or the monotone convergence theorem (for $\alpha<0$) it suffices to prove that 
\begin{equation}\label{eq:expectation_projected}
 \E\left[ \e^{- \alpha \| P_n (X + \mu) \|^2_H } \right]
 =
 \det(\id_H+2\alpha P_n Q)^{-\frac{1}{2}}
 \exp(-\langle (\id_H + 2\alpha Q )^{-1} P_n \mu,P_n\mu\rangle_H)
\end{equation}
for all $n\in \N$ (note that $\det(\id_H + 2 \alpha P_n Q)=\prod_{k=1}^{n} (1 + 2\alpha q_k)$). \par 
Now fix $n\in \N$ and let $(e_k)_{k=1}^{n}$ denote an orthonormal basis for $\R^n$. Define $Q_n \in \R^{n\times n}$ by $Q_n = \sum_{k=1}^{n} q_k e_k \otimes e_k$, and $\mu_n \in \R^n$ by $\langle \mu_n, e_k \rangle_{\R^n} = \langle \mu , h_k \rangle_H$ for all $k=1,\ldots,n$. Let $Z\sim \mathcal{N}(0,\id_{\R^n})$ be an $n$-dimensional vector of i.i.d.\ standard Gaussians, and observe that $\| \sqrt{Q}_n Z + \mu_n\|_{\R^n}$ and $\| P_n(X+\mu) \|_H$ are equal in distribution. \par 
To ease the notation, we introduce $\hat{Q}_n=\id_{\R^n}+2\alpha Q_n$. Note that $\hat{Q}_n$ is positive definite and commutes with $Q_n$. By first completing the square and then performing a coordinate transform we obtain:
\begin{align*}
& \E \left[ \exp\left( 
- \alpha \| Q_n^{\nicefrac{1}{2}} Z + \mu_n \|_{\R^n}^2 \right)\right] 
\\ & = \frac{1}{(2\pi)^{\nicefrac{n}{2}}} 
\exp(-\alpha \langle (\id_{\R^n} - 2 \alpha  \hat{Q}_n^{-1} Q_n) \mu_n, \mu_n\rangle_{\R^n}) 
\\
& \qquad \times 
\int_{\R^n}
\exp\left(
 -\tfrac{1}{2}
 \| \hat{Q}_n^{\nicefrac{1}{2}} z + 2 \alpha \hat{Q}_n^{-\nicefrac{1}{2}} Q_n^{\nicefrac{1}{2}}\mu_n \|_{\R^n}^2
\right)
\,dz
\\ & = 
\frac{1}{(2\pi)^{\nicefrac{n}{2}} \operatorname{det}(\hat{Q}_n^{\nicefrac{1}{2}})} 
\exp(-\langle ( \id_{\R^n} + 2\alpha  Q_n)^{-1} \mu_n, \mu_n\rangle_{\R^n})
\int_{\R^n}
\exp\left(
 -\tfrac{1}{2}
 \| y \|_{\R^n}^2
\right)
\,dy
\\ & = 
\operatorname{det}(\id_{\R^n} + 2 \alpha Q_n)^{-\nicefrac{1}{2}}
\exp(-\langle ( \id_{\R^n} + 2 \alpha Q_n)^{-1} \mu_n, \mu_n\rangle_{\R^n})
.
\end{align*}
Identity~\eqref{eq:expectation_projected} now follows 
by noting that $\det(\id_{\R^n} + 2 \alpha Q_n)=$ $\det(\id_H+2 \alpha  P_n Q)$ and 
\begin{equation*}
\langle ( \id_{\R^n} + 2\alpha  Q_n)^{-1} \mu_n, \mu_n\rangle_{\R^n} = \langle ( \id_{H} + 2 Q)^{-1} P_n \mu, P_n \mu \rangle_{H}. 
\end{equation*}
\end{proof}

\noindent{\it Proof of Proposition \ref{thm:laplace_transform_OU}.}
First of all note that $\sqrt{|u|}Q_t \sqrt{|u|}\in S^+_1(H)$, indeed, self-adjointness is obvious and moreover we have $\|\sqrt{|u|}Q_t \sqrt{|u|}\|_{L_1(H)} \leq \| \sqrt{|u|} \|_{L(H)}^2 \| Q_t \|_{L_1(H)}<\infty$ as $Q_t\in L_1(H)$ due to assumption~\eqref{eq:ass_integrability} in Theorem~\ref{thm:existence} and~\eqref{eq:SchattenHolder}. Furthermore, by the functional calculus for self-adjoint operators (see, e.g.,~\cite[Theorem VII.1.4]{Werner:2000}) we have $\| \sqrt{|u|} \|_{L(H)}^2 = \| u \|_{L(H)}$. It thus follows that $\sup(\sigma(\sqrt{|u|}Q_t \sqrt{|u|}))=\|\sqrt{|u|}Q_t \sqrt{|u|}\|_{L(H)} \leq \| u \|_{L(H)} \| Q_t\|_{L(H)}$, thus indeed $\id_H + 2\sign(u)\sqrt{|u|}Q_t \sqrt{|u|}$ is invertible (i) for all $u\in S^-(H)$ satisfying $\| u \|_{L(H)}\leq \frac{1}{2}\| Q_t \|_{L(H)}^{-1}$, and (ii) for all $u\in S^+(H)$.\par 
Let $ Y $ be an Ornstein-Uhlenbeck process as defined in   \eqref{eq:solY} (so $X=Y^*Y$), in particular, let $Y_0=y_0\in L_2(H,\R^n)$ be such that $y_0^*y_0=x_0$.
Note that if $u\in S^+(H)\cup S^-(H)$, then
\begin{equation*}
 \tr(uX_t)
 =
 \sign(u) \| Y_t \sqrt{|u|} \|_{L_2(H,\R^n)}^2
\end{equation*}
for all $t\geq 0$,
and note that $Y_t\sqrt{|u|}$ is a $L_2(H,\R^n)$-valued 
Gaussian with covariance operator $\cQ\in L(L_2(H,\R^n))$ given by $\cQ(B) = B \sqrt{|u|} Q_t \sqrt{|u|} $ and expectation $\mu= y_0 \e^{tA} \sqrt{|u|}$. Note that if $u\in S^-(H)$ satisfies $\| u \|_{L(H)}\leq \frac{1}{2} \| Q_t \|_{L(H)}^{-1}$ then $\| \cQ \|_{L(L_2(H,\R^n))} = \| \sqrt{|u|}Q_t \sqrt{|u|}\|_{L(H)}  \leq \| u \|_{L(H)}\|Q_t\|_{L(H)} <\frac{1}{2}$. Thus Lemmas~\ref{lemma:Lp(L2)} and~\ref{lemma:expexpnorm} imply
\begin{align*}
& \E \left[ 
 \exp( 
    -\tr(uX_t)
 )
 \right]
 = \E \left[ 
 \exp\left( 
    -\sign(u)\left\| Y_t \sqrt{|u|} \right\|_{L_2(H,\R^n)}^2
 \right)
 \right]
 \\ &\qquad =
 \det\left(\id_H + 2 \sign(u)\sqrt{|u|} Q_t \sqrt{|u|}\right)^{-\frac{n}{2}}
 \\ & \qquad \qquad \times
 \exp\left( 
   \left\langle y_0 \e^{tA} \sqrt{|u|} 
    (\id_H + 2 \sign(u)\sqrt{|u|} Q_t \sqrt{|u|})^{-1}, 
    y_0 \e^{tA} \sqrt{|u|} 
   \right\rangle_{L_2(H)}
 \right),
\end{align*}
whence~\eqref{eq:Laplacetrafoexplicit} follows.

\section{Proof of Proposition~\ref{prop:alpha_geq_n}
}\label{app:proof_Prop_alpha_geq_n}

We begin with some folklore\footnote{It is unclear to us whether this lemma also holds in infinite dimensions, but this suffices for our purposes.}:

\begin{lemma}\label{lem:measurable_EVD}
Let $n\in \N$. Then there exist Borel measurable mappings $\lambda_{1},\ldots,$ $\lambda_{n}\colon$ $ S(\R^n)\rightarrow \R$ and $h_{1},\ldots,h_{n}\colon S(\R^n)\rightarrow \R^n$ such that $(h_{k}(A))_{k=1}^n$ is an orthonormal basis for $\R^n$,
$|\lambda_{1}(A)|\geq |\lambda_{2}(A)| \geq \ldots \geq |\lambda_{n}(A)|$, and 
\begin{align*}
A = \sum_{k=1}^{n} \lambda_{k}(A) h_{k}(A)\otimes h_{k}(A)
\end{align*}
for all $A\in S(\R^n)$.
\end{lemma}

\begin{proof}
Apply e.g.~\cite[Theorem 1]{Azoff:1974}
with $X= S(\R^n)$, $Y=\R^n\times \R^{n\times n}$,
and 
\begin{equation*}
E = \left\{ 
 (A,(\lambda_k)_{k=1}^{n},(h_k)_{k=1}^n)\in X\times Y
 \colon 
 \begin{aligned}
 &\langle h_k, h_j \rangle_{\R^n} = 1_{\{k=j\}},\, 
 |\lambda_1|\geq |\lambda_2| \geq \ldots \geq |\lambda_n|,
 \\
 &A = \sum_{k=1}^n \lambda_k h_k \otimes h_k
 \end{aligned}
\right\}.
\end{equation*}
\end{proof}

We will also need the following:

\begin{lemma}\label{lem:approx_finrank}
Let $H$ be a separable Hilbert space, $(h_k)_{k\in \N}$ an orthonormal basis for $H$, $A\in S(H)$ of rank $m$ for some $m\in \N$, and let $P_n\in L(H)$ be given by $P_n=\sum_{k=1}^{n} h_k \otimes h_k$, $n\in \N$.
Then $\lim_{n\rightarrow \infty} \| A - P_n A P_n \|_{L_2(H)} = 0$. 
\end{lemma}

\begin{proof}
Write $A = \sum_{k=1}^{m} \lambda_k h_k \otimes h_k$ for some $\lambda_1,\ldots,\lambda_m\in \R$ and some orthonormal system $h_1,\ldots,h_m\in H$. Then 
$A- P_n A P_n = \sum_{k=1}^{m} \lambda_k h_k \otimes (h_k - P_n h_k) + \sum_{k=1}^{m} \lambda_k (h_k - P_n h_k)\otimes P_n h_k $, and the result follows 
from the fact that $\| f \otimes g \|_{L^2(H)} = \|f \|_H \| g \|_H$ for any $f,g\in H$.
\end{proof}

\begin{lemma}\label{lem:dense_in_kernel}
Let $H$ be a separable Hilbert space, let $V\subseteq H$ be a dense linear subspace, let $m\in \N$, and let $\eps>0$. Then there exists a Borel measurable mapping $\phi_{\eps}\colon H^{m+1}\rightarrow H$ such that $\phi_{\eps}(g_1,\ldots,g_m,x) \in V$, $\| \phi_{\eps}(g_1,\ldots,g_m,x) - x\|_{H}<\eps$, and $\phi_{\eps}(g_1,\ldots,g_m,x)$ is orthogonal to $\operatorname{span}\{g_1,\ldots,g_m\}$ whenever $x$ is orthogonal to $\operatorname{span}\{g_1,\ldots,g_m\}$.
\end{lemma}

\begin{proof}
Let $(h_n)_{n\in \N}$ be an orthonormal basis for $H$ such that $h_n\in V$ for all $n\in \N$ (obtained by applying Gram-Schmidt to a sequence in $V$ that is dense in $H$). Let $g_1,\ldots,g_m,x\in H$ be given and let $(w_k)_{k=1}^{d}$ be an orthonormal basis for $\operatorname{span}\{g_1,\ldots,g_m\}=:W$ (obtained by applying Gram-Schmidt to $h_1,\ldots,h_m$, here $d\leq m$). Let $P_n\colon H \rightarrow H$, $n\in \N$, be given by $P_n = \sum_{k=1}^{n} h_k \otimes h_k$ and let $N\in \N$ be the smallest integer for which $\{ P_N w_1,\ldots, P_N w_d\}$ are linearly independent. By Gauss elimination applied to the $d\times N$-dimensional matrix $\tilde{A}_{j,k} = \langle w_j, h_k \rangle_H$ ($1\leq j\leq d$, $1\leq k \leq N$) we can find the smallest $1\leq i_{1}< \ldots <i_{d} \leq N$ such that the matrix $A\in \R^{d\times d}$ satisfying $A_{j,k} = \langle w_j, h_{i_k} \rangle_H$ ($j,k\in \{1,\ldots,d\}$) is invertible.\par 
Let $N\in \N$ be the smallest value for which \begin{equation*}
\sum_{n=N+1}^{\infty} |\langle x, h_n\rangle_H|^2< \eps \min\left(\tfrac{1}{2}, \tfrac{1}{4 \sqrt{d}\| A^{-1} \|_{L(\R^d)}}\right);
\end{equation*}
we set $x^{(\eps)}=\sum_{n=1}^{N} \langle x, h_n\rangle_H h_n$. If $x \notin W^{\perp}$, 
we simply set $\phi_{\eps}(g_1,\ldots,g_m,x) = x^{(\eps)}$.\par 
If $x \in W^{\perp}$, we set  $x^{(\eps)}_{\text{err}} = (\langle x^{(\eps)}, w_j \rangle_H)_{j=1}^{d}\in \R^d$, $c= A^{-1}x^{(\eps)}_{\text{err}}$, and  $y^{(\eps)} = \sum_{k=1}^{d} c_k h_{i_k}$. As $x\in W^{\perp}$, we have 
$| \langle x^{(\eps)}, w_j \rangle_H|
= | \langle x^{(\eps)} - x , w_j \rangle_H | \leq \| x^{(\eps)} - x \|_{H} \leq \tfrac{\eps}{2 \sqrt{d}\| A^{-1}\|_{L(\R^d)}} $, whence $\| y^{(\eps)} \|_{H} = \| A^{-1} x^{(\eps)}_{\text{err}} \|_{\R^d} \leq \| A^{-1}\|_{L(\R^d)} \| x^{(\eps)}_{\text{err}} \|_{\R^d}\leq  \frac{\eps}{2}$. Finally, note that $x^{(\eps)} - y^{(\eps)} \in V$, $\| x - (x^{(\eps)} - y^{(\eps)}) \|_H < \frac{\eps}{2} + \frac{\eps}{2} = \eps$, and $\langle x^{(\eps)} - y^{(\eps)}, w_j \rangle_H = \langle x^{(\eps)}, w_j\rangle_H - (Ac)_{j} = \langle x^{(\eps)}, w_j\rangle_H - \langle x^{(\eps)}, w_j\rangle_H = 0$ for all $j\in \{1,\ldots,d\}$. Noting that the construction of $w_1,\ldots,w_d$, $i_1,\ldots,i_d$, $A$, $N$,  $x^{(\eps)}$, $W^{\perp}$ and $y^{(\eps)}$ is Borel measurable, we conclude that $\phi_{\eps}(g_1,\ldots,g_m,x):=x^{(\eps)}-y^{(\eps)}1_{W^{\perp}}(x)$, $g_1,\ldots,g_m,x\in H$, is indeed a measurable mapping satisfying the desired properties.
\end{proof}

\begin{lemma}\label{lem:GS_cont}
Let $m\in \N$, $H$ a Hilbert space, $(h_0,\ldots,h_m)$ an orthonormal 
system in $H$, let $(\tilde{g}_1,\ldots,\tilde{g}_m)$ be another orthonormal system in $H$, and let $(h_0,g_1,\ldots,g_m)$ be obtained by applying the Gram-Schmidt procedure to $(h_0,\tilde{g}_1,\ldots,\tilde{g}_m)$. Then there exists a constant $C_m$ (depending only on $m$) such that $\sup_{k\in \N} \| g_k - h_k \|_{H} \leq C_m \sup_{k\in \N} \| \tilde{g}_k - h_k \|_{H}.$
\end{lemma}

\begin{proof}
This lemma is easily verified by induction on $m$.
\end{proof}

These lemmas allow us to prove the following approximation lemma, which is crucial for the proof of Proposition~\ref{prop:alpha_geq_n}.

\begin{lemma}\label{lem:approx_ONS2}
Let $H$ be a separable Hilbert space, $\eps>0$, $V\subseteq H$ a dense subspace of $H$, $(\Omega,\cF,\P)$ a probability space, $m\in \N$, $Q\in S^+(H)$ with $\operatorname{rank}(Q)>m$, and let $X\colon \Omega \rightarrow S^+(H)$ be such that $\operatorname{rank}(X)=m$ a.s. Then there exist $\cF$-measurable $(\lambda_{k})_{k=1}^{m},\,(\lambda_{k}^{(n)})_{k=1}^{m}\colon \Omega \rightarrow [0,\infty)$ and $\cF$-measurable $(h_k)_{k=1}^{m+1},\,\colon \Omega \rightarrow H^{m+1}$, $(h_k^{(n)})_{k=1}^{m} \colon \Omega \rightarrow H^{m}$, $n\in \N$,
with the following properties:
\begin{enumerate}
\item \label{it:approx_ONS_evalueconverge} $\lim_{n\rightarrow \infty} \lambda_k^{(n)} = \lambda_k$ a.s.\ for all $k\in \{1,\ldots,m\}$,
\item \label{it:approx_ONS_eigendecomposition} $(h_k)_{k=1}^{m+1}$ is a.s.\ an orthonormal system in $H$
and $X = \sum_{k=1}^{m} \lambda_k h_k \otimes h_k$ 
a.s.,
\item \label{it:approx_ONS_Qhm+1neq0} $\P(Qh_{m+1}\neq 0)>1-\eps$ and $h_{m+1}\in V$ a.s.,
\item \label{it:approx_ONS_hnONS} $(h_1^{(n)},\ldots,h_m^{(n)},h_{m+1})$ is an orthonormal system in $H$ a.s. for all $n\in \N$,
\item \label{it:approx_ONS_hninV} $h_k^{(n)}\in V$ a.s.\ for all $n\in \N$, $k\in \{1,\ldots, m\}$,
\item \label{it:approx_ONS_evectorsconverge} $\lim_{n\rightarrow \infty} h_k^{(n)} = h_k$ a.s.\ for all $k\in \{1,\ldots,m\}$.
\end{enumerate}
\end{lemma}

\begin{proof}
Let $(g_n)_{n\in \N}$ be an orthonormal basis for $H$ such that $g_n\in V$ for all $n\in \N$ (see proof of Lemma~\ref{lem:dense_in_kernel}). Let $P_n \in L(\R^n, H)$ be defined by $P_n x=\sum_{k=1}^{n} x_k g_k$, $x\in \R^n$ (i.e., $P_n$ is an isometry onto its range). Note that $P_n^* X P_n \colon \Omega \rightarrow S^+(\R^n)$ is $\cF$-measurable and of rank at most $m$, whence by Lemma~\ref{lem:measurable_EVD} there exist $\cF$-measurable $\lambda^{(n)}_{1},\ldots,\lambda^{(n)}_{m}\colon \Omega\rightarrow \R$ and $e^{(n)}_{1},\ldots,e^{(n)}_{m}\colon \Omega \rightarrow \R^n$ such that $(e^{(n)}_{k})_{k=1}^{m}$ is an orthonormal system in $\R^n$, $\lambda_1^{(n)} \geq \ldots \geq \lambda_m^{(n)}\geq 0$, and 
\begin{equation}\label{eq:PnXPn_decomp}
P_n^* X P_n = \sum_{k=1}^{m} \lambda^{(n)}_{k} e^{(n)}_{k} \otimes e^{(n)}_{k},
\end{equation}
i.e.,
\begin{equation}\label{eq:PnPnXPnPn_decomp}
P_n P_n^* X P_n P_n^* 
=
\sum_{k=1}^{m} \lambda^{(n)}_{k} (P_n e^{(n)}_{k}) \otimes (P_n e^{(n)}_{k}) 
\in S^+(H)
.
\end{equation}
Note moreover that for all $n\in \N$ we have that $(P_n e^{(n)}_{k})_{k=1}^{m}$ is an orthonormal system in $H$ satisfying $P_n e^{(n)}_k \in V$ a.s.\ for all $k\in \{1,\ldots,m\}$. \par 
By Lemma~\ref{lem:approx_finrank} we have that 
\begin{equation}\label{eq:approx_X}
\lim_{n\rightarrow \infty} \| X - P_n P_n^* X P_n P^*_n \|_{L_2(H)} = 0 \quad \text{a.s.}, 
\end{equation}
whence by Weyl's inequality we have $\lim_{n\rightarrow \infty} (\lambda_{k}^{(n)})_{k=1}^{m}  = (\lambda_{k})_{k=1}^{m}$ a.s., where $(\lambda_{k})_{k=1}^{m}$ are the non-zero eigenvalues of $X$ in decreasing order (note that in particular, $(\lambda_{k})_{k=1}^{m}$ is $\cF$-measurable). We have thus proven~\ref{it:approx_ONS_evalueconverge}.\par 
Note that as an eigenvector basis is not uniquely determined, we cannot conclude that $(P_n e_{k}^{(n)})_{k=1}^{m}$ converges. However, by Banach-Alaoglu, there exists a subsequence $((P_{n_j} e_{k}^{(n_j)})_{k=1}^{m})_{j\in \N}$ that converges weakly in $L^2(\Omega,H^m)$ to some limit $(h_{k})_{k=1}^{m} \in L^2(\Omega,H^m)$. This,~\eqref{eq:PnPnXPnPn_decomp}, and~\eqref{eq:approx_X} imply that 
\begin{equation}\label{eq:Xdecomp}
X = \sum_{k=1}^{m} \lambda_k h_k \otimes h_k.
\end{equation}
To see that $(h_k)_{k=1}^{n}$ form an orthonormal system, first of all note that 
\begin{equation*}
\E(|\langle h_k, h\rangle_H| 1_{B}) = \lim_{j\rightarrow \infty} \E(|\langle P_{n_j} e_k^{(n_j)}, h\rangle_H| 1_{B}) \leq \| h\|_{H} \P(B)
\end{equation*}
for all $h\in H$ and all $B\in \cF$, whence $\| h_k \|_H \leq 1$ a.s. It follows from
\eqref{eq:Xdecomp} that 
$\tr( X ) = \sum_{k=1}^{m} \lambda_k \|h_k\|_H$.
On the other hand, as $(\lambda_k)_{k=1}^{m}$ are the eigenvalues of $X$, we have $\tr(X) = \sum_{k=1}^{m} \lambda_k$. As $\lambda_k>0$ a.s.\ for all $k\in \{1,\ldots,m\}$, it follows that $\| h_k \|_{H} = 1$ a.s., and thus $\| h_k \|_{L^2(\Omega,H)} =1 = \| P_{n_j} e_k^{(n_j)} \|_{L^2(\Omega,H)}$. This in combination with the weak convergence in $L^2(\Omega,H)$ implies that $P_{n_j} e_k^{(n_j)} $ converges (strongly) in $L^2(\Omega,H)$ to $h_k$ for all $k\in \{1,\ldots,m\}$, which in particular  implies that $(h_k)_{k=1}^{m}$ is an orthonormal system. By again passing to subsequence of $((P_{n_j} e_k^{(n_j)})_{k=1}^{m})_{j\in \N}$ we obtain a sequence $((\hat{h}^{(n)}_{k})_{k=1}^{m})_{n\in \N}$ of $\cF$-measurable sequence $V$-valued orthonormal systems in $H$,
and $\lim_{n\rightarrow \infty}\hat{h}^{(n)}_{k} = h_k$ a.s.\ (and in $L^2(\Omega,H)$).\par 
We now turn to constructing $h_{m+1}$: let $(g_k)_{k=1}^{m+1}$ be an orthonormal system of eigenvectors of $Q$ corresponding to non-zero eigenvalues (such an orthonormal system exists as $\operatorname{rank}(Q)>m$). By applying the Gram-Schmidt procedure to $(h_1,\ldots,h_m,g_1,\ldots,g_{m+1})$ we obtain $(h_1,\ldots, h_m, \tilde{g}_1,\ldots,\tilde{g}_{m+1})$. Note that by our choice of $(g_k)_{k=1}^{m+1}$, there must be at least one $k\in \{1,\ldots,m+1\}$ for which $Q \tilde{g}_k \neq 0$; we set $\hat{h}_{m+1} = \tilde{g}_{i}$, where $i=\inf\{ k \colon Q\tilde{g}_k \neq 0\}$. Note that $\hat{h}_{m+1}$ is indeed $\cF$-measurable.
Let $\eta>0$ be such that $\P(\| Q\hat{h}_{m+1} \|_H > \eta) > 1-\eps$ and set $h_{m+1} = \phi_{\eta/2}(h_1,\ldots,h_m,\hat{h}_{m+1})$, where $\phi_{\eta/2}$ is the Borel measurable mapping from Lemma~\ref{lem:dense_in_kernel}. In particular, we have now established~\ref{it:approx_ONS_eigendecomposition} and~\ref{it:approx_ONS_Qhm+1neq0}. Finally,
we obtain $h_1^{(n)},\ldots,h_m^{(n)}$ by applying the Gram-Schmidt procedure to $(h_{m+1},\hat{h}_1^{(n)},\ldots,\hat{h}_m^{(n)})$. Note that the resulting $H$-valued functions $h_1^{(n)},\ldots,h_m^{(n)}$ are $V$-valued and $\cF$-measurable, i.e.,~\ref{it:approx_ONS_hnONS} and~\ref{it:approx_ONS_hninV} hold. Finally, Lemma~\ref{lem:GS_cont} implies that~\ref{it:approx_ONS_evectorsconverge} holds.
\end{proof}

\begin{proof}[Proof of Proposition~\ref{prop:alpha_geq_n}]
Note that~\eqref{eq:solX_alpha} implies that
\begin{equation}\label{eq:solX_alpha_measurable}
\begin{aligned}
 \langle X_t g, h\rangle_{H} 
 & = 
 \langle x_0 g, h \rangle_{H}
 + 
 \int_0^{t} 
 (\alpha \langle Q g, h \rangle_H 
 + \langle X_s A g, h \rangle_{H} 
 + \langle X_s g, Ah \rangle_H ) \,ds
 \\ & \quad 
 + \int_{0}^{t} \langle \sqrt{X}_s \,dW_s \sqrt{Q}g, h\rangle_H + \int_0^t \langle \sqrt{Q} \,dW^*_s \sqrt{X}_s g, h \rangle_H
\end{aligned}
\end{equation}
for all $t\geq 0$ and all $\cF_0$-measurable $h,g$ taking values in $D(A)$ a.s.
Set 
\begin{equation}\label{eq:def_B}
B = \{ \operatorname{rank}(X_0) = m \} \in \cF_0.
\end{equation}
By applying Lemma~\ref{lem:approx_ONS2} with $V=D(A)$ and $(\Omega,\cF,\P)=(B,\cF\cap B,\frac{\P}{\P(B)})$
we obtain $\cF_0\cap B$-measurable $(\lambda_{k})_{k=1}^{m}\colon B \rightarrow [0,\infty)$ and $\cF_0\cap B$-measurable $(h_k)_{k=1}^{m+1}$, $(h_k^{(n)})_{k=1}^{m} \colon$ $ B \rightarrow H^{m+1}$, $n\in \N$, satisfying properties~\ref{it:approx_ONS_evalueconverge}--\ref{it:approx_ONS_evectorsconverge} of that lemma. 
For notational brevity we introduce $\lambda_{m+1}:=0$, $I_n:=\{1,\ldots,n\}$ ($n\in \N$), and $h_{m+1}^{(n)}:=h_{m+1}$ ($n\in \N$). 
We now define the $\cF_0\otimes \cB(S_1(H))$-measurable mapping $g_{n}\colon \Omega \times S_1(H)\rightarrow \R$ by
\begin{align}\label{eq:geps}
 g_{n}(\omega,x) 
 & =
 \begin{cases}
 \sum_{\sigma \in \Sigma_{m+1}}
   \operatorname{sgn}(\sigma)
   \prod_{k\in I_{m+1}}
   \langle x h_k^{(n)}(\omega), h_{\sigma(k)}^{(n)}(\omega) \rangle_H,
   & \omega \in B;
   \\
   0, 
   & \omega \in \Omega \setminus B,
   \end{cases}
\end{align}
where $\Sigma_{m+1}$ is the set of all permutations of $I_{m+1}$ and $\operatorname{sgn}(\sigma)$ is the sign of the permutation (i.e., $\operatorname{sgn}(\sigma)= 1$ if $\sigma$ is even and $\operatorname{sgn}(\sigma)= -1$ if $\sigma$ is odd). 
Note that by Lemma~\ref{lem:approx_ONS2}~\ref{it:approx_ONS_hnONS} we have
\begin{equation}\label{eq:gn_is_det}
g_{n}(\omega,x) = \operatorname{det}(P_{n}^*(\omega) x P_{n}(\omega))1_{B}(\omega),
\end{equation} 
where\footnote{Warning: this is not the same $P_n$ as the $P_n$ in the proof of Lemma~\ref{lem:approx_ONS2}.} $P_{n} \colon B \rightarrow L(\R^{m+1}, H)$, $P_{n}(\omega)y = \sum_{k\in I_{m+1}} y_k h_k^{(n)}(\omega)$ for $y\in \R^{m+1}$, $\omega\in B$.

Note moreover that $g_{n}$ is twice continuously (Fr\'echet) differentiable in the second variable with 
\begin{align}\label{eq:g_deriv1}
 g_{n}'(x)(a) & = 
 \begin{cases}
 \sum_{\sigma \in \Sigma_{m+1}}
   \operatorname{sgn}(\sigma)
 \sum_{i \in I_{m+1}}
   \langle a h_i^{(n)}, h_{\sigma(i)}^{(n)} \rangle_H
   \\ \quad 
   \prod_{k\in I_{m+1}\setminus\{i\}}
   \langle x h_k^{(n)}, h_{\sigma(k)}^{(n)} \rangle_H,
   & \omega \in B;
   \\
   0, 
   & \omega \in \Omega \setminus B,
   \end{cases}
\end{align}
and  
\begin{equation}\label{eq:g_deriv2}
\begin{aligned}
 g_{n}''(x)(a,b) 
 & =  
 \begin{cases}
 \sum_{\sigma \in \Sigma_{m+1}}
 \operatorname{sgn}(\sigma)
 \sum_{i\in I_{m+1}}
 \sum_{j \in I_{m+1}\setminus\{i\}}
   \langle a h_i^{(n)}, h_{\sigma(i)}^{(n)} \rangle_H
 \\  \quad \cdot 
   \langle b h_j^{(n)}, h_{\sigma(j)}^{(n)} \rangle_H
   \prod_{k\in I_{m+1}\setminus\{i,j\}}
   \langle x h_k^{(n)}, h_{\sigma(k)}^{(n)} \rangle_H,
   & \omega \in B;
   \\
   0, 
   & \omega \in \Omega \setminus B
   \end{cases}
\end{aligned}
\end{equation}
for all $x,a,b\in L(H)$.
\par  
Identity~\eqref{eq:solX_alpha_measurable} and It\^o's formula (here we use that $h_1^{(n)},\ldots,h_{m+1}^{(n)}\in D(A)$ a.s., see Lemma~\ref{lem:approx_ONS2}~\ref{it:approx_ONS_hninV}) imply that 
\begin{align}\label{eq:defM}
 M(t):= g_{n}(X_t) - g_{n}(x_0) - \int_{0}^{t} Lg_{n}(X_s) \,ds
\end{align}
is a local martingale, where 
\begin{equation}\label{eq:Lg}
\begin{aligned}
Lg_{n}(x) 
& =
 g_{n}'(x)(\alpha Q + xA + A^*x) 
 +
 \tfrac{1}{2} \sum_{\ell,p\in \N} g_{n}''(x)
 ( H_{\ell,p}(x), H_{\ell,p}(x))\,,
\end{aligned}
\end{equation}
and $H_{\ell,p}(x) =  (\sqrt{Q} g_{\ell})\otimes (\sqrt{x} g_p)  + (\sqrt{x} g_p) \otimes (\sqrt{Q} g_{\ell}) $ ($\ell,p\in\N$), with $(g_{\ell})_{\ell\in \N}$ a (deterministic) orthonormal basis for $H$.  
Set $$\mu =  \inf\{t\geq 0: Lg_n(X_t)\geq 0\}.$$ 
Because $\operatorname{rank}(X_0) = m$ on $B$ we have (see~\eqref{eq:gn_is_det}) that $g_{n}(X_0)=0$ a.s., and as $P_n^*(\omega)X_t P_n(\omega)$ ($\omega \in B$) is positive semi-definite it follows from~\ref{eq:gn_is_det} that $g_{n}(X_t)\geq 0$ a.s.\ for all $t\geq 0$. Thus if $\P(Lg_{n}(X_0)<0)>0$, then $\P(\mu>0)>0$ and $(M_{t\wedge \mu})_{t\geq 0}$ is a positive local martingale starting in $0$ that is strictly positive on $(0,\mu)$,
which leads to a contradiction. We conclude that $Lg_{n}(X_0)\geq 0$ a.s.\par 
Now let us calculate $Lg_{n}(X_0)$.  Inserting~\eqref{eq:g_deriv1} and~\eqref{eq:g_deriv2} into~\eqref{eq:Lg} we obtain
\begin{align}\label{eq:Lg(x0)1}
& Lg_{n}(X_0)|_{B} 
 \\ 
& =    
 \sum_{\sigma \in \Sigma_{m+1}}
   \operatorname{sgn}(\sigma)
 \sum_{i \in I_{m+1}}
   \langle (\alpha Q + X_0 A + A^* X_0) h_i^{(n)}, h_{\sigma(i)}^{(n)} \rangle_H
   \prod_{k\in I_{m+1}\setminus\{i\}}
   \langle X_0 h_k^{(n)}, h_{\sigma(k)}^{(n)} \rangle_H
\notag \\ 
& \quad
+ \tfrac{1}{2}
\sum_{\ell,p\in \N} 
\sum_{\sigma\in \Sigma_{m+1}} 
\operatorname{sgn}(\sigma)
\sum_{i\in I_{m+1}} \sum_{j\in I_{m+1}\setminus \{i\}} 
   \langle H_{\ell,p}(X_0) h_i^{(n)}, h_{\sigma(i)}^{(n)}\rangle_H 
   \langle H_{\ell,p}(X_0) h_j^{(n)}, h_{\sigma(j)}^{(n)} \rangle_H
\notag \\ 
& \quad \quad \cdot 
\prod_{k\in I_{m+1}\setminus\{i,j\}}
   \langle X_0 h_k^{(n)}, h_{\sigma(k)}^{(n)} \rangle_H.
\notag
\end{align}
Recall that $h_{m+1}^{(n)}=h_{m+1}$, and that $h_{m+1}(\omega) \in \operatorname{ker}(X_0)$ whenever $\omega\in B$ by Lemma~\ref{lem:approx_ONS2}~\ref{it:approx_ONS_eigendecomposition}, 
and thus $\langle X_0 h_k, h_{\sigma(k)} \rangle_H = 0$ whenever $m+1\in \{k,\sigma(k)\}$, and note that Lemma~\ref{lem:approx_ONS2}~\ref{it:approx_ONS_eigendecomposition} also implies that
\begin{align*}
 \langle H_{\ell,p}(X_0) h_{m+1}^{(n)}, h_{m+1}^{(n)} \rangle_H
 = 0, \quad \ell, p \in \N. 
\end{align*}
These two observations imply that~\eqref{eq:Lg(x0)1} reduces to
\begin{align}\label{eq:Lg(x0)2}
& Lg_{n}(X_0)|_{B}
\\ 
& =  
 \alpha \| \sqrt{Q}h_{m+1} \|_H^2 
   \sum_{\sigma \in \Sigma_{m}}
   \operatorname{sgn}(\sigma)
   \prod_{k\in I_{m}}
   \langle X_0 h_k^{(n)}, h_{\sigma(k)}^{(n)} \rangle_H
\notag \\ 
& \quad
+ \sum_{\ell,p\in \N} \sum_{i\in I_{m}} \sum_{\substack{\sigma\in \Sigma_{m+1},\\ \sigma(i)=m+1, \sigma(m+1)=i}} 
\operatorname{sgn}(\sigma)
   \big|
    \langle H_{\ell,p}(X_0) h_i^{(n)}, h_{\sigma(i)}\rangle_H
    \big|^2
\prod_{k\in I_{m}\setminus\{i\}}
   \langle X_0 h_k^{(n)}, h_{\sigma(k)}^{(n)} \rangle_H
\notag
\\ 
& =  
 \alpha \| \sqrt{Q}h_{m+1} \|_H^2 
   \sum_{\sigma \in \Sigma_{m}}
   \operatorname{sgn}(\sigma)
   \prod_{k\in I_{m}}
   \langle X_0 h_k^{(n)}, h_{\sigma(k)}^{(n)} \rangle_H
\notag \\ 
& \quad
- \sum_{\ell,p\in \N} \sum_{i\in I_{m}} \sum_{\sigma\in \Sigma_{m},\, \sigma(i)=i} 
\operatorname{sgn}(\sigma)
   \big|
    \langle H_{\ell,p}(X_0) h_i^{(n)}, h_{m+1}\rangle_H
    \big|^2
\prod_{k\in I_{m}\setminus\{i\}}
   \langle X_0 h_k^{(n)}, h_{\sigma(k)}^{(n)} \rangle_H.
\notag
\end{align}
Next, observe that
\begin{align*}
\langle H_{\ell,n}(X_0) h_i^{(n)}, h_{m+1}\rangle_H
& =
 \langle \sqrt{Q} g_{\ell}, h_{i}^{(n)} \rangle_H
 \langle \sqrt{X_0} g_p, h_{m+1} \rangle_H
\\ 
& \quad +
 \langle \sqrt{Q} g_{\ell}, h_{m+1} \rangle_H
 \langle \sqrt{X_0} g_p, h_{i}^{(n)} \rangle_H
\\ & = 
 \langle g_{\ell}, \sqrt{Q} h_{m+1} \rangle_H
 \langle g_p, \sqrt{X_0} h_{i}^{(n)} \rangle_H
\end{align*}
whence 
\begin{align*}
\sum_{\ell,p\in \N} 
 \big|  \langle H_{\ell,p}(X_0) h_i^{(n)}, h_{m+1}\rangle_H \big|^2
 = \| \sqrt{Q} h_{m+1} \|_H^2 \| X_0 h_i^{(n)} \|_H^2.
\end{align*}
In conclusion, we obtain 
\begin{align*}
Lg_{n}(X_0)|_{B}
& =  
 \alpha \| \sqrt{Q}h_{m+1} \|_H^2 
   \sum_{\sigma \in \Sigma_{m}}
   \operatorname{sgn}(\sigma)
   \prod_{k\in I_{m}}
   \langle X_0 h_k^{(n)}, h_{\sigma(k)}^{(n)} \rangle_H
\notag \\ 
& \quad
- \| \sqrt{Q} h_{m+1} \|_H^2 
\sum_{i\in I_{m}} \sum_{\sigma\in \Sigma_{m},\, \sigma(i)=i} 
\operatorname{sgn}(\sigma)
   \| X_0 h_i^{(n)} \|_H^2
\prod_{k\in I_{m}\setminus\{i\}}
   \langle X_0 h_k^{(n)}, h_{\sigma(k)}^{(n)} \rangle_H.
\notag
\end{align*}
It follows from Lemma~\ref{lem:approx_ONS2}~\ref{it:approx_ONS_eigendecomposition} and \ref{it:approx_ONS_evectorsconverge} that
\begin{align}\label{eq:Lgnlim}
& \lim_{n\rightarrow \infty} L g_n(X_0)|_{B}=  
 (\alpha - m) \| \sqrt{Q}h_{m+1} \|_H^2
   \prod_{k\in I_{m}} \lambda_k.
\end{align}
As $Lg_{n}(X_0)\geq 0$ a.s.\ for all $n\in \N$, $\prod_{k \in I_m} \lambda_k >0$ a.s., and
$\P(\| \sqrt{Q} h_{m+1} \|_H^2 >0 \cap B)>0$ (see Lemma~\ref{lem:approx_ONS2}~\ref{it:approx_ONS_Qhm+1neq0}), it follows 
that $\alpha\geq m$. \par 
Finally, assume that~\eqref{eq:Xstaysofrankm} holds, i.e., we have $\P(\tau>0)>0$, where 
\begin{equation*}
\tau(\omega) =
\begin{cases}
\inf\{t\geq 0\colon \operatorname{rank}(X_t(\omega))>m\},&
\omega \in B;
\\
0,& \omega \in \Omega \setminus B.
\end{cases}
\end{equation*}
Note that $\tau$ is a stopping time ($\operatorname{rank}(\cdot)$ is lower semi-continuous).
Thus $(M_{t\wedge \tau})_{t\geq 0}$ is a local martingale. Note that~\eqref{eq:gn_is_det} implies that $g_{n}(X_{t\wedge\tau})=0$ for all $t\geq 0$. Hence we can conclude from~\eqref{eq:defM} that $Lg_{n}(X_0)=0$, whence~\eqref{eq:Lgnlim} implies that $m=\alpha$. 
\end{proof}

\begin{remark}
The following conditions on the $H$-valued functions $(h_k^{(n)})_{k=1}^{m+1}$ defining $g_m$, see~\eqref{eq:geps}, are the reason the proof of Proposition~\ref{prop:alpha_geq_n} is so technical:
\begin{enumerate}
\item $(h_k^{(n)})_{k=1}^{m}$ must approximate a system of eigenvectors of $X_0$ as $n\rightarrow \infty$ in order to obtain~\eqref{eq:Lgnlim};
\item $(h_k^{(n)})_{k=1}^{m+1}$ must be $\cF_0$-measurable and $D(A)$-valued to ensure that we can apply~\eqref{eq:solX_alpha_measurable};
\item $(h_k^{(n)})_{k=1}^{m+1}$ must be an orthonormal system in $H$ to ensure that $g_n$ in~\eqref{eq:geps} is a determinant (which is needed to conclude that $g_n(X_t)\geq 0$);
\item we need $X_0 h_{m+1}^{(n)} = 0$ to ensure that the terms involving $A$ in~\eqref{eq:Lg(x0)1} vanish \emph{before} taking $n\rightarrow \infty$ (otherwise we have a problem when taking the limit as $A$ is unbounded);
\item we need that $Q h_{m+1} \neq 0$ on a set of positive measure to be able to draw the desired conclusions at the end of the proof.
\end{enumerate}
\end{remark}

\bibliographystyle{abbrv}
\bibliography{biblioversieCCK}

\begin{thebibliography}{10}

\bibitem{Azoff:1974}
E.~A. Azoff.
\newblock Borel measurability in linear algebra.
\newblock {\em Proceedings of the American Mathematical Society},
  42(2):346--350, 1974.

\bibitem{benth2014representation}
F.~E. Benth and P.~Kr{\"u}hner.
\newblock Representation of infinite-dimensional forward price models in
  commodity markets.
\newblock {\em Communications in Mathematics and Statistics}, 2:47--106, 2014.

\bibitem{benth2021barndorff}
F.~E. Benth and C.~Sgarra.
\newblock A {B}arndorff-{N}ielsen and {S}hephard model with leverage in
  {H}ilbert space for commodity forward markets.
\newblock {\em Available at SSRN 3835053}, 2021.

\bibitem{Benth:2018}
F.~E. Benth and I.~C. Simonsen.
\newblock The {H}eston stochastic volatility model in {H}ilbert space.
\newblock {\em Stoch. Anal. Appl.}, 36(4):733--750, 2018.

\bibitem{bertucci2022spectral}
C.~Bertucci, M.~Debbah, J.-M. Lasry, and P.-L. Lions.
\newblock A spectral dominance approach to large random matrices.
\newblock {\em Journal de Math{\'e}matiques Pures et Appliqu{\'e}es},
  164:27--56, 2022.

\bibitem{bru1989diffusions}
M.-F. Bru.
\newblock Diffusions of perturbed principal component analysis.
\newblock {\em J. Multivariate Anal.}, 29(1):127--136, 1989.

\bibitem{Bru:1991}
M.-F. Bru.
\newblock Wishart processes.
\newblock {\em Journal of Theoretical Probability}, 4(4):725--751, 1991.

\bibitem{buraschi2010correlation}
A.~Buraschi, P.~Porchia, and F.~Trojani.
\newblock Correlation risk and optimal portfolio choice.
\newblock {\em The Journal of Finance}, 65(1):393--420, 2010.

\bibitem{carmona2007interest}
R.~Carmona and M.~R. Tehranchi.
\newblock {\em Interest rate models: an infinite dimensional stochastic
  analysis perspective}.
\newblock Springer Science \& Business Media, 2007.

\bibitem{carr1999option}
P.~Carr and D.~Madan.
\newblock Option valuation using the fast {F}ourier transform.
\newblock {\em Journal of computational finance}, 2(4):61--73, 1999.

\bibitem{Conway:1990}
J.~B. Conway.
\newblock {\em A course in functional analysis}, volume~96 of {\em Graduate
  Texts in Mathematics}.
\newblock Springer-Verlag, New York, second edition, 1990.

\bibitem{Conway:2000}
J.~B. Conway.
\newblock {\em A course in operator theory}, volume~21 of {\em Graduate Studies
  in Mathematics}.
\newblock American Mathematical Society, Providence, RI, 2000.

\bibitem{cox2022affine}
S.~Cox, S.~Karbach, and A.~Khedher.
\newblock Affine pure-jump processes on positive {H}ilbert--{S}chmidt
  operators.
\newblock {\em Stochastic Processes and their Applications}, 151:191--229,
  2022.

\bibitem{cox2022infinite}
S.~Cox, S.~Karbach, and A.~Khedher.
\newblock An infinite-dimensional affine stochastic volatility model.
\newblock {\em Mathematical Finance}, 32(3):878--906, 2022.

\bibitem{CFMT:11}
C.~Cuchiero, D.~Filipovi{\'c}, E.~Mayerhofer, and J.~Teichmann.
\newblock Affine processes on positive semidefinite matrices.
\newblock {\em The Annals of Applied Probability}, 21(2):397--463, 2011.

\bibitem{da2011hedging}
J.~Da~Fonseca, M.~Grasselli, and F.~Ielpo.
\newblock Hedging (co) variance risk with variance swaps.
\newblock {\em International Journal of Theoretical and Applied Finance},
  14(06):899--943, 2011.

\bibitem{da2008multifactor}
J.~Da~Fonseca, M.~Grasselli, and C.~Tebaldi.
\newblock A multifactor volatility {H}eston model.
\newblock {\em Quantitative Finance}, 8(6):591--604, 2008.

\bibitem{DaPratoZabczyk:1992}
G.~Da~Prato and J.~Zabczyk.
\newblock {\em Stochastic equations in infinite dimensions}, volume~44 of {\em
  Encyclopedia of Mathematics and its Applications}.
\newblock Cambridge University Press, Cambridge, 1992.

\bibitem{eberlein2010analysis}
E.~Eberlein, K.~Glau, and A.~Papapantoleon.
\newblock Analysis of {F}ourier transform valuation formulas and applications.
\newblock {\em Applied Mathematical Finance}, 17(3):211--240, 2010.

\bibitem{EngelNagel:2000}
K.-J. Engel and R.~Nagel.
\newblock {\em One-parameter semigroups for linear evolution equations}, volume
  194 of {\em Graduate Texts in Mathematics}.
\newblock Springer-Verlag, New York, 2000.
\newblock With contributions by S. Brendle, M. Campiti, T. Hahn, G. Metafune,
  G. Nickel, D. Pallara, C. Perazzoli, A. Rhandi, S. Romanelli and R.
  Schnaubelt.

\bibitem{filipovic2001consistency}
D.~Filipovic.
\newblock {\em Consistency problems for {H}eath-{J}arrow-{M}orton interest rate
  models}.
\newblock Springer Science \& Business Media, 2001.

\bibitem{FonescaEtAl:2007}
J.~D. Fonseca, M.~Grasselli, and C.~Tebaldi.
\newblock Option pricing when correlations are stochastic: an analytical
  framework.
\newblock {\em Review of Derivatives Research}, 10:151--180, 2007.

\bibitem{gourieroux2010derivative}
C.~Gourieroux and R.~Sufana.
\newblock Derivative pricing with {W}ishart multivariate stochastic volatility.
\newblock {\em Journal of Business \& Economic Statistics}, 28(3):438--451,
  2010.

\bibitem{graczyk2013multidimensional}
P.~Graczyk and J.~Ma{\l}ecki.
\newblock Multidimensional {Y}amada-{W}atanabe theorem and its applications to
  particle systems.
\newblock {\em Journal of Mathematical Physics}, 54(2):021503, 2013.

\bibitem{graczyk2018characterization}
P.~Graczyk, J.~Ma{\l}ecki, and E.~Mayerhofer.
\newblock A characterization of {W}ishart processes and {W}ishart
  distributions.
\newblock {\em Stochastic Processes and their Applications}, 128(4):1386--1404,
  2018.

\bibitem{GM:11}
P.~Graczyk and E.~Mayerhofer.
\newblock Stochastic analysis methods in {W}ishart theory.
\newblock In {\em CIMPA Workshop}, 2011.

\bibitem{grasselli2008solvable}
M.~Grasselli and C.~Tebaldi.
\newblock Solvable affine term structure models.
\newblock {\em Mathematical Finance: An International Journal of Mathematics,
  Statistics and Financial Economics}, 18(1):135--153, 2008.

\bibitem{HarveyEtAl:1994}
A.~Harvey, E.~Ruiz, and N.~Shephard.
\newblock Multivariate stochastic variance models.
\newblock {\em The Review of Economic Studies}, 61(2):247--264, 1994.

\bibitem{katori2004symmetry}
M.~Katori and H.~Tanemura.
\newblock Symmetry of matrix-valued stochastic processes and noncolliding
  diffusion particle systems.
\newblock {\em Journal of mathematical physics}, 45(8):3058--3085, 2004.

\bibitem{leippold2008asset}
M.~Leippold and F.~Trojani.
\newblock Asset pricing with matrix affine jump diffusions.
\newblock In {\em Stern NYU Seminars, Stern NYU link}. Citeseer, 2008.

\bibitem{letac2008noncentral}
G.~Letac and H.~Massam.
\newblock The noncentral {W}ishart as an exponential family, and its moments.
\newblock {\em Journal of Multivariate Analysis}, 99(7):1393--1417, 2008.

\bibitem{letac2018laplace}
G.~Letac and H.~Massam.
\newblock The {L}aplace transform {$(\det s)^{-p}\exp{\rm tr}(s^{-1}w)$} and
  the existence of non-central {W}ishart distributions.
\newblock {\em J. Multivariate Anal.}, 163:96--110, 2018.

\bibitem{M:19}
E.~Mayerhofer.
\newblock On {W}ishart and noncentral {W}ishart distributions on symmetric
  cones.
\newblock {\em Transactions of the American Mathematical Society},
  371(10):7093--7109, 2019.

\bibitem{Mayerhofer:2019}
E.~Mayerhofer.
\newblock Reforming the {W}ishart characteristic function.
\newblock {\em arXiv preprint arXiv:1901.09347}, 2019.

\bibitem{Neerven:2022}
J.~v. Neerven.
\newblock {\em Functional Analysis}, volume 201.
\newblock Cambridge University Press, 2022.

\bibitem{NeervenVeraarWeis:2015}
J.~v. Neerven, M.~Veraar, and L.~Weis.
\newblock Stochastic integration in {B}anach spaces---a survey.
\newblock In {\em Stochastic analysis: a series of lectures}, volume~68 of {\em
  Progr. Probab.}, pages 297--332. Birkh\"{a}user/Springer, Basel, 2015.

\bibitem{Pazy:1983}
A.~Pazy.
\newblock {\em \rm ``{S}emigroups of {L}inear {O}perators and {A}pplications to
  {P}artial {D}ifferential {E}quations''}, volume~44 of {\em Applied
  Mathematical Sciences}.
\newblock Springer-Verlag, New York, 1983.

\bibitem{PhilipovGlickman:2006}
A.~Philipov and M.~E. Glickman.
\newblock Multivariate stochastic volatility via {W}ishart processes.
\newblock {\em Journal of Business \& Economic Statistics}, 24(3):313--328,
  2006.

\bibitem{RevuzYor:1999}
D.~Revuz and M.~Yor.
\newblock {\em Continuous martingales and {B}rownian motion}, volume 293 of
  {\em Grundlehren der mathematischen Wissenschaften [Fundamental Principles of
  Mathematical Sciences]}.
\newblock Springer-Verlag, Berlin, third edition, 1999.

\bibitem{Werner:2000}
D.~Werner.
\newblock {\em Funktionalanalysis}.
\newblock Springer-Verlag, Berlin, extended edition, 2000.

\end{thebibliography}
\end{document}